\documentclass[11pt,reqno]{amsart}
\usepackage{a4wide}
\usepackage{amssymb}
\usepackage{amsthm}
\usepackage{amsmath,todonotes}
\usepackage{amscd}
\usepackage{cancel}
\usepackage{verbatim}
\usepackage{color}
\usepackage[all]{xy}
\usepackage[mathscr]{eucal}
\usepackage{mathrsfs}
\usepackage{fullpage}
\usepackage{enumitem}
\usepackage{hyperref}
\usepackage{bbm}
\usepackage{qtree}
\usepackage{bm}
\allowdisplaybreaks

\numberwithin{equation}{section}

\newtheorem{theorem}{Theorem}

\newtheorem{lemma}[theorem]{Lemma}

\newtheorem{corollary}[theorem]{Corollary}

\newtheorem{proposition}[theorem]{Proposition}

\theoremstyle{remark}

\newtheorem*{remark}{Remark}

\theoremstyle{definition}

\numberwithin{theorem}{section} 
\numberwithin{equation}{section}
\numberwithin{table}{section}

\newcommand{\ord}{\text {\rm ord}}

\newcommand{\GG}{\mathcal{G}}

\newcommand{\R}{\mathbb{R}}
\newcommand{\C}{\mathbb{C}}

\newcommand{\Z}{\mathbb{Z}}
\newcommand{\N}{\mathbb{N}}
\newcommand{\SL}{{\text {\rm SL}}}

\newcommand{\lcm}{{\text {\rm lcm}}}

\newcommand{\im}{\textnormal{Im}}

\def\H{\mathbb{H}}

\newcommand{\vol}{\operatorname{vol}}
\renewcommand{\S}{\mathbb{S}}

\begin{document}
\title[Finiteness theorems for universal sums of squares of almost primes]{Finiteness theorems for universal sums of squares of almost primes}
\author{Soumyarup Banerjee}
\address{\rm Discipline of Mathematics, Indian Institute of Technology Gandhinagar, Palaj, Gujarat - 382355, India}
\email{soumyarup.b@iitgn.ac.in}
\author{Ben Kane}
\address{\rm Mathematics Department, University of Hong Kong, Pokfulam, Hong Kong}
\email{bkane@hku.hk}
\date{\today}


\medskip
\begin{abstract}
In this paper we study quadratic forms which are universal when restricted to almost prime inputs, establishing finiteness theorems akin to the Conway--Schneeberger 15 theorem.
\end{abstract}

\thanks{Some of the research was conducted while the first author was a postdoctoral fellow at the University of Hong Kong. The research of the second author was supported by grants from the Research Grants Council of the Hong Kong SAR, China (project numbers HKU 17302515, 17316416, 17301317, and 17303618).}
\maketitle

\section{Introduction And Statement Of Results}\label{sec:intro}
The Conway--Schneeberger Fifteen theorem (first proven by Conway--Schneeberger--Miller in unpublished work and then elegantly reproven and generalized by the escalator tree method of Bhargava \cite{Bhargava}) states that a given positive definite integral quadratic form is \begin{it}universal\end{it} (i.e., represents every positive integer with integer inputs) if and only if it represents the integers up to $15$ (a smaller subset of these numbers actually suffices).  In particular, the sums of squares
\begin{equation}\label{eqn:diagonalsum}
Q(x)=\sum_{j=1}^{\ell} a_j x_j^2
\end{equation}
are universal if and only if they represent every integer up to $15$.  This classification of universal quadratic forms can be considered a generalization of Lagrange's Theorem, which states that every integer may be written as the sum of 4 squares of integers (i.e., the choice $a_j=1$ with $\ell=4$ is universal). Following work of Br\"udern and Fouvry \cite{BrudernFouvry}, many authors have considered generalizations of Lagrange's Theorem where the choice of $x$ is sieved into natural subsets. Specifically, they consider $x\in P_r^4$, where $r$ is a fixed integer and $P_r$ denotes the set of \begin{it}almost primes\end{it}, i.e., those integers whose prime factorizations $\prod_p p^{a_p}$ satisfy $\sum_p a_p \leq r$. In this paper, we allow $0\in P_r$ for all $r$ throughout. Br\"udern and Fouvry show that if $r\geq 34$, then every sufficiently large integer $n\equiv 4\pmod{24}$ may be written as the sum of squares of four $P_r$ numbers. The bound on $r$ has since been reduced by a number of authors, including for example \cite{ChingTW,HB-Tolev,TsangZhao}. 

The goal of this paper is to combine these generalizations by asking for the infimum $N_r$ such that a diagonal form $Q$ of the type \eqref{eqn:diagonalsum} is universal with almost prime $x\in P_r^{\ell}$ if and only if it represents every $n\leq N_r$ with almost prime $x\in P_r^{\ell}$ (one can consider analogous questions where the form is not diagonal, but there are some issues which arise because the set $P_r$ is not closed under addition, and hence two equivalent quadratic forms will represent different integers when restricted to $P_r$ numbers, so we restrict ourselves here to the diagonal case); for ease of notation, we say that $Q$ is \begin{it}$P_r$-universal\end{it} if it represents every positive integer with $x\in P_r^{\ell}$. We are interested in showing that $N_r$ is finite, which would be a \begin{it}finiteness theorem\end{it} for $P_r$-universality. Moreover, since $P_r\subset P_{r+1}$, one might expect that if $N_r$ exists, then $N_{r}\geq N_{r+1}$ and the Conway--Schneeberger 15 theorem together with $\Z=P_{\infty}:=\bigcup_{r=0}^{\infty} P_r$ hints that one might expect a uniform bound $N_r=15$ for $r$ sufficiently large (or in other words $\lim_{r\to\infty} N_r =N_{\infty}=15$). However, a simple argument shows that quite the opposite is true.
\begin{theorem}\label{thm:Nrinfinite}
We have $N_{r}>2^{2r+2}$, and in particular $\lim_{r\to\infty} N_{r}=\infty$. In other words, there is no uniform bound for $P_r$-universality. 
\end{theorem}
The problem that leads to Theorem \ref{thm:Nrinfinite} arises due to the same issue which led Br\"udern and Fouvry to assume that $n\equiv 4\pmod{24}$. Namely, any representation of $2^m$ as the sum of four squares must necessarily have a large power of $2$ dividing each $x_j$. In trying to obtain $P_r$-universality of other sums of the type \eqref{eqn:diagonalsum}, one also finds a similar obstruction when high powers of $3$ divide $n$. Letting $S$ be a finite set of primes, we say that an integer $x$ is \begin{it}nearly a $P_r$-number with respect to $S$\end{it} (or a \begin{it}nearly almost prime\end{it} when the sets $P_r$ and $S$ are clear from context) if $x=\prod_{p\in P} p^{a_p} \prod_{p\notin P} p^{b_p}$ and $\prod_{p\notin P}p^{b_p}$ is a $P_r$-number. Let $P_{r,S}$ be the set of nearly $P_r$-numbers with respect to $S$. We next consider the question of $P_{r,S}$-representations. That is to say, we consider whether a quadratic form is universal with $x\in P_{r,S}^{\ell}$. Let $N_{r,S}$ be the infimum of integers such that if $Q$ represents every $n\leq N_{r,S}$ with $x\in P_{r,S}$, then $Q$ is $P_{r,S}$-universal. We next see that the expected uniform bound of $15$ is obtained when we exclude the primes $2$ and $3$ that cause the obstructions noted above.
\begin{theorem}\label{thm:PrSuniversal}
Let $S:=\{2,3,5\}$. Then for $r\geq 694$, we have $N_{r,S}=15$. In other words, for $r\geq 694$, $Q$ is $P_{r,S}$-universal if and only if it represents every $m\leq 15$ with $P_{r,S}$-numbers. Furthermore, for $r\geq 694$, every quadratic form is $P_{r,S}$-universal if and only if it is universal.  
\end{theorem}
\begin{remark}
The set $S$ is optimal in the sense that for each $p\in S$ there exists a sum of squares $Q$ and a family of natural numbers $n_{\ell}$ for which every solution $Q(\bm{x})=n_{\ell}$ has $p^{\ell}\mid x_j$ for all $j$. 
\end{remark}
As an immediate corollary, one obtains the following explicit version of Lagrange's sum of four squares theorem.
\begin{corollary}\label{cor:Lagrange}
For $S=\{2,3,5\}$, every integer is a sum of four squares of $P_{694,S}$ numbers.
\end{corollary}
The paper is organized as follows. In Section \ref{sec:prelim}, we give a number of bounds on a number of elementary functions and introduce the preliminaries needed for the rest of the paper. In Section \ref{sec:initial}, we prove Theorem \ref{thm:Nrinfinite}. In Section \ref{sec:ThetaCoefficients}, we prove bounds on the coefficients of theta functions. Finally, in Section \ref{sec:sieving}, we apply sieving theory in order to prove Theorem \ref{thm:PrSuniversal}.

\section*{Acknowledgements}
The authors thank Kai-Man Tsang for many helpful discussions and Jeremy Rouse for pointing out an error in an earlier version of the paper.

\section{Preliminaries}\label{sec:prelim}

\subsection{Inequalities for some arithmetic functions}
Throughout the paper, we require some well-known inequalities for certain arithmetic functions. We begin with the following bound of Rosser and Schoenfeld \cite[Corollary of Theorem 8, i.e., (3.30), and Theorem 7, i.e., (3.26)]{RosserSchoenfeld}. 
\begin{lemma}\label{lem:RosserSchoenfeld}
Letting $\gamma$ denote the Euler--Mascheroni constant, we have 
\[
\prod_{p\leq z} \left(1-\frac{1}{p}\right)> \frac{e^{-\gamma}}{\log(z)}\left(1+\frac{1}{\log^2(z)}\right)^{-1}
\]

We also have 
\[
\prod_{p\leq z} \left(1-\frac{1}{p}\right)< \frac{e^{-\gamma}}{\log(z)}\left(1+\frac{1}{2\log^2(z)}\right).
\]
\end{lemma}

We require some additional well-known bounds for some other simple functions. 
\begin{lemma}\label{lem:cdelbnd}
\noindent

\noindent
\begin{enumerate}[leftmargin=*,label={\rm(\arabic*)}]
\item
Let $\delta>0$ be given and suppose that $p_0$ is a prime for which $p_0> \left(1+p_0^{-1}\right)^{\frac{1}{\delta}}$. Then for $m\in\N$ we have 
\begin{equation}\label{eqn:cdelbnd}
\prod_{p\mid m}\left(1+p^{-1}\right)\leq c_{\delta}m^{\delta}
\end{equation}
with $c_{\delta}:=\prod_{p<p_0}\frac{1+p^{-1}}{p^{\delta}}$. In particular, for $\delta=10^{-6}$ we have 
\begin{equation}\label{eqn:delpart}
\prod_{p\mid m}\left(1+p^{-1}\right)\leq 11.3\times m^{10^{-6}}.
\end{equation}
\item 
We have 
\[
\prod_{p\mid m} \left(1-p^{-1}\right)\geq  \frac{m^{-\delta}}{\zeta(2)c_{\delta}}.
\]
In particular,
\[
\prod_{p\mid m} \left(1-p^{-1}\right)\geq \frac{1}{20} m^{-10^{-6}} .
\]
\end{enumerate}
\end{lemma}
\begin{proof}
(1) Note that $m^{-\delta}\prod_{p\mid m}\left(1+p^{-1}\right)$ is multiplicative, and hence for a fixed $p_0$ we have
\[
m^{-\delta}\prod_{p\mid m}\left(1+p^{-1}\right)\leq \prod_{\substack{p^j\| m\\ p<p_0}} \frac{1+p^{-1}}{p^{j\delta}} \prod_{\substack{p^j\| m\\ p\geq p_0}} \frac{1+p_0^{-1}}{p^{j\delta}}\leq c_{\delta} \prod_{\substack{p^j\| m\\ p\geq p_0,\ p^{j\delta}\leq 1+p_0^{-1}}} \frac{1+p_0^{-1}}{p^{j\delta}}.
\]
Since $p\geq p_0$ and $p_0>1+p_0^{-1}$, the remaining product is empty, yielding the first claim. 

For \eqref{eqn:delpart}, we take $\delta=10^{-6}$ and $p_0=87853$. One easily verifies that $(1+1/p_0)^{10^6}<87800<p_0$ and a quick computer check verifies that $c_{10^{-6}}<11.3$. 

\noindent (2) Using part (1), for any $\delta>0$ we may bound 
\[
\prod_{p\mid m} \left(1-\frac{1}{p}\right) =\prod_{p\mid m} \frac{1-\frac{1}{p^2}}{1+\frac{1}{p}}\geq \frac{\prod_{p}\left(1-\frac{1}{p^2}\right)}{\prod_{p\mid m} \left(1+\frac{1}{p}\right)} =\frac{1}{\zeta(2)\prod_{p\mid m} \left(1+\frac{1}{p}\right)}\geq \frac{m^{-\delta}}{\zeta(2)c_{\delta}}.
\]
For the second claim, we take $\delta=10^{-6}$ and use \eqref{eqn:delpart}.
\end{proof}

We use a trick of Ramanujan to obtain an explicit bound on $\sigma_0(n)$, where $\sigma_{k}(n):=\sum_{d\mid n} d^k$ is the sum of powers of divisors function. 

\begin{lemma}\label{lem:sigma0bnd}
For $5\leq n\in\N$ and $\alpha>0$, we have 
\[
\sigma_0(n)\leq \mathcal{C}_{\alpha} n^{\alpha}
\]
where $\mathcal{C}_{\alpha}:=\prod_{p<2^{\frac{1}{\alpha}}} \max\left\{\frac{j+1}{p^{j\alpha}}:j\geq 1\right\}$. Specifically, we have 
\[
\mathcal{C}_{\frac{1}{10}}<4.175\times 10^{10},\qquad \qquad \mathcal{C}_{\frac{1}{14}}<2.634\times 10^{71},\qquad \qquad \mathcal{C}_{\frac{1}{15}}<2.751\times 10^{120}.
\]
\end{lemma}
\begin{proof}
Writing
\[
\frac{\sigma_0(n)}{n^{\alpha}} = \prod_{p^j\| n} \frac{j+1}{p^{j\alpha}},
\]
we see that if $p^{\alpha}\geq 2$, then $\frac{j+1}{p^{j\alpha}}<\frac{j+1}{2^j}\leq 1$. Thus we may take the product over those $p<2^{\frac{1}{\alpha}}$ and maximize over $j$. We compute $\mathcal{C}_{\alpha}$ explicitly for $\alpha\in\left\{\frac{1}{15},\frac{1}{14},\frac{1}{10}\right\}$.
\end{proof}

\begin{lemma}\label{lem:sigma-1bnd}
For $5\leq n\in\N$ and $\alpha>0$, we have 
\[
\sigma_{-1}(n)\leq \mathcal{G}_{\alpha} n^{\alpha}
\]
where $\mathcal{G}_{\alpha}:=\prod_{p^{j\alpha}(p-1)<p} \max\left\{\frac{1-p^{-j-1}}{\left(1-p^{-1}\right)p^{j\alpha}}\right\}$. Specifically, we have 
\[
\mathcal{G}_{\frac{1}{50}}<3.466,\qquad\qquad \mathcal{G}_{\frac{1}{100}}<4.369.
\]
\end{lemma}
\begin{proof}
Writing
\[
\frac{\sigma_{-1}(n)}{n^{\alpha}} = \prod_{p^j\| n} \frac{\sum_{\ell=0}^{j} p^{-\ell}}{p^{j\alpha}}= \prod_{p^j\| n} \frac{1-p^{-j-1}}{\left(1-p^{-1}\right)p^{j\alpha}},
\]
we see that if $p^{j\alpha}\geq \frac{p}{p-1}$ (it always suffices for $p^{\alpha}>2$), then 
\[
\frac{1-p^{-j-1}}{\left(1-p^{-1}\right)p^{j\alpha}}\leq \frac{p}{(p-1)p^{j\alpha}}\leq 1.
\]
 Thus we may take the product over those $p$ with $p^{\alpha}(p-1)<p$ and maximize over $j$. We compute $\mathcal{G}_{\alpha}$ explicitly for $\alpha\in\left\{\frac{1}{50},\frac{1}{100}\right\}$.
\end{proof}

\begin{lemma}\label{lem:2omega(n)bnd}
Let $\omega(n)$ denote the number of distinct prime divisors of $n$. Then we have 
\[
2^{\omega(n)}<\mathcal{D}_{\alpha} n^{\alpha},\qquad 3^{\omega(n)}\leq \mathcal{E}_{\alpha} n^{\alpha},\qquad\qquad \text{ where } \qquad
\mathcal{D}_{\alpha}:=\prod_{p<2^{\frac{1}{\alpha}}} \frac{2}{p^{\alpha}},\qquad \mathcal{E}_{\alpha}:=\prod_{p<3^{\frac{1}{\alpha}}} \frac{3}{p^{\alpha}}.
\]
Specifically, we have 
\begin{align*}
\mathcal{D}_{\frac{1}{10}}&<1.110\times 10^{9},& \mathcal{D}_{\frac{1}{15}}&<6.556\times 10^{115},& \mathcal{D}_{\frac{1}{14}}&<5.098\times 10^{67},& \mathcal{E}_{\frac{1}{2}}&<1.614.
\end{align*}
\end{lemma}
\begin{proof}
Writing (for $r>1$)
\[
\frac{r^{\omega(n)}}{n^{\alpha}} = \prod_{p^j\| n} \frac{r}{p^{j\alpha}},
\]
we see that if $p^{\alpha}\geq r$, then $\frac{r}{p^{j\alpha}}\leq r^{1-j}\leq 1$ and the case $j=1$ is clearly the worst case. Thus we may take the product over those $p<r^{\frac{1}{\alpha}}$ with $j=1$ assumed. We compute $\mathcal{E}_{\frac{1}{2}}$ and also $\mathcal{D}_{\alpha}$ explicitly for $\alpha\in\left\{\frac{1}{15},\frac{1}{14},\frac{1}{10}\right\}$.  
\end{proof}

\begin{lemma}\label{lem:expsum}
For any $k\in\N_{0}$ we have 
\[
\sum_{n=1}^{\infty} n^ke^{-\frac{2\pi n}{N}}\leq \frac{k!}{\left(1-e^{-\frac{2\pi}{N}}\right)^{k+1}}.
\]
Moreover, defining $\delta_1:=\frac{1}{4}$, $\delta_3=\delta_2:=\frac{1}{2}$ and $\delta_N:=1$ for $N>3$, we have 
\[
\sum_{n=1}^{\infty} n^ke^{-\frac{2\pi n}{N}}\leq \frac{k!}{(\delta_N\pi) ^{k+1}} N^{k+1}.
\]
\end{lemma}
\begin{proof}
We use the binomial series expansion (valid because $e^{-\frac{2\pi}{N}}<1$) and the well-known identity 
\[
\binom{-k-1}{n} = (-1)^n\binom{k+n}{n}=(-1)^n\binom{k+n}{k}
\]
 to write 
\[
 k!\left(1-e^{-\frac{2\pi}{N}}\right)^{-k-1}= k!\sum_{n=0}^{\infty} \binom{k+n}{k} e^{-\frac{2\pi n}{N}}= \sum_{n=0}^{\infty}e^{-\frac{2\pi n}{N}}\prod_{j=1}^{k}(n+j) \geq \sum_{n=0}^{\infty} n^ke^{-\frac{2\pi n}{N}}.
\]
This gives the first claim. For the second claim, it remains to show that
\[
\frac{k!}{\left(1-e^{-\frac{2\pi}{N}}\right)^{k+1}}\leq \frac{k!}{(\delta_N\pi)^{k+1}} N^{k+1}.
\]
For $1\leq N\leq 6$ this follows by direct calculation. For $N>2\pi$, this follows by expanding the Taylor series for $1-e^{-\frac{2\pi}{N}}$ and noting that it is alternating and decreasing in absolute value, so it may be bounded from below by the sum of the first $2J$ Taylor coefficients. 
\end{proof}
\subsection{Theta functions}
We let $Q$ be a quadratic form $Q(\bm{x})=\sum_{1\leq i\leq j\leq \ell} a_{ij} x_ix_j$ with $a_{ij}\in\Z$ (see \cite{OM} for general information about quadratic forms). For $S\subset\Z$, we let $r_{Q,S}(n)$ denote the number of solutions to $Q(\bm{x})=n$ for $\bm{x}\in S^{\ell}$, and we abbreviate $r_Q(n):=r_{Q,\Z}(n)$. For $\tau\in \H$ (with $\H$ denoting the complex upper half-plane $\tau=u+iv\in\C$ with $v>0$), we define the \begin{it}theta function\end{it}
\[
\Theta_Q(\tau):=\sum_{n\geq 0} r_Q(n) q^n. 
\]
It is well-known (for example, see \cite[Proposition 2.1]{Shimura}) that $\theta_Q$ is a modular form of weight $\ell/2$ on a particular congruence subgroup $\Gamma$ of $\SL_2(\Z)$ with a certain Nebentypus $\chi$. As such, it naturally decomposes as
\[
\Theta_Q=E_Q+f_Q,
\]
where $E_Q$ is in the space of Eisenstein series of weight $\ell/2$ on $\Gamma$ with Nebentypus $\chi$ and $f_Q$ is a cusp form in the same space.  If $\ell\geq 5$ (with a mild modification for $\ell=4$ and a more complicated modification for $\ell=3$ that only works on a restricted set of $n$), then for those $n$ for which the $n$th coefficient $a_{E_Q}(n)$ of $E_Q$ is non-zero it turns out that the coefficients of $E_Q$ grow much faster than the coefficients of $f_Q$. Moreover, one can show that $a_{E_Q}(n)>0$ if and only if $n$ is \begin{it}locally represented\end{it} (i.e., it is represented modulo every $N\in\N$). One can hence think of $E_Q$ as the main term of $\Theta_Q$ and $f_Q$ as the error term, with $r_Q(n)>0$ for sufficiently large $n$ that are locally represented. The sieving arguments in \cite{BrudernFouvry} rely on a splitting into a main term and an error term; although their arguments rely on the main term coming from the major arcs and the error term coming from the minor arcs in the Hardy-Littlewood circle method, one may alternatively take the main term to be the contribution from the Eisenstein series $E_Q$ and the error term to be a contribution from $f_Q$. Although our main term and error term in this splitting do not exactly match those from the major and minor arcs, upon further refinement the main term may be identified as the singular series in the Circle method and this indeed matches the contribution from $E_Q$. We use the interpretation of their work in terms of modular forms in order to use known results to obtain a quantitative version of their theorem that is needed to prove our main theorems.

\subsection{Uniform bounds on coefficients}
We require the following quantitative version of the uniform bound from \cite[Lemma 4.1]{Blomer} for $r_Q(n)$.
\begin{lemma}\label{lem:Blomer}
Let $\Delta_Q$ denote the determinant of the Gram matrix of a quadratic form $Q$ with rational coefficients and let 
\[
R_Q(n):=\#\{\bm{x}\in\Z^{\ell}: Q(\bm{x})\leq n\}.
\]
Then  
\[
R_Q(n)\leq \frac{\left(3\sqrt{n}\right)^{\ell}}{\sqrt{\Delta_Q}} + \ell\left(3\sqrt{n}\right)^{\ell-1}.
\]
\end{lemma}
\begin{proof}
Write 
\[
Q(\bm{x})=\sum_{1\leq i\leq j\leq \ell} a_{ij} x_ix_j.
\]
Every quadratic form is isometric over $\Z$ to a unique Minkowski-reduced form $Q$. Hence without loss of generality we may assume that $Q$ is Minkowski-reduced with $a_{11}\leq a_{22}\leq \dots\leq a_{\ell\ell}$. By the definition of a Minkowski-reduced form, for any $\bm{u}\in\Z^{\ell}$ with $\gcd(u_{i},u_{i+1},\dots,u_{\ell})=1$, we have 
\begin{equation}\label{eqn:Minkowski}
Q(\bm{u})\geq Q(\bm{e}_i),
\end{equation}
where $\bm{e}_i$ is the canonical basis element with $\bm{e}_{i,j}=\delta_{i=j}$. Taking $\bm{u}=\bm{e}_i\pm \bm{e}_j$ yields $|a_{ij}|\leq a_{ii}$ for all $i$ and $j$.

We prove the claim by induction on $\ell$. For $\ell=1$ we have $Q(x)=ax^2$ for some $a$ and $\Delta_Q=a$. We see that $Q(x)\leq n$ if and only if $|x|\leq \sqrt{\frac{n}{a}}=\sqrt{\frac{n}{\Delta_Q}}$, so there are at most $2\sqrt{\frac{n}{\Delta_Q}}+1$ solutions. 

Now suppose that the claim is true for $\ell>1$. Consider 
\begin{align*}
Q(\bm{x})&=a_{11}\left(x_1+\frac{\sum_{j=2}^{\ell} a_{1j}x_j}{2a_{11}}\right)^2 + \sum_{2\leq i\leq j\leq \ell} \left(a_{ij}-\frac{\delta_{i\neq j}}{2a_{11}} a_{1i}a_{1j}-\frac{\delta_{i=j}}{4a_{11}}a_{1j}^2\right)x_ix_j\\
&=:a_{11}\left(x_1+\frac{\sum_{j=2}^{\ell} a_{1j}x_j}{2a_{11}}\right)^2 + \widetilde{Q}\!\left(x_2,\dots,x_{\ell}\right).
\end{align*}
Set $\widetilde{\bm{x}}:=(x_2,\dots,x_{\ell})$ and $x_1'=x_1'(\widetilde{\bm{x}}):=x_1+\frac{\sum_{j=2}^{\ell} a_{1j}x_j}{2a_{11}}$. note that for each fixed $\widetilde{\bm{x}}$, there exists $0\leq m\leq 2a_{11}$ such that $x_1'(\widetilde{\bm{x}})\in \frac{m}{2a_{11}} +\Z$. Thus for each $\widetilde{\bm{x}}$ there are at most $2\sqrt{\frac{n}{a_{11}}}+1$ choices of $x_1'$ such that $|x_1'|\leq \sqrt{\frac{n}{a_{11}}}$. Using elementary row operations to relate the determinants, we have
\begin{equation}\label{eqn:Delrel}
\Delta_{\widetilde{Q}}=\frac{\Delta_{Q}}{a_{11}}.
\end{equation}
Using \eqref{eqn:Delrel}, by induction there are at most 
\[
\frac{3^{\ell-1}}{\sqrt{\Delta_{\widetilde{Q}}}}n^{\frac{\ell-1}{2}}+(\ell-1)(3\sqrt{n})^{\ell-2} = \frac{3^{\ell-1}\sqrt{a_{11}}}{\sqrt{\Delta_{Q}}}n^{\frac{\ell-1}{2}}+(\ell-1)(3\sqrt{n})^{\ell-2}
\]
points $\widetilde{\bm{x}}$ with $Q(\widetilde{\bm{x}})\leq n$, and for each of these there are at most $\frac{2\sqrt{n}}{\sqrt{a_{11}}}+1$ choices of $x_1'$. After simplifying, we conclude that the number of points is at most 
\[
\frac{(3\sqrt{n})^{\ell}}{\sqrt{\Delta_Q}} + \ell (3\sqrt{n})^{\ell-1}\left(\frac{\ell-1}{3\ell\sqrt{n}}\left(1+\frac{2\sqrt{n}}{\sqrt{a_{11}}}\right) + \frac{1}{\ell \sqrt{\Delta_Q}} \left(\sqrt{a_{11}} - \sqrt{n}\right)\right).
\]
For $n\geq a_{11}$, we have $\sqrt{a_{11}} - \sqrt{n}\leq 0$ and hence the terms inside the parentheses in the second term may be bounded by 
\[
\frac{\ell-1}{3\ell\sqrt{n}}\left(1+\frac{2\sqrt{n}}{\sqrt{a_{11}}}\right) + \frac{1}{\ell \sqrt{\Delta_Q}} \left(\sqrt{a_{11}} - \sqrt{n}\right)\leq \frac{\ell-1}{3\ell\sqrt{n}}\left( 1+\frac{2\sqrt{n}}{\sqrt{a_{11}}}\right)\leq \frac{\ell-1}{3\ell\sqrt{n}}\left(\frac{3\sqrt{n}}{\sqrt{a_{11}}}\right)\leq 1,
\]
so we are done as long as $n\geq a_{11}$.

We claim that in the case $n<a_{11}$ we have $R_Q(n)=1$ (i.e., only the zero vector $\bm{0}$ has $Q(\bm{x})<a_{11}$). Suppose for contradiction that $\bm{x}\neq 0$ exists with $Q(\bm{x})<a_{11}$. First note that 
\[
Q\left(\frac{\bm{x}}{\gcd(\bm{x})}\right)=\frac{Q(\bm{x})}{\gcd(\bm{x})^2}\leq Q(\bm{x})< a_{11}, 
\]
so without loss of generality we may assume that $\gcd(\bm{x})=1$. Using \eqref{eqn:Minkowski}, we have $Q(\bm{x})> a_{11}$, which is a contradiction. Thus for $n<a_{11}$ there are no $\bm{x}\in \Z^{\ell}\setminus\{0\}$ with $Q(\bm{x})\leq n$ and hence $R_Q(n)=1$. 
\end{proof}

\subsection{Sieving and theta functions}
We begin by explaining how sieving is used to obtain results about representations of integers as sums of squares of almost primes. For ease of notation, for $w\in\R$ we set  
\[
P_{w}(z):=\prod_{w\leq p<z} p
\]
and denote $P(z):=P_2(z)$. 

For an integer $M$, let 
\[
S_M:=\{x\in\Z: \gcd(x,M)=1\}
\]
 denote the set of integers that are relatively prime to $M$. Note that if we set 
\[
M=\prod_{p<z} p=P(z)
\]
 and $z\in\R$, then for $x\in S_{M}$ we may write 
\[
x=\prod_{p\geq z} p^{a_j},
\]
where $a_j\in\N_0$. If we have the additional restriction $x\leq\sqrt{n}$ and set $z=n^{\theta}$, then we may conclude that 
\[
n^{\theta \sum_{p\mid x} a_j} \leq  \prod_{p\geq n^{\theta}} p^{a_j}=x\leq n^{\frac{1}{2}},
\]
and thus
\[
\sum_{p\mid x} a_j \leq \frac{1}{2\theta}.
\]
Therefore we have $x\in P_{\left\lfloor\frac{1}{2\theta}\right\rfloor}$ in particular. Thus if we let $r:= \left\lfloor\frac{1}{2\theta}\right\rfloor$, then we have
\[
\{ x\in S_{P(n^{\theta})}:x\leq \sqrt{n}\} \subseteq P_r.
\]
Similarly, if we replace $P(n^{\theta})$ with $P(n^{\theta})/\prod_{p\in S} p$, then 
\[
\{ x\in S_{P(n^{\theta})/\prod_{p\in S} p }:x\leq \sqrt{n}\} \subseteq P_{r,S}.
\]
Hence if every representation of $n$ satisfies $x_j\leq \sqrt{n}$ (one sees easily that this is true for the sum of four squares, for example), then 
\begin{align*}
r_{Q,P_r}(n)&\geq r_{Q,S_{P(n^{\theta})}}(n),\\
r_{Q,P_{r,S}}(n)&\geq r_{Q,S_{P(n^{\theta})/ \prod_{p\in S} p}}(n),
\end{align*}
and it suffices to show that $r_{Q,S_{P(z)}}(n)>0$ (resp. $ r_{Q,S_{P(n^{\theta})/ \prod_{p\in S} p}}(n)>0$). It is here that sieving theory may be applied. For a vector $\bm{d}\in\N^{\ell}$ and a quadratic form $Q$, we define 
\[
Q_{\bm{d}^2}(\bm{x}):=Q(\bm{d}\cdot \bm{x})=Q(d_1x_1,d_2x_2,\dots,d_{\ell}x_{\ell}),
\]
and in particular, the diagonal quadratic forms are denoted by (with $\bm{a},\bm{d}\in\N^{\ell}$)
\[
Q_{\bm{a}\cdot \bm{d}^2}(\bm{x}):=\sum_{j=1}^{\ell} a_j (d_jx_j)^2.
\]
Then by inclusion-exclusion we have
\begin{equation}\label{eqn:sieving}
r_{Q,S_{P(z)}}(n)=\sum_{\substack{\bm{d}\in \N^{\ell}\\ d_j\mid P(z)}} \mu(d) r_{Q,\bm{d}\cdot \Z^{\ell}}(n)= \sum_{\substack{\bm{d}\in \N^{\ell}\\ d_j\mid P(z)}} \mu(d) r_{Q_{\bm{d}}}(n). 
\end{equation}
One then uses sieving theory techniques to bound the right-hand side of \eqref{eqn:sieving} from above and below by replacing $\mu(d)$ with Rosser's weights and then naturally splitting $r_{Q_{\bm{d}}}(n)$ into the contribution from the Eisenstein series and the cuspidal part. In the next sections, we investigate the contributions from the Eisenstein series and the cusp form and revisit the details for using Rosser's weights and the vector sieve of Brudern and Fouvry \cite{BrudernFouvry} in the proofs of our main theorems.

\subsection{Modular forms and congruence subgroups}
We require some well-known identities and relations involving congruence subgroups. 

\begin{lemma}\label{lem:GammaIndex}
\noindent

\noindent
\begin{enumerate}[leftmargin=*,label={\rm(\arabic*)}]
\item For $N\in\N$ we have
\begin{equation}\label{eqn:Gamma0index}
\left[\SL_2(\Z):\Gamma_0(N)\right]=N\prod_{p\mid N} \left(1+\frac{1}{p}\right).
\end{equation}
\item For $N\in\N$ we have
\begin{equation}\label{eqn:Gammaindex}
\left[\SL_2(\Z):\Gamma(N)\right]=N^3\prod_{p\mid N} \left(1-\frac{1}{p^2}\right).
\end{equation}
\end{enumerate}
\end{lemma}

\begin{lemma}\label{lem:cuspsGammaN}
Let $N\in\N$ be given. For every $\delta\mid N$, there are $\varphi\left(\frac{N}{\delta}\right)\varphi(\delta)\frac{N}{\delta}$ cusps $\rho=\gamma(i\infty)$ with $\gamma=\left(\begin{smallmatrix}a&b\\ c&d\end{smallmatrix}\right)\in \SL_2(\Z)$ for which $(c,N)=\delta$. 
\end{lemma}
\begin{proof}
Suppose that $\delta\mid c$. Writing 
\[
\rho=\frac{a}{c}=\frac{a}{\delta c'},
\]
we have $(c,N)=\delta$ if and only if $\gcd(c',N/\delta)=1$. There are hence $\varphi(N/\delta)$ choices of $c'$. Writing $a=\delta m +r$, we have $\gcd(r,\delta)=1$ and there are $\frac{N}{\delta}$ choices of $r$ modulo $N$. This gives the claim. 
\end{proof}

\subsection{Eisenstein series and the Siegel--Weil average}
In this section, we fix a vector $\bm{a}\in\N^{4}$ and consider the Eisenstein series part $E_{\bm{d}}:=E_{Q_{\bm{a}\cdot \bm{d}^2}}$ of the theta function for the diagonal form $Q_{\bm{a}\cdot \bm{d}^2}$, writing $a_{E_{\bm{d}}}(n)$ as an explicit multiple of $E_{\bm{1}}(n)$, where $\bm{1}$ is the vector of all ones.

Let $\GG(Q)$ furthermore denote a set of representatives of the classes in the genus of $Q$ and $w_Q$ denote the number of automorphs of $Q$. The Siegel--Weil average is then given by (the first identity is due to Siegel \cite{Siegel} and a generalization by Weil \cite{Weil})
\begin{equation}\label{eqn:SiegelWeil}
E_{Q}=\frac{1}{\sum_{Q'\in\GG(Q)}w_{Q'}^{-1}}  \sum_{Q'\in\GG(Q)} \frac{\Theta_{Q'}}{w_{Q'}}. 
\end{equation}

\section{An initial attempt at a finiteness theorem and the proof of Theorem \ref{thm:Nrinfinite}}\label{sec:initial}

In this section, we prove Theorem \ref{thm:Nrinfinite}.
\begin{proof}[Proof of Theorem \ref{thm:Nrinfinite}]

In order to prove the theorem, we need to show that there exists a quadratic form $Q$ which represents every positive integer up to $2^{2r+2}$ but is not universal over $P_r^{\ell}$. We choose Lagrange's quaternary form $Q(x)=\sum_{j=1}^4 x_j^2$. Since the genus of $Q$ only has one class, \eqref{eqn:SiegelWeil} becomes a single term and a computation of the local densities (or their realization as coefficients of an explicit Eisenstein series) imply that (cf. \cite[Proposition 11]{123}) that the number of representations of $n$ is precisely 
\begin{equation}\label{eqn:sumsquaresformula}
8\sum_{\substack{d\mid n\\ 4\nmid d}} d.
\end{equation}
Taking $n=2^{2r+3}$, there are precisely $24$ representations over $\Z$ and it is easy to see that the solutions are  $x_i=\pm 2^{r+1}$, $x_j=\pm 2^{r+1}$, and $x_{i'}=x_{j'}=0$, where $\{i',j'\}$ is the complement of $\{i,j\}$ in the positive integers up to $4$. 

Since every integer which is not a $P_r$ number is at least $2^{r+1}$, it is straightforward to see that every integer up to $2^{2r+2}$ is represented by $Q$ with $P_r$ numbers. 
\end{proof}

\section{Bounds on coefficients of theta functions}\label{sec:ThetaCoefficients}

\subsection{Cuspidal contribution}
In this section, we fix a vector $\bm{a}\in\N^{\ell}$ with $\ell\geq 4$ even and consider the cuspidal part $f_{\bm{a}\cdot \bm{d}^2}:=f_{Q_{\bm{a}\cdot \bm{d}^2}}$ of the theta function for the diagonal form $Q_{\bm{a}\cdot \bm{d}^2}$ in order to obtain an explicit bound on its $n$th coefficient $a_{f_{\bm{d}}}(n)$ depending only on $n$ and $\bm{d}$. We first require the following useful lemma for bounding a cusp form in terms of the  \begin{it}Petersson norm\end{it} by $\|g\|:=\sqrt{\left<g,g\right>}$, where we define  the \begin{it}Petersson inner product\end{it} between $f,g\in S_{k}(\Gamma)$ with $\Gamma\subseteq \SL_2(\Z)$ via (recalling that $\tau=u+iv$)
\[
\left<f,g\right>:=\frac{1}{\left[\SL_2(\Z): \Gamma\right]} \int_{\Gamma\backslash\H} f(\tau) \overline{g(\tau)} v^{k} \frac{du dv}{v^2}.
\]
Here we have normalized so that the inner product is independent of the choice of $\Gamma$.
\begin{lemma}\label{lem:cf(n)<norm}
 Suppose that $f\in S_{k}(\Gamma_0(N)\cap\Gamma_1(L),\psi)$ with $L\mid N$ and $\psi$ a character modulo $N$. If $f$ has the Fourier expansion $f(\tau)=\sum_{n\geq 1} c_f(n) q^n,$ then for $\alpha,\delta>0$ and $c_{\delta}$ and $C_{\delta}$ given as in Lemmas \ref{lem:cdelbnd} (1) and \ref{lem:2omega(n)bnd}, we have 
\begin{equation}\label{eqn:afbnd}
\left|c_{f}(n)\right|\leq \sqrt{\frac{\pi k}{3}} e^{2\pi} \zeta(1+4\delta)^{\frac{1}{2}}c_{\delta}^{\frac{5}{2}}\sigma_0(n)n^{\frac{k-1}{2}}\|f\| N^{1+2\delta}\prod_{p\mid N}\left(1+\frac{1}{p}\right)^{\frac{1}{2}}\varphi(L).
\end{equation}

In particular, for $k=2$, $\delta=10^{-6}$, and $\alpha=\frac{1}{15}$, $\alpha=\frac{1}{14}$, or $\alpha=\frac{1}{10}$, respectively, we have 
\begin{align}\label{eqn:k=2spec1}
\left|c_{f}(n)\right|&\leq 4.58\cdot 10^{128}\cdot  n^{\frac{17}{30}}\|f\| N^{1+2\cdot 10^{-6}}\prod_{p\mid N}\left(1+\frac{1}{p}\right)^{\frac{1}{2}} \varphi(L),\\
\label{eqn:k=2spec3}\left|c_{f}(n)\right|&\leq 4.39\cdot 10^{79}\cdot  n^{\frac{4}{7}}\|f\| N^{1+2\cdot 10^{-6}}\prod_{p\mid N}\left(1+\frac{1}{p}\right)^{\frac{1}{2}} \varphi(L),\\
\label{eqn:k=2spec2}\left|c_{f}(n)\right|&\leq 6.95\cdot 10^{18}\cdot  n^{\frac{3}{5}}\|f\| N^{1+2.5\times 10^{-6}}\prod_{p\mid N}\left(1+\frac{1}{p}\right)^{\frac{1}{2}}\varphi(L).
\end{align}
\end{lemma}
\begin{proof}
For $M\mid N$, let $H_k^{\operatorname{new}}(M,\chi)$ denote the set of normalized newforms of weight $k$ and level $M$ with Nebentypus $\chi$. Schulze-Pillot and Yenirce \cite{S-PY} constructed an explicit orthonormal basis $\{F_{g,m} : g\in H_k^{\operatorname{new}}(M,\chi), M\mid N, m\mid \frac{N}{M}\}$ with respect to the Petersson inner product on $S_{k}(N,\chi)$ such that (using a bound of Deligne \cite{Deligne})
\[
\left|a_{F_{g,m}}(n)\right|\leq \frac{\sigma_0(n) n^{\frac{k-1}{2}}}{\|g\|} m^{\frac{1}{2}}\prod_{p\mid m}\left(1+\frac{1}{p}\right)^2. 
\]
Since this basis is orthonormal and \cite[Theorem 2.5]{Cho} implies that
\[
S_{k}(\Gamma_0(N)\cap\Gamma_1(L),\psi)=\bigoplus_{\chi\pmod{L}} S_{k}(\Gamma_0(N),\psi\chi),
\]
 we have
\[
f=\sum_{\chi\pmod{L}} \sum_{M\mid N}  \sum_{g\in H_{k}^{\operatorname{new}}(M,\psi\chi)}\sum_{m\mid \frac{N}{M}} \left<f,F_{g,m}\right> F_{g,m}.
\]
Thus, noting that 
\[
\|f\|^2 = \sum_{\chi\pmod{L}}\sum_{M\mid N} \sum_{g\in H_{k}^{\operatorname{new}}(M,\psi\chi)}\sum_{m\mid \frac{N}{M}} \left|\left<f,F_{g,m}\right>\right|^2,
\]
we may use the Cauchy--Schwartz inequality to obtain 
\begin{align*}
\left|c_{f}(n)\right|&\leq \sum_{\chi\pmod{L}} \sum_{M\mid N} \sum_{g\in H_{k}^{\operatorname{new}}(M,\psi\chi)}\sum_{m\mid \frac{N}{M}} \left|\left<f_{Q,\bm{d}},F_{g,m}\right>\right|\left| c_{F_{g,m}}(n)\right|\\
&\leq \sigma_0(n)n^{\frac{k-1}{2}}\|f\|\left(\sum_{\chi\pmod{L}}\sum_{M\mid N} \sum_{m\mid \frac{N}{M}}m\prod_{p\mid m}\left(1+\frac{1}{p}\right)^{4}\sum_{g\in H_{k}^{\operatorname{new}}(M,\psi\chi)}\frac{1}{\|g\|^2}\right)^{\frac{1}{2}}.
\end{align*}
We then use the bound of Fomenko \cite{Fomenko} 
\[
\|g\|^2\geq \frac{1}{4\pi e^{4\pi}\left[\SL_2(\Z): \Gamma_0(M)\right]}
\]
together with \eqref{eqn:Gamma0index} to bound $\left|c_{f}(n)\right|$ from above by 
\begin{equation}\label{eqn:cf(n)<norm2}
\frac{\sigma_0(n)n^{\frac{k-1}{2}}\|f\|}{(4\pi)^{-\frac{1}{2}}e^{-2\pi}}
\left(\sum_{\chi\pmod{L}}\sum_{M\mid N}M\prod_{p\mid M}\left(1+\frac{1}{p}\right) \sum_{m\mid \frac{N}{M}}m\prod_{p\mid m}\left(1+\frac{1}{p}\right)^{4} \#H_{k}^{\operatorname{new}}(M,\psi\chi)\right)^{\frac{1}{2}}.
\end{equation}
Plugging in Lemma \ref{lem:cdelbnd} (1) and making the change of variables $m\to\frac{N}{Mm}$, we have 
\begin{align*}
\left|c_{f}(n)\right|&\leq \sqrt{4\pi} e^{2\pi}c_{\delta}^{\frac{5}{2}}\sigma_0(n)n^{\frac{k-1}{2}}\|f\| N^{\frac{1}{2}+2\delta} \left(\sum_{\chi\pmod{L}}\sum_{M\mid N}M^{-3\delta}\#H_{k}^{\operatorname{new}}(M,\psi\chi) \sum_{m\mid \frac{N}{M}} m^{-1-4\delta} \right)^{\frac{1}{2}}\\
&\leq \sqrt{4\pi} e^{2\pi}c_{\delta}^{\frac{5}{2}}\sigma_0(n)n^{\frac{k-1}{2}}\|f\| N^{\frac{1}{2}+2\delta} \sigma_{-1-4\delta}(N)^{\frac{1}{2}}\left(\sum_{\chi\pmod{L}}\sum_{M\mid N}\#H_{k}^{\operatorname{new}}(M,\psi\chi)\right)^{\frac{1}{2}}.
\end{align*}
Since $\#H_{k}^{\operatorname{new}}(M,\psi\chi)$ equals the dimension of the new space, the valence formula,  \eqref{eqn:Gamma0index}, and \eqref{eqn:cdelbnd} implies that 
\[
\sum_{M\mid N}\#H_{k}^{\operatorname{new}}(M,\psi\chi) \leq \dim_{\C} \left(S_{k}(N,\psi\chi)\right)\leq  \frac{k}{12}\left[\SL_2(\Z):\Gamma_0(N)\right]=\frac{k}{12} N \prod_{p\mid N}\left(1+\frac{1}{p}\right).
\]
Since this bound is independent of $\chi$, we may then bound the sum over $\chi$ by the number of characters modulo $L$, which is $\varphi(L)$. Bounding 
\[
\sigma_{-1-4\delta}(N)=\sum_{m\mid N} m^{-1-4\delta}\leq \sum_{m\geq 1} m^{-1-4\delta}=\zeta(1+4\delta),
\]
we hence obtain \eqref{eqn:afbnd}. 
 
Choosing $k=2$, $\delta=10^{-6}$, and $\alpha=\frac{1}{15}$ or $\alpha=\frac{1}{10}$ in particular and plugging in \eqref{eqn:delpart} together with $\zeta(1+4\cdot 10^{-6})<250000.6$, the evaluations of $c_{\delta}$ in Lemma \ref{lem:cdelbnd} (1) and $\mathcal{C}_{\alpha}$ in Lemma \ref{lem:sigma0bnd} imply \eqref{eqn:k=2spec1} and \eqref{eqn:k=2spec2}.
\end{proof}

\begin{lemma}\label{lem:cuspbound}
Suppose that the level of $Q_{\bm{a}\cdot \bm{d}^2}$ is $N=N_{\bm{a}\cdot \bm{d}^2}$ and abbreviate $\Delta=\Delta_{\bm{a}\cdot \bm{d}^2}:=\Delta_{Q_{\bm{a}\cdot \bm{d}^2}}$. 
\noindent

\noindent
\begin{enumerate}[leftmargin=*,label={\rm(\arabic*)}]
\item We have 
\begin{multline}\label{eqn:afQbndgen}
\left|a_{f_{\bm{a}\cdot \bm{d}^2}}(n)\right|\leq \sqrt{\frac{2^{2\ell-1}\pi \ell}{3}} e^{2\pi} \zeta(1+4\delta)^{\frac{1}{2}}c_{\delta}^{\frac{5}{2}}\mathcal{C}_{\alpha}n^{\frac{k-1}{2}+\alpha}N^{\frac{3}{2}+2\delta} \frac{3^{\ell-1}\sqrt{\left(\frac{\ell}{2}-2\right)!}}{2^{\frac{\ell}{4}-2}\sqrt{\pi}}\sigma_{-1}(N)^{\frac{1}{2}}\\
\times 
\left(\sum_{m=0}^{\frac{\ell}{2}-2} \frac{\left(\frac{2\pi}{N}\right)^{-m}}{\left(\frac{\ell}{2}-2-m\right)!}(\ell-m-2)!\left(\frac{N}{\pi\delta_N}\right)^{\ell-m-1}\left(\frac{9\pi \delta_N \Delta}{N^{\ell-1}}(\ell-m-1)+\ell^2\right)\right)^{\frac{1}{2}}.
\end{multline}
\item 
In the special case that $\ell=4$, we have 
\begin{align*}
\left|a_{f_{\bm{a}\cdot \bm{d}^2}}(n)\right|&\leq \sqrt{\frac{512\pi }{3}} e^{2\pi} \zeta(1+4\cdot 10^{-6})^{\frac{1}{2}}c_{10^{-6}}^{\frac{5}{2}}\mathcal{C}_{\alpha}n^{\frac{1}{2}+\alpha}N^{\frac{3}{2}+2\cdot 10^{-6}}\sigma_{-1}(N)^{\frac{1}{2}} \frac{54\sqrt{2}}{\pi^2}\left(27\pi \Delta+16N^3\right)^{\frac{1}{2}}\\
&\leq 20591008784 \mathcal{C}_{\alpha}\mathcal{G}_{\alpha'}^{\frac{1}{2}} n^{\frac{1}{2}+\alpha}N^{\frac{3}{2}+2\cdot 10^{-6}+\frac{\alpha}{2}} \left(27\pi \Delta+16N^3\right)^{\frac{1}{2}},\\
\left|a_{f_{\bm{a}\cdot \bm{d}^2}}(n)\right|&\leq 1.797\times 10^{21} n^{\frac{3}{5}}N^{\frac{3}{2}+2\cdot 10^{-6}+\frac{1}{200}}\left(27\pi \Delta+16N^3\right)^{\frac{1}{2}},\\
\left|a_{f_{\bm{a}\cdot\bm{d}^2}}(n)\right|&\leq 1.134\times 10^{82} n^{\frac{4}{7}}N^{\frac{3}{2}+2\cdot 10^{-6}+\frac{1}{200}}\left(27\pi \Delta+16N^3\right)^{\frac{1}{2}},\\
\left|a_{f_{\bm{a}\cdot\bm{d}^2}}(n)\right|&\leq 1.184\times 10^{131} n^{\frac{17}{30}}N^{\frac{3}{2}+2\cdot 10^{-6}+\frac{1}{200}}\left(27\pi \Delta+16N^3\right)^{\frac{1}{2}}.
\end{align*}
\end{enumerate}
\item 
\end{lemma}

\begin{proof}
(1)  We consider $\bm{a}$ to be fixed and write $f_{\bm{d}}$ for $f_{\bm{a}\cdot \bm{d}^2}$ and $\Theta_{\bm{d}}$ for $\Theta_{Q_{\bm{a}\cdot \bm{d}^2}}$. Using Lemma \ref{lem:cf(n)<norm}, it remains to bound $\|f_{\bm{d}}\|$.  By \eqref{eqn:SiegelWeil}, we may write
\[
f_{\bm{d}}= E_{\bm{d}}-\Theta_{\bm{d}}=\frac{1}{\sum_{Q\in \GG(Q_{\bm{a}\cdot \bm{d}^2})}w_Q^{-1}}\sum_{Q\in \GG(Q_{\bm{a}\cdot \bm{d}^2})} \frac{\Theta_Q-\Theta_{\bm{d}}}{w_Q}. 
\]
We set $f_{Q,\bm{d}}:=\Theta_{\bm{d}}-\Theta_Q$ and use an argument of Blomer \cite{Blomer} to bound $\|f_{Q,\bm{d}}\|$ independent of $Q$.  Specifically, we need an explicit version of \cite[Lemma 4.2]{Blomer}. For ease of notation, set
\begin{align*}
\mu_0(N)&:=\left[\SL_2(\Z):\Gamma_0(N)\right],& \mu(N)&:=\left[\SL_2(\Z):\Gamma(N)\right],\\
\mathcal{F}_0(N)&:=\Gamma_0(N)\backslash\H& \mathcal{F}(N)&:=\Gamma(N)\backslash\H,& \mathcal{F}:=\SL_2(\Z)\backslash\H.
\end{align*}
Letting $\{\gamma_1,\dots,\gamma_{\mu(N)}\}$ denote a set of representatives of $\SL_2(\Z)/\Gamma(N)$, since $f_{Q,\bm{d}}\in S_{\ell/2}(\Gamma(N))$, 
\begin{align}
\nonumber \|f_{Q,\bm{d}}\|^2 &
= \frac{1}{\mu(N)}\int_{\mathcal{F}(N)}|f_{Q,\bm{d}}(\tau)|^2 v^{\frac{\ell}{2}} \frac{du dv}{v^2}
= \frac{1}{\mu(N)}\sum_{j=1}^{\mu(N)} \int_{\gamma_j\mathcal{F}}|f_{Q,\bm{d}}(\tau)|^2 v^{\frac{\ell}{2}} \frac{du dv}{v^2}\\
&= \frac{1}{\mu(N)}\sum_{j=1}^{\mu(N)} \int_{\mathcal{F}}|f_{Q,\bm{d}}(\gamma_j\tau)|^2 \im(\gamma_j\tau)^{\frac{\ell}{2}} \frac{du dv}{v^2}\label{eqn:innerprod}= \frac{1}{\mu(N)}\sum_{j=1}^{\mu(N)} \int_{\mathcal{F}}\left|f_{Q,\bm{d}}\big|_{\frac{\ell}{2}}\gamma_j(\tau)\right|^2 v^{\frac{\ell}{2}} \frac{du dv}{v^2}.
\end{align}
The cusp width at each cusp is precisely $N$, so we may write the cusps as $\rho_{1},\dots, \rho_{\mu(N)/N}$. For the cusp $\rho_j=\gamma_j(i\infty)$ and $T:=\left(\begin{smallmatrix}1&1\\ 0&1\end{smallmatrix}\right)$, the set
\[
\left\{\gamma_j T^r: 0\leq N-1\right\}
\]
gives a set of representatives of the cosets in $\SL_2(\Z)/\Gamma(N)$ which send $i\infty$ to $\rho_j$. We may thus rewrite the right-hand side of \eqref{eqn:innerprod} as 
\[
\frac{1}{\mu(N)}\sum_{j=1}^{\mu(N)/N} \sum_{r=1}^N \int_{\mathcal{F}}\left|f_{Q,\bm{d}}\big|_{\frac{\ell}{2}}\gamma_jT^r(\tau)\right|^2 v^{\frac{\ell}{2}} \frac{du dv}{v^2}=\frac{1}{\mu(N)}\sum_{j=1}^{\mu(N)/N} \int_{\bigcup_{r=1}^N T^r\mathcal{F}}\left|f_{Q,\bm{d}}\big|_{\frac{\ell}{2}}\gamma_j(\tau)\right|^2 v^{\frac{\ell}{2}} \frac{du dv}{v^2}.
\]
Denoting the cusps with $\gcd(c,N)=\delta$ by $\rho_{\delta,j}=\gamma_{\delta,j}(i\infty)$, Lemma \ref{lem:cuspsGammaN} implies that we have
\[
\frac{1}{\mu(N)}\sum_{\delta\mid N} \sum_{j=1}^{\varphi\left(\frac{N}{\delta}\right)\varphi(\delta)\frac{N}{\delta}} \int_{\bigcup_{r=1}^N T^r\mathcal{F}}\left|f_{Q,\bm{d}}\big|_{\frac{\ell}{2}}\gamma_{\delta,j}(\tau)\right|^2 v^{\frac{\ell}{2}} \frac{du dv}{v^2}.
\]
Letting (with $q:=e^{2\pi i \tau}$)
\[
f_{\delta,j}:=f_{Q,\bm{d}}\big|_{\frac{\ell}{2}}\gamma_{\delta,j}=\sum_{n=1}^{\infty} a_{\delta,j}(n) q^{\frac{n}{N}}
\]
 and noting that 
\[
\bigcup_{r=1}^N T^r\mathcal{F}\subseteq \left\{\tau: -\frac{1}{2}\leq u\leq N-\frac{1}{2}, v\geq \frac{1}{2}\right\},
\]
we then bound the integral by 
\begin{multline}
 \int_{\bigcup_{r=1}^N T^r\mathcal{F}}\left|f_{\delta_j}(\tau)\right|^2 v^{\frac{\ell}{2}} \frac{du dv}{v^2}\leq \int_{\frac{1}{2}}^{\infty} \int_{-\frac{1}{2}}^{N-\frac{1}{2}}\left|\sum_{n=1}^{\infty} a_{\delta,j}(n) q^{\frac{n}{N}}\right|^2 v^{\frac{\ell}{2}} \frac{du dv}{v^2}\\
=\sum_{n=1}^{\infty}\sum_{m=1}^{\infty} a_{\delta,j}(n)\overline{a_{\delta,j}(m)} \int_{\frac{1}{2}}^{\infty} e^{-\frac{2\pi (n+m)v}{N}} v^{\frac{\ell}{2}-2} dv  \int_{-\frac{1}{2}}^{N-\frac{1}{2}}e^{2\pi i \frac{(n-m)}{N}u}  du\\
=N\sum_{n=1}^{\infty} |a_{\delta,j}(n)|^2 \int_{\frac{1}{2}}^{\infty} e^{-\frac{4\pi nv}{N}} v^{\frac{\ell}{2}-2} dv\label{eqn:innerexpand}=N\sum_{n=1}^{\infty} |a_{\delta,j}(n)|^2\left(\frac{N}{4\pi n}\right)^{\frac{\ell}{2}-1}\Gamma\left(\frac{\ell}{2}-1,\frac{2\pi n}{N}\right),
\end{multline}
where $\Gamma(s,x):=\int_{x}^{\infty} t^{s-1}e^{-t} dt$ is the \begin{it}incomplete gamma function\end{it}. We use \cite[Lemma 12]{Waibel} to bound
\[
a_{\delta,j}(n)=b_{Q_{\bm{a}},\delta,j}(n)-b_{Q,\delta,j}(n),
\]
where $b_{Q,\delta,j}(n)$ is the coefficient of the theta function $\Theta_Q$ at the corresponding cusp. Trivially bounding $\det(D)\leq 2^{\ell} \Delta_Q$ for the diagonal form defined before \cite[Lemma 11]{Waibel}, \cite[Lemma 12]{Waibel} implies that 
\[
\left|b_{Q,\delta,j}(n)\right|\leq 2^{\ell}r_S(n),
\]
where $S$ is an integral $\ell$-ary quadratic form of discriminant $\Delta_Q$ and level $\leq N_Q$.  Using  Lemma \ref{lem:Blomer}, we have 
\[
\left|a_{\delta,j}(n)\right|^2\leq 4^{\ell+1}\left(\frac{\left(3\sqrt{n}\right)^{\ell}}{\sqrt{N^{\ell}/\Delta_Q}} + \ell\left(3\sqrt{n}\right)^{\ell-1}\right)^2\leq 2^{2\ell+3}3^{2\ell-2}n^{\ell-1}\left(\frac{9\Delta_Q}{N^{\ell}}n+ \ell^2\right).
\]
Here we have used the inequality $(x+y)^2\leq 2x^2+2y^2$ in the last step. Noting that $\ell\geq 4$ is even, we have (for example, see \cite[8.4.8]{NIST})
\begin{align*}
\left(\frac{N}{2\pi n}\right)^{\frac{\ell}{2}-2}\Gamma\left(\frac{\ell}{2}-1,\frac{2\pi n}{N}\right) &= \frac{\left(\frac{\ell}{2}-2\right)!}{e^{\frac{2\pi n}{N}}}\sum_{m=0}^{\frac{\ell}{2}-2}\frac{\left(\frac{2\pi n}{N}\right)^{m+2-\frac{\ell}{2}}}{m!}=\frac{\left(\frac{\ell}{2}-2\right)!}{e^{\frac{2\pi n}{N}}}\sum_{m=0}^{\frac{\ell}{2}-2}\frac{\left(\frac{2\pi n}{N}\right)^{-m}}{\left(\frac{\ell}{2}-2-m\right)!}.
\end{align*}
Plugging back into \eqref{eqn:innerexpand} yields
\[
\|f_{Q,\bm{d}}\|^2\leq \frac{3^{2\ell-2}2^{\frac{3\ell}{2}+3}\left(\frac{\ell}{2}-2\right)!}{\pi} \frac{N^2}{\mu(N)}\sum_{\delta\mid N}\sum_{j=1}^{\varphi\left(\frac{N}{\delta}\right)\varphi(\delta)\frac{N}{\delta}} \sum_{m=0}^{\frac{\ell}{2}-2} \frac{\left(\frac{2\pi}{N}\right)^{-m}}{\left(\frac{\ell}{2}-2-m\right)!}\sum_{n=1}^{\infty} \frac{n^{\ell-2-m}}{e^{\frac{2\pi n}{N}}}\left(\frac{9\Delta_Q}{N^{\ell}}n + \ell^2\right).
\]
Noting that the sum over $j$ is now independent of $j$ and plugging in Lemma \ref{lem:expsum} to bound the inner sum, we obtain 
\begin{multline}\label{eqn:fQdNormbnd2}
\|f_{Q,\bm{d}}\|^2\leq \frac{3^{2\ell-2}2^{\frac{3\ell}{2}+4}\left(\frac{\ell}{2}-2\right)!}{2^{\frac{\ell}{2}-4}\pi} \frac{N^2}{\mu(N)}\sum_{m=0}^{\frac{\ell}{2}-2} \frac{\left(\frac{2\pi}{N}\right)^{-m}}{\left(\frac{\ell}{2}-2-m\right)!}(\ell-m-2)!\left(\frac{N}{\pi\delta_N}\right)^{\ell-m-1}\\
\times \left(\frac{9\pi \delta_N \Delta_Q}{N^{\ell-1}}(\ell-m-1)+\ell^2\right)\sum_{\delta\mid N} \frac{N}{\delta}\varphi\left(\frac{N}{\delta}\right)\varphi(\delta).
\end{multline}
Using the formula $\varphi(M)=M\prod_{p\mid M}\left(1-\frac{1}{p}\right)$, the sum over $\delta$ may be rewritten as 
\begin{align*}
\sum_{\delta\mid N}\frac{N}{\delta}\varphi(\delta)\varphi\left(\frac{N}{\delta}\right)&=N^2 \prod_{p\mid N}\left(1-\frac{1}{p}\right) \sum_{\delta\mid N}\frac{1}{\delta} \prod_{p\mid \gcd\left(\delta,\frac{N}{\delta}\right)}\left(1-\frac{1}{p}\right)\\
&=N\varphi(N)\sum_{\delta\mid N}\frac{1}{\delta} \prod_{p\mid \gcd\left(\delta,\frac{N}{\delta}\right)}\left(1-\frac{1}{p}\right)\leq N\varphi(N)\sigma_{-1}(N).
\end{align*}
Thus \eqref{eqn:fQdNormbnd2} becomes
\begin{multline*}
\|f_{Q,\bm{d}}\|^2\leq \frac{3^{2\ell-2}2^{\frac{3\ell}{2}+4}\left(\frac{\ell}{2}-2\right)!}{\pi} \frac{N^3\varphi(N)\sigma_{-1}(N)}{\mu(N)}\sum_{m=0}^{\frac{\ell}{2}-2} \frac{\left(\frac{2\pi}{N}\right)^{-m}}{\left(\frac{\ell}{2}-2-m\right)!}(\ell-m-2)!\left(\frac{N}{\pi\delta_N}\right)^{\ell-m-1}\\
\times \left(\frac{9\pi \delta_N \Delta_Q}{N^{\ell-1}}(\ell-m-1)+\ell^2\right).
\end{multline*}
Plugging in \eqref{eqn:Gammaindex}, we obtain 
\begin{multline}\label{eqn:fQdbndoverall}
\|f_{Q,\bm{d}}\|^2\leq \frac{3^{2\ell-2}2^{\frac{3\ell}{2}+4}\left(\frac{\ell}{2}-2\right)!}{\pi} \frac{N\sigma_{-1}(N)}{\prod_{p\mid N} \left(1+\frac{1}{p}\right)}\sum_{m=0}^{\frac{\ell}{2}-2} \frac{\left(\frac{2\pi}{N}\right)^{-m}}{\left(\frac{\ell}{2}-2-m\right)!}(\ell-m-2)!\left(\frac{N}{\pi\delta_N}\right)^{\ell-m-1}\\
\times \left(\frac{9\pi \delta_N \Delta_Q}{N^{\ell-1}}(\ell-m-1)+\ell^2\right).
\end{multline}
Plugging back into Lemma \ref{lem:cf(n)<norm}, we obtain \eqref{eqn:afQbndgen}, giving us part (1).

\noindent
(2) In the special case that $\ell=4$ and $4\mid N$ (so that $\delta_N=1$), the remaining sum over $m$ is a single term and we evaluate it as 
\[
\sqrt{2}\left(\frac{N}{\pi}\right)^{\frac{3}{2}} \left(\frac{27\pi \Delta_Q}{N^{3}}+16\right)^{\frac{1}{2}}.
\]
Thus in the case $\ell=4$ we obtain overall 
\[
\left|a_{f_{Q,\bm{d}}}(n)\right|\leq \sqrt{\frac{512\pi }{3}} e^{2\pi} \zeta(1+4\delta)^{\frac{1}{2}}c_{\delta}^{\frac{5}{2}}\mathcal{C}_{\alpha}n^{\frac{k-1}{2}+\alpha}N^{\frac{3}{2}+2\delta} \frac{54\sqrt{2}}{\pi^2} \sigma_{-1}(N)^{\frac{1}{2}}\left(27\pi \Delta_Q+16N^3\right)^{\frac{1}{2}}.
\]
The evaluations of $\mathcal{C}_{\alpha}$ from Lemma \ref{lem:sigma0bnd} and $\mathcal{G}_{\alpha'}$ from Lemma \ref{lem:sigma-1bnd} yields the explicit bounds.
\end{proof}

We next bound the determinant $\Delta_{\bm{a}\cdot \bm{d}^2}$ and the level $N_{\bm{a}\cdot \bm{d}^2}$ for the quadratic forms of interest. Let $\S$ denote the set of squarefree positive integers. 
\begin{lemma}\label{lem:DelNbnd}
Suppose that $\bm{a}$ is one of the choices $[1,1,1,k]$ with $1\leq k\leq 7$, $[1,1,2,k]$ with $2\leq k\leq 8$, $[1,1,3,k]$ with $3\leq k\leq 6$, $[1,2,2,k]$ with $2\leq k\leq 7$, $[1,2,3,k]$ with $3\leq k\leq 8$, $[1,2,4,k]$ with $4\leq k\leq 14$, or $[1,2,5,k]$ with $5\leq k\leq 15$. Then for any $\bm{d}\in \S^4$ we have 
\begin{align*}
\Delta_{\bm{a}\cdot \bm{d}^2}&\leq 2400\prod_{j=1}^4 d_j^2,&N_{\bm{a}\cdot \bm{d}^2}&\leq 520\lcm(\bm{d})^2.
\end{align*}
\end{lemma}
\begin{proof}
Since $\prod_{j=1}^4a_j\leq 150$ in each of cases, we have 
\[
\Delta_{\bm{a}\cdot \bm{d}^2} = 2^4\prod_{j=1}^4 a_j d_j^2\leq 2400\prod_{j=1}^4 d_j^2.
\]
To bound $N_{\bm{a}\cdot \bm{d}^2}$ we write $Q_{\bm{a}\cdot\bm{d}^2}(\bm{x})=\frac{1}{2}\bm{x}^T A \bm{x}$ with $A$ the diagonal matrix with entries $2a_jd_j^2$ along the diagonal. We then see that $A^{-1}$ has diagonal entries $\frac{1}{2a_jd_j^2}$, so $\lcm(\bm{a})\leq 130$ implies that 
\[
N_{\bm{a}\cdot \bm{d}^2} = 4\lcm(a_jd_j^2:1\leq j\leq 4)\leq 4\lcm(\bm{a}) \lcm(\bm{d})^2\leq 520 \lcm(\bm{d})^2.\qedhere
\]
\end{proof}

Plugging Lemma \ref{lem:DelNbnd} into Lemma \ref{lem:cuspbound} (2) yields that for these choices of $\bm{a}$ we may bound $\left|a_{f_{\bm{a}\cdot \bm{d}^2}}(n)\right|$ less than or equal to the minimum of 
\begin{align}\label{eqn:cuspboundfinal}
&1.797\times 10^{21} n^{\frac{3}{5}}(520\lcm(\bm{d})^2)^{1+2\cdot 10^{-6}+\frac{1}{200}}\left(64800\pi\prod_{j=1}^4d_j^2+16(520\lcm(\bm{d})^2)^3\right)^{\frac{1}{2}},\\
\label{eqn:cuspboundfinal2}
&1.134\times 10^{82} n^{\frac{4}{7}}(520\lcm(\bm{d})^2)^{1+2\cdot 10^{-6}+\frac{1}{200}}\left(64800\pi\prod_{j=1}^4d_j^2+16(520\lcm(\bm{d})^2)^3\right)^{\frac{1}{2}},\\
\label{eqn:cuspboundfinal3}
&1.184\times 10^{131} n^{\frac{17}{30}}(520\lcm(\bm{d})^2)^{1+2\cdot 10^{-6}+\frac{1}{200}}\left(64800\pi\prod_{j=1}^4d_j^2+16(520\lcm(\bm{d})^2)^3\right)^{\frac{1}{2}}.
\end{align}

\subsection{Eisenstein series contribution}
We next consider the contribution from the Eisenstein series. We suppose that $\bm{a}\in\N^4$ is a fixed vector and consider the quadratic form $Q_{\bm{a}\cdot \bm{d}^2}$.

We first recall the local densities for a quadratic form. For an $\ell$-ary quadratic form $Q$, set
\[
\beta_{Q,p}(m):=\lim_{U\to\{m\}} \frac{\vol_{\Z_p^{\ell}}\left(Q^{\leftarrow}(U)\right)}{\vol_{\Z_p}(U)}. 
\]
Here $U\subseteq\Z_p$ runs over open subsets of $\Z_p$ containing $m$ and for $p=\infty$ we have open subsets of $\R$. Siegel has shown that
\begin{equation}\label{eqn:Siegellocal}
a_{E_{Q}}(m)=\prod_{p} \beta_{Q,p}(m),
\end{equation}
where the product goes over all primes (including the infinite prime). We must therefore compute a lower bound on the local densities in order to obtain a bound on $a_{E_Q}(m)$. We begin by writing down an evaluation of $\beta_{Q,p}(m)$. To state this, for a multiplicative character $\chi$ and an additive character $\psi$, both of modulus $c$, we set
\[
\tau(\chi,\psi):=\sum_{x\pmod{c}} \chi(x)\psi(x).
\]
We particularly let $\chi=\chi_{a,b}$ denote a character of modulus $b$ induced from a character of conductor $a$ (we will always have either the principal character $\chi_{1,p^k}$ or the real Dirichlet character $\chi_{p,p^k}=\left(\frac{\cdot}{p}\right)$ coming from the Legendre symbol) and take $\psi(x)=\psi_{m,p^k}(x):=e^{\frac{2\pi i mx}{p^k}}$.
\begin{lemma}\label{lem:LocalDensityCompute}
Let $\bm{a},\bm{d}\in\Z^4$ be given and suppose that $p\neq 2$. We set $\alpha_j:=\ord_{p}\left(a_jd_j^2\right)$ and choose $a_j'$, $d_j'$ so that $a_jd_j^2=p^{\alpha_j}a_j'd_j'^2$. Without loss of generality, we assume that $\alpha_1\leq \alpha_2\leq \alpha_3\leq \alpha_4$. Note that the parity of $\alpha_j$ is completely determined by the parity of $\beta_j:=\ord_p(a_j)$. We also denote $\mathcal{A}_j:=\sum_{d=1}^j\alpha_d$ and for $S\subseteq\{1,2,3,4\}$ we define $a_S':=\prod_{j\in S}a_j'$, and $a_S:=\prod_{j\in S} a_{j}$. Moreover, for $1\leq \ell\leq 4$ we write
\begin{align*}
\eta_{p,\ell}^{\mathcal{A},\operatorname{o}}=\eta_{p,\ell}^{\mathcal{A},\operatorname{o}}(\bm{a}\cdot \bm{d}^2)&:=\prod_{\substack{1\leq j\leq \ell\\ \mathcal{A}_{\ell}-\beta_j\text{ odd}}} \left(\frac{a_j'}{p}\right)\varepsilon_p,\\
\eta_{p,\ell}^{\mathcal{A},\operatorname{e}}=\eta_{p,\ell}^{\mathcal{A},\operatorname{e}}(\bm{a}\cdot\bm{d}^2)&:=\prod_{\substack{1\leq j\leq \ell\\ \mathcal{A}_{\ell}-\beta_j\text{ even}}} \left(\frac{a_j'}{p}\right)\varepsilon_p.
\end{align*}
We often omit the dependence on $\bm{a}\cdot\bm{d}^2$ in the notation when it is clear. Let $\chi_{\mathcal{A}_j,p^k}$ be $\chi_{1,p^k}$ if $\mathcal{A}_j$ is even and $\chi_{p,p^k}$ otherwise.

We have
\begin{multline}\label{eqn:RQp^r-3}
\beta_{Q_{\bm{a}\cdot \bm{d}^2},p}(m)=\sum_{k=0}^{\alpha_1}\tau\left(\chi_{1,p^k},\psi_{-m,p^k}\right)+ \eta_{p,1}^{\mathcal{A},\operatorname{e}}\sum_{\substack{k=\alpha_1+1\\ k- \mathcal{A}_1\text{ odd}}}^{\alpha_2}p^{\frac{\alpha_1-k}{2}} \tau\left(\chi_{p,p^k},\psi_{-m,p^k}\right)\\
+\eta_{p,1}^{\mathcal{A},\operatorname{o}} \sum_{\substack{k=\alpha_1+1\\ k-\mathcal{A}_1\text{ even}}}^{\alpha_2}p^{\frac{\alpha_1-k}{2}} \tau\left(\chi_{1,p^k},\psi_{-m,p^k}\right)+ \left(\eta_{p,2}^{\mathcal{A},\operatorname{o}}\sum_{\substack{k=\alpha_2+1\\ k-\mathcal{A}_2\text{ even}}}^{\alpha_3}+ \eta_{p,2}^{\mathcal{A},\operatorname{e}}\sum_{\substack{k=\alpha_2+1\\ k-\mathcal{A}_2\text{ odd}}}^{\alpha_3}\right)p^{\frac{\mathcal{A}_2}{2}-k}\tau\left(\chi_{\mathcal{A}_2,p^k},\psi_{-m,p^k}\right)\\
+ \eta_{p,3}^{\mathcal{A},\operatorname{o}} \sum_{\substack{k=\alpha_3+1\\ k-\mathcal{A}_3\text{ even}}}^{\alpha_4}p^{\frac{\mathcal{A}_3-3k}{2}}\tau\left(\chi_{1,p^{k}},\psi_{-m,p^k}\right)+\eta_{p,3}^{\mathcal{A},\operatorname{e}}\sum_{\substack{k=\alpha_3+1\\ k-\mathcal{A}_3\text{ odd}}}^{\alpha_4}p^{\frac{\mathcal{A}_3-3k}{2}}\tau\left(\chi_{p,p^{k}},\psi_{-m,p^k}\right)\\
+\left(\eta_{p,4}^{\mathcal{A},\operatorname{o}}\sum_{\substack{k=\alpha_4+1\\ k\text{ even}}}^{\infty}+ \eta_{p,4}^{\mathcal{A},\operatorname{e}}\sum_{\substack{k=\alpha_4+1\\ k\text{ odd}}}^{\infty} \right)p^{\frac{\mathcal{A}_4}{2}-2k} \tau\left(\chi_{\mathcal{A}_4,p^k},\psi_{-m,p^k}\right).
\end{multline}

\end{lemma}
\begin{remark}
Note that the condition $\alpha_j$ even or odd is entirely determined by $\beta_j=\ord_p(a_j)$.
\end{remark}
\begin{proof}
One computes the local densities as in \cite{Yang}. The local density may be realized by choosing $U$ to be a ball of radius $p^{-r}$ around $m$, in which case we may write 
\[
\beta_{Q,p}(m)=\lim_{r \to \infty}\frac{R_{Q,p^r}(m)}{p^{(\ell-1)r}}\qquad\text{ with }\qquad R_{Q,p^r}(m):=\#\left\{ \bm{x}\in (\Z/p^r\Z)^{\ell}: Q(\bm{x})\equiv m\pmod{p^{r}}\right\}.
\]
Using the orthogonality of roots of unity, namely 
\[
\frac{1}{p^r}  \sum_{n\, (\operatorname{mod}\, p^r)} e^{\frac{2\pi i nm}{p^r}} = \begin{cases} 1 &\text{if }p^{r}\mid m,\\ 0 &\text{otherwise},\end{cases}
\]
we compute 
\begin{multline*}
R_{Q,p^r}(m)=\sum_{\substack{\bm{x}\in(\Z/p^r\Z)^{\ell}\\ Q(\bm{x})\equiv m\,(\operatorname{mod}\, p^r)}} 1 = \sum_{\bm{x}\in(\Z/p^r\Z)^{\ell}}\frac{1}{p^r} \sum_{n\, (\operatorname{mod}\, p^r)} e^{\frac{2\pi i n}{p^r}\left(Q(\bm{x})-m\right)}\\
= \frac{1}{p^r} \sum_{n\, (\operatorname{mod}\, p^r)} e^{-\frac{2\pi i nm}{p^r}} \sum_{\bm{x}\in(\Z/p^r\Z)^{\ell}}e^{\frac{2\pi i n}{p^r}Q(\bm{x})}.
\end{multline*}
Plugging in $Q=Q_{\bm{a}\cdot \bm{d}^2}$ with $\ell=4$, the right-hand side becomes
\[
 \frac{1}{p^r} \sum_{n\, (\operatorname{mod}\, p^r)} e^{-\frac{2\pi i nm}{p^r}} \prod_{j=1}^{4}\sum_{x_j\in \Z/p^r\Z}e^{\frac{2\pi i a_j nd_j^2x_j^2}{p^r}}= \frac{1}{p^r} \sum_{n\, (\operatorname{mod}\, p^r)} e^{-\frac{2\pi i nm}{p^r}} \prod_{j=1}^{4}G_2\!\left(na_jd_j^2,0,p^r\right),
\]
where $G_2$ is the quadratic Gauss sum 
\[
G_2(A,B,C):=\sum_{x\, (\operatorname{mod}\, c)} e^{\frac{2\pi i\left(Ax^2+Bx\right)}{C}}.
\]
We now split the sum over $n$ by $\ord_{p}(n)$, writing $n=p^kn'$ with $p\nmid n'$ and then make the change of variables $k\mapsto r-k$. Using the fact that 
\begin{equation}\label{eqn:G2gcd}
G_2(gA,gB,gC)=gG_{2}(A,B,C),
\end{equation}
 this yields
\[
p^{3r} \sum_{k=0}^{r}p^{-4k} \sum_{n'\in (\Z/p^k\Z)^{\times}} e^{-\frac{2\pi i n'm}{p^{k}}} \prod_{j=1}^{4}G_2\!\left(n'a_jd_j^2,0,p^{k}\right).
\]
Set $r_j:=\min\{\alpha_j,k\}$, so that $p^{r_j}=\gcd(n'a_jd_j^2, p^k)$. Without loss of generality, we assume that $\alpha_1\leq \alpha_2\leq \alpha_3\leq \alpha_4$. Then \eqref{eqn:G2gcd} implies that
\[
G_2\!\left(n'a_jd_j^2,0,p^k\right)=p^{r_j}G_2\left(\frac{n'a_jd_j^2}{p^{r_j}},0,p^{k-r_j}\right).
\]
In particular, if $k\leq \alpha_j$, then $r_j=k$ and 
\[
G_2\!\left(n'a_jd_j^2,0,p^k\right)=p^{k}G_2\left(\frac{n'a_jd_j^2}{p^{k}},0,1\right)=p^k,
\]
while if $k>\alpha_j$ we obtain 
\[
G_2\!\left(n'a_jd_j^2,0,p^k\right)=p^{\alpha_j}G_2\left(n'a_j'd_j'^2,0,p^{k-\alpha_j}\right)=p^{\alpha_j}G_2\left(n'a_j',0,p^{k-\alpha_j}\right),
\]
where in the last equality we used the fact that for $\gcd(t,C)=1$ we have $G_2(At^2,B,C)=G_2(A,Bt^{-1},C)$ by the change of variables $x\mapsto t^{-1}x$ in the sum.  We therefore conclude that 
\begin{multline}\label{eqn:RQp^r}
p^{-3r}R_{Q,p^r}(m)=\sum_{k=0}^{\alpha_1}\sum_{n'\in \left(\Z/p^k\Z\right)^{\times}} e^{-\frac{2\pi i n'm}{p^k}} + \sum_{k=\alpha_1+1}^{\alpha_2}p^{\alpha_1-k} \sum_{n'\in \left(\Z/p^k\Z\right)^{\times}} e^{-\frac{2\pi i n'm}{p^k}} G_2\left(n'a_1',0,p^{k-\alpha_1}\right)\\
+ \sum_{k=\alpha_2+1}^{\alpha_3}p^{\alpha_1+\alpha_2-2k} \sum_{n'\in \left(\Z/p^k\Z\right)^{\times}} e^{-\frac{2\pi i n'm}{p^k}} \prod_{j=1}^2 G_2\left(n'a_j',0,p^{k-\alpha_j}\right)\\
+ \sum_{k=\alpha_3+1}^{\alpha_4}p^{\alpha_1+\alpha_2+\alpha_3-3k} \sum_{n'\in \left(\Z/p^k\Z\right)^{\times}} e^{-\frac{2\pi i n'm}{p^k}} \prod_{j=1}^3 G_2\left(n'a_j',0,p^{k-\alpha_j}\right)\\
+ \sum_{k=\alpha_4+1}^{r}p^{\alpha_1+\alpha_2+\alpha_3+\alpha_4-4k} \sum_{n'\in \left(\Z/p^k\Z\right)^{\times}} e^{-\frac{2\pi i n'm}{p^k}} \prod_{j=1}^4 G_2\left(n'a_j',0,p^{k-\alpha_j}\right).
\end{multline}
Setting 
\[
\varepsilon_d:=\begin{cases} 1 &\text{if }d\equiv 1\pmod{4},\\ i&\text{if }d\equiv 3\pmod{4}\end{cases}
\]
for $d$ odd and letting $\left(\frac{\cdot}{\cdot}\right)$ denote the Legendre--Jacobi--Kronecker symbol, recall next that for $C$ odd and $\gcd(A,C)=1$ we have 
\[
G_2(A,0,C)=\varepsilon_C\sqrt{C}\left(\frac{A}{C}\right).
\]
Hence for $p\neq 2$, we thus have 
\begin{equation}\label{eqn:G2rewrite}
G_2\left(n'a_j',0,p^{k-\alpha_j}\right)=\begin{cases} p^{\frac{k-\alpha_j}{2}} &\text{if }k\equiv \alpha_j\pmod{2},\\ \varepsilon_p\left(\frac{n'a_j'}{p}\right) p^{\frac{k-\alpha_j}{2}}&\text{if }k\not\equiv \alpha_j\pmod{2}.\end{cases}
\end{equation}
We plug \eqref{eqn:G2rewrite} back into \eqref{eqn:RQp^r} and then split the sum depending on which $\alpha_j$ we have $k\equiv \alpha_j\pmod{2}$. Noting that if $\alpha_j\not\equiv \alpha_i\pmod{2}$ then $k$ cannot simultaneously be congruent to both of them excludes many possible combinations if we first restrict the parity of $\alpha_j$, a short calculation yields the claim.
\end{proof}

We next evaluate $\tau(\chi,\psi)$. 
\begin{lemma}\label{lem:taueval}
For an odd prime $p$ and $k\in\N_0$ we have
\begin{align}
\label{eqn:tau1} \tau\left(\chi_{1,p^k},\psi_{-m,p^k}\right)&=\begin{cases} 1 &\text{if }k=0,\\ -p^{k-1}&\text{if }\gcd(m,p^k)=p^{k-1},\\ p^{k}-p^{k-1}&\text{if }\gcd(m,p^{k})=p^k,\\ 0&\text{otherwise},\end{cases}\\
\label{eqn:taup} \tau\left(\chi_{p,p^k},\psi_{-m,p^k}\right) &=\begin{cases} \varepsilon_p p^{k-\frac{1}{2}}\chi_p\left(-\frac{m}{p^{k-1}}\right) &\text{if }\ord_p(m)=k-1\\
0&\text{otherwise}.
\end{cases}
\end{align}

\end{lemma}
\begin{proof}
Letting $\chi^*$ denote the primitive character of modulus $m^*$ associated to $\chi$ of modulus $m$ and abbreviating $\tau(\chi^*):=\tau(\chi^*,\psi_{1,m^*})$, a corrected version of \cite[Lemma 3.2]{IwaniecKowalski} yields
\[
\tau(\chi,\psi_{a,m})=\tau(\chi^*)\sum_{d\mid\gcd \left(a,\frac{m}{m^*}\right)}d \chi^*\left(\frac{m}{m^*d}\right)\overline{\chi^*\left(\frac{a}{d}\right)} \mu\left(\frac{m}{m^*d}\right).
\]
Hence we have (noting that $\chi_p(n)=0$ if $p\mid n$) 
\begin{align*}
\tau\left(\chi_{1,p^k},\psi_{-m,p^k}\right)&=\sum_{d\mid \gcd(m,p^k)} d \mu\left(\frac{p^{k}}{d}\right),\\
\tau\left(\chi_{p,p^k},\psi_{-m,p^k}\right)&=\tau(\chi_p)\sum_{d\mid \gcd(m,p^{k-1})} d \chi_p\left(\frac{p^{k-1}}{d}\right)\chi_p\left(-\frac{m}{d}\right)\mu\left(\frac{p^{k-1}}{d}\right).
\end{align*}
Note that $\chi_p\left(\frac{p^{k-1}}{d}\right)\chi_p\left(-\frac{m}{d}\right)$ vanishes for every $d=p^{j}$ except $d=p^{k-1}$, which occurs in the sum if and only if $p^{k-1}\mid m$. The claim then follows by Gauss's evaluation $\tau(\chi_p)=\varepsilon_p \sqrt{p}$.
\end{proof}

We next evaluate the local densities with a small restriction on the $\alpha_j$. We abbreviate $\delta_{R,\alpha_1;2}:=\delta_{2\mid (R-\alpha_1)}$ and $\delta_{R,\alpha_1;\cancel{2}}:=\delta_{2\nmid (R-\alpha_1)}$. Plugging \eqref{eqn:tau1} and \eqref{eqn:taup} into Lemma \ref{lem:LocalDensityCompute}, noting which sums in \eqref{eqn:RQp^r-3} vanish, and simplifying yields the following lemma.
\begin{lemma}\label{lem:LocalDensitySimplify}
Let $p$ be an odd prime and $\bm{\alpha}\in\N_0^4$ be given with $0\leq \alpha_1\leq \alpha_2\leq \alpha_3\leq \alpha_4$ and either $\bm{\alpha}=(0,0,1,1)$ or at most one $\alpha_j$ is odd. Suppose that $m\in\N$ and choose $R=\ord_p(m)$ so that $m$ is of the shape $m=p^{R}m'$ with $p\nmid m'$. If $\alpha_j=\ord_{p}(a_jd_j^2)$ with $\bm{a},\bm{d}\in\N^4$, then the following hold.
\noindent

\noindent
\begin{enumerate}[leftmargin=*,label={\rm(\arabic*)}]
\item Suppose that all $\alpha_j$ are even. Then the following hold.
\begin{itemize}
\item[(a)]
If $R<\alpha_1$, then 
\[
\beta_{Q_{\bm{a}\cdot \bm{d}^2},p}(m)=0.
\]
\item[(b)] If $\alpha_1\leq R<\alpha_2$, then 
\[
\beta_{Q_{\bm{a}\cdot \bm{d}^2},p}(m)=\eta_{p,1}^{\mathcal{A},\operatorname{o}}p^{\frac{\alpha_1}{2}+\left\lfloor\frac{R}{2}\right\rfloor}+\eta_{p,1}^{\mathcal{A},\operatorname{e}}\varepsilon_p^{3}\delta_{R,\alpha_1,2}p^{\frac{\alpha_1+R}{2}}\chi_p(m') -\eta_{p,1}^{\mathcal{A},\operatorname{o}}\delta_{R,\alpha_1,\cancel{2}}p^{\frac{\alpha_1+R-1}{2}}.
\]
\item[(c)] If $\alpha_2\leq R<\alpha_3$, then 
\begin{multline*}
\beta_{Q_{\bm{a}\cdot\bm{d}^2},p}(m)=p^{\left\lfloor\tfrac{\mathcal{A}_2}{2}\right\rfloor}+\left(\eta_{p,2}^{\mathcal{A},\operatorname{o}}\left\lfloor\tfrac{R-\alpha_2}{2}\right\rfloor +\eta_{p,2}^{\mathcal{A},\operatorname{e}}\left\lfloor\tfrac{R+1-\alpha_2}{2}\right\rfloor\right) p^{\tfrac{\mathcal{A}_2}{2}}\left(1-\tfrac{1}{p}\right)\\
-\left(\eta_{p,2}^{\mathcal{A},\operatorname{o}}\delta_{2\nmid R} +\eta_{p,2}^{\mathcal{A},\operatorname{e}}\delta_{2\mid R}\right) p^{\tfrac{\mathcal{A}_2}{2}-1}.
\end{multline*}
\item[(d)] If $\alpha_3\leq R<\alpha_4$, then 
\begin{multline*}
\beta_{Q_{\bm{a}\cdot\bm{d}^2},p}(m)
=p^{\tfrac{\mathcal{A}_2}{2}}+\tfrac{\alpha_3-\alpha_2}{2}\left(\eta_{p,2}^{\mathcal{A},\operatorname{o}}+\eta_{p,2}^{\mathcal{A},\operatorname{e}}\right)p^{\tfrac{\mathcal{A}_2}{2}}\left(1-\tfrac{1}{p}\right)+\eta_{p,3}^{\mathcal{A},\operatorname{o}}p^{\tfrac{\mathcal{A}_2}{2}-1}\left(1-p^{\tfrac{\alpha_3}{2}-\left\lfloor\tfrac{R}{2}\right\rfloor}\right)-\eta_{p,3}^{\mathcal{A},\operatorname{o}}\delta_{2\nmid R} p^{\frac{\mathcal{A}_3-3-R}{2}}\\
+\eta_{p,3}^{\mathcal{A},\operatorname{e}}\delta_{2\mid R}p^{\frac{\mathcal{A}_3-R}{2}-1}\chi_p(-m').
\end{multline*}
\item[(e)] If $R\geq \alpha_4$, then
\begin{multline*} 
\beta_{Q_{\bm{a}\cdot\bm{d}^2},p}(m)=p^{\tfrac{\mathcal{A}_2}{2}}+\tfrac{\alpha_3-\alpha_2}{2} \left(1 +\left(\frac{a_1'a_2'}{p}\right)\varepsilon_p^2\right) p^{\tfrac{\mathcal{A}_2}{2}}\left(1-\tfrac{1}{p}\right)+p^{\tfrac{\mathcal{A}_2}{2}-1}\left(1-p^{\tfrac{\alpha_3-\alpha_4}{2}}\right)\\
+\tfrac{p^{\tfrac{\mathcal{A}_3-\alpha_4}{2}-2}}{1+\tfrac{1}{p}}\left(1-p^{\alpha_4-2\left\lfloor \tfrac{R}{2}\right\rfloor} +\left(\frac{a_{1234}'}{p}\right)p\left(1-p^{\alpha_4-2\left\lfloor\tfrac{R+1}{2}\right\rfloor}\right)\right) - \left(\delta_{2\nmid R} + \left(\frac{a_{1234}'}{p}\right)\delta_{2\mid R}\right)p^{\frac{\mathcal{A}_4}{2}-R-2}.
\end{multline*}
\end{itemize}

\item If $\alpha_1$ is odd, then the following hold.
\begin{itemize}
\item[(a)]
If $R<\alpha_1$, then 
\[
\beta_{Q_{\bm{a}\cdot \bm{d}^2},p}(m)=0.
\]
\item[(b)] If $\alpha_1\leq R<\alpha_2$, then 
\[
\beta_{Q_{\bm{a}\cdot \bm{d}^2},p}(m)=\left(1-\eta_{p,1}^{\mathcal{A},\operatorname{o}}\right)p^{\alpha_1}+ \eta_{p,1}^{\mathcal{A},\operatorname{o}}p^{\left\lfloor\frac{R+\alpha_1}{2}\right\rfloor}+\eta_{p,1}^{\mathcal{A},\operatorname{e}}\varepsilon_p^{3}\delta_{R,\alpha_1,2}p^{\frac{\alpha_1+R}{2}}\chi_p(m') -\eta_{p,1}^{\mathcal{A},\operatorname{o}}\delta_{R,\alpha_1,\cancel{2}}p^{\frac{\alpha_1+R-1}{2}}.
\]

\item[(c)] If $\alpha_2\leq R<\alpha_3$, then 
\[
\beta_{Q_{\bm{a}\cdot\bm{d}^2},p}(m)= p^{\left\lfloor\tfrac{\mathcal{A}_2}{2}\right\rfloor}+\varepsilon_p^3\chi_p(m')\left(\eta_{p,2}^{\mathcal{A},\operatorname{o}}\delta_{2\nmid R} +\eta_{p,2}^{\mathcal{A},\operatorname{e}}\delta_{2\mid R}\right) p^{\tfrac{\mathcal{A}_2-1}{2}}.
\]

\item[(d)] If $\alpha_3\leq R<\alpha_4$, then 
\[
\beta_{Q_{\bm{a}\cdot \bm{d}^2},p}(m)=p^{\left\lfloor\tfrac{\mathcal{A}_2}{2}\right\rfloor}+\eta_{p,3}^{\mathcal{A},\operatorname{o}}p^{\tfrac{\mathcal{A}_2-1}{2}}\left(1-p^{-\left\lfloor\tfrac{R-\alpha_3+1}{2}\right\rfloor}\right)-\eta_{p,3}^{\mathcal{A},\operatorname{o}}\delta_{2\mid R} p^{\frac{\mathcal{A}_3-3-R}{2}}+\eta_{p,3}^{\mathcal{A},\operatorname{e}}\delta_{2\nmid R}p^{\frac{\mathcal{A}_3-R}{2}-1}\chi_p(-m').
\]

\item[(e)] If $R\geq \alpha_4$, then 
\begin{multline*}
\beta_{Q_{\bm{a}\cdot \bm{d}^2},p}(m)=p^{\left\lfloor\frac{\mathcal{A}_2}{2}\right\rfloor}+ \eta_{p,3}^{\mathcal{A},\operatorname{o}}p^{\frac{\mathcal{A}_2-1}{2}}\left(1-p^{-\frac{\alpha_4-\alpha_3}{2}}\right)\\
+ \delta_{2\mid R}\eta_{p,4}^{\mathcal{A},\operatorname{e}} p^{\frac{\mathcal{A}_4}{2}-2(R+1)+R+\frac{1}{2}}\varepsilon_p \chi_p(-m')+  \delta_{2\nmid R}\eta_{p,4}^{\mathcal{A},\operatorname{o}} p^{\frac{\mathcal{A}_4}{2}-2(R+1)+R+\frac{1}{2}}\varepsilon_p \chi_p(-m').
\end{multline*}
\end{itemize}
\item If $\alpha_2$ is odd, then the following hold.
\begin{itemize}
\item[(a)]
If $R<\alpha_1$, then 
\[
\beta_{Q_{\bm{a}\cdot \bm{d}^2},p}(m)=0.
\]
\item[(b)] If $\alpha_1\leq R<\alpha_2$, then 
\[
\beta_{Q_{\bm{a}\cdot \bm{d}^2},p}(m)=\eta_{p,1}^{\mathcal{A},\operatorname{o}}p^{\frac{\alpha_1}{2}+\left\lfloor\frac{R}{2}\right\rfloor}+\eta_{p,1}^{\mathcal{A},\operatorname{e}}\varepsilon_p^{3}\delta_{R,\alpha_1,2}p^{\frac{\alpha_1+R}{2}}\chi_p(m') -\eta_{p,1}^{\mathcal{A},\operatorname{o}}\delta_{R,\alpha_1,\cancel{2}}p^{\frac{\alpha_1+R-1}{2}}.
\]

\item[(c)] If $\alpha_2\leq R<\alpha_3$, then 
\[
\beta_{Q_{\bm{a}\cdot\bm{d}^2},p}(m)= p^{\left\lfloor\tfrac{\mathcal{A}_2}{2}\right\rfloor}+\varepsilon_p^3\chi_p(m')\left(\eta_{p,2}^{\mathcal{A},\operatorname{o}}\delta_{2\nmid R} +\eta_{p,2}^{\mathcal{A},\operatorname{e}}\delta_{2\mid R}\right) p^{\tfrac{\mathcal{A}_2-1}{2}}.
\]

\item[(d)] If $\alpha_3\leq R<\alpha_4$, then 
\[
\beta_{Q_{\bm{a}\cdot \bm{d}^2},p}(m)=p^{\left\lfloor\tfrac{\mathcal{A}_2}{2}\right\rfloor}+\eta_{p,3}^{\mathcal{A},\operatorname{o}}p^{\tfrac{\mathcal{A}_2-1}{2}}\left(1-p^{-\left\lfloor\tfrac{R-\alpha_3+1}{2}\right\rfloor}\right)-\eta_{p,3}^{\mathcal{A},\operatorname{o}}\delta_{2\mid R} p^{\frac{\mathcal{A}_3-3-R}{2}}+\eta_{p,3}^{\mathcal{A},\operatorname{e}}\delta_{2\nmid R}p^{\frac{\mathcal{A}_3-R}{2}-1}\chi_p(-m').
\]

\item[(e)] If $R\geq \alpha_4$, then 
\begin{multline*}
\beta_{Q_{\bm{a}\cdot \bm{d}^2},p}(m)=p^{\left\lfloor\frac{\mathcal{A}_2}{2}\right\rfloor}+ \eta_{p,3}^{\mathcal{A},\operatorname{o}}p^{\frac{\mathcal{A}_2-1}{2}}\left(1-p^{-\frac{\alpha_4-\alpha_3}{2}}\right)\\
+ \delta_{2\mid R}\eta_{p,4}^{\mathcal{A},\operatorname{e}} p^{\frac{\mathcal{A}_4}{2}-2(R+1)+R+\frac{1}{2}}\varepsilon_p \chi_p(-m')+  \delta_{2\nmid R}\eta_{p,4}^{\mathcal{A},\operatorname{o}} p^{\frac{\mathcal{A}_4}{2}-2(R+1)+R+\frac{1}{2}}\varepsilon_p \chi_p(-m').
\end{multline*}

\end{itemize}

\item If $\alpha_3$ is odd and $\bm{\alpha}\neq (0,0,1,1)$, then the following hold.
\begin{itemize}
\item[(a)]
If $R<\alpha_1$, then 
\[
\beta_{Q_{\bm{a}\cdot \bm{d}^2},p}(m)=0.
\]
\item[(b)] If $\alpha_1\leq R<\alpha_2$, then 
\[
\beta_{Q_{\bm{a}\cdot \bm{d}^2},p}(m)= \eta_{p,1}^{\mathcal{A},\operatorname{o}}p^{\frac{\alpha_1}{2}+\left\lfloor\frac{R}{2}\right\rfloor}+\eta_{p,1}^{\mathcal{A},\operatorname{e}}\varepsilon_p^{3}\delta_{R,\alpha_1,2}p^{\frac{\alpha_1+R}{2}}\chi_p(m') -\eta_{p,1}^{\mathcal{A},\operatorname{o}}\delta_{R,\alpha_1,\cancel{2}}p^{\frac{\alpha_1+R-1}{2}}.
\]
\item[(c)] If $\alpha_2\leq R<\alpha_3$, then 
\begin{multline*}
\beta_{Q_{\bm{a}\cdot\bm{d}^2},p}(m)=p^{\left\lfloor\tfrac{\mathcal{A}_2}{2}\right\rfloor}+\left(\eta_{p,2}^{\mathcal{A},\operatorname{o}}\left\lfloor\tfrac{R-\alpha_2}{2}\right\rfloor +\eta_{p,2}^{\mathcal{A},\operatorname{e}}\left\lfloor\tfrac{R+1-\alpha_2}{2}\right\rfloor\right) p^{\tfrac{\mathcal{A}_2}{2}}\left(1-\tfrac{1}{p}\right)\\
-\left(\eta_{p,2}^{\mathcal{A},\operatorname{o}}\delta_{2\nmid R} +\eta_{p,2}^{\mathcal{A},\operatorname{e}}\delta_{2\mid R}\right) p^{\tfrac{\mathcal{A}_2}{2}-1}.
\end{multline*}

\item[(d)] If $\alpha_3\leq R<\alpha_4$, then
\begin{multline*}
\beta_{Q_{\bm{a}\cdot\bm{d}^2},p}(m)=p^{\tfrac{\mathcal{A}_2}{2}}+\left(\eta_{p,2}^{\mathcal{A},\operatorname{o}}\tfrac{\alpha_3-\alpha_2-1}{2}+\eta_{p,2}^{\mathcal{A},\operatorname{e}}\tfrac{\alpha_3-\alpha_2+1}{2}\right) p^{\tfrac{\mathcal{A}_2}{2}}\left(1-\tfrac{1}{p}\right)\\
+\eta_{p,3}^{\mathcal{A},\operatorname{o}}p^{\tfrac{\mathcal{A}_2}{2}-1}\left(1-p^{-\left\lfloor\tfrac{R-\alpha_3}{2}\right\rfloor}\right)-\eta_{p,3}^{\mathcal{A},\operatorname{o}}\delta_{2\mid R} p^{\frac{\mathcal{A}_3-3-R}{2}}+\eta_{p,3}^{\mathcal{A},\operatorname{e}}\delta_{2\nmid R}p^{\frac{\mathcal{A}_3-R}{2}-1}\chi_p(-m').
\end{multline*}

\item[(e)] If $R\geq \alpha_4$, then 
\begin{multline*}
\beta_{Q_{\bm{a}\cdot\bm{d}^2},p}(m)=p^{\tfrac{\mathcal{A}_2}{2}}+\left(\eta_{p,2}^{\mathcal{A},\operatorname{o}}\tfrac{\alpha_3-\alpha_2-1}{2}+\eta_{p,2}^{\mathcal{A},\operatorname{e}}\tfrac{\alpha_3-\alpha_2+1}{2}\right) p^{\tfrac{\mathcal{A}_2}{2}}\left(1-\tfrac{1}{p}\right)+\eta_{p,3}^{\mathcal{A},\operatorname{o}}p^{\tfrac{\mathcal{A}_2}{2}-1}\left(1-p^{-\tfrac{\alpha_4-\alpha_3-1}{2}}\right)\\
+ \delta_{2\mid R}\eta_{p,4}^{\mathcal{A},\operatorname{e}} p^{\frac{\mathcal{A}_4}{2}-2(R+1)+R+\frac{1}{2}}\varepsilon_p \chi_p(-m')+  \delta_{2\nmid R}\eta_{p,4}^{\mathcal{A},\operatorname{o}} p^{\frac{\mathcal{A}_4}{2}-2(R+1)+R+\frac{1}{2}}\varepsilon_p \chi_p(-m').
\end{multline*}
\end{itemize}

\item If $\alpha_4$ is odd and $\bm{\alpha}\neq (0,0,1,1)$, then the following hold.
\begin{itemize}
\item[(a)]
If $R<\alpha_1$, then 
\[
\beta_{Q_{\bm{a}\cdot \bm{d}^2},p}(m)=0.
\]
\item[(b)] If $\alpha_1\leq R<\alpha_2$, then 
\[
\beta_{Q_{\bm{a}\cdot \bm{d}^2},p}(m)= \eta_{p,1}^{\mathcal{A},\operatorname{o}}p^{\frac{\alpha_1}{2}+\left\lfloor\frac{R}{2}\right\rfloor}+\eta_{p,1}^{\mathcal{A},\operatorname{e}}\varepsilon_p^{3}\delta_{R,\alpha_1,2}p^{\frac{\alpha_1+R}{2}}\chi_p(m') -\eta_{p,1}^{\mathcal{A},\operatorname{o}}\delta_{R,\alpha_1,\cancel{2}}p^{\frac{\alpha_1+R-1}{2}}.
\]

\item[(c)] If $\alpha_2\leq R<\alpha_3$, then 
\begin{multline*}
\beta_{Q_{\bm{a}\cdot\bm{d}^2},p}(m)=p^{\left\lfloor\tfrac{\mathcal{A}_2}{2}\right\rfloor}+\left(\eta_{p,2}^{\mathcal{A},\operatorname{o}}\left\lfloor\tfrac{R-\alpha_2}{2}\right\rfloor +\eta_{p,2}^{\mathcal{A},\operatorname{e}}\left\lfloor\tfrac{R+1-\alpha_2}{2}\right\rfloor\right) p^{\tfrac{\mathcal{A}_2}{2}}\left(1-\tfrac{1}{p}\right)\\
-\left(\eta_{p,2}^{\mathcal{A},\operatorname{o}}\delta_{2\nmid R} +\eta_{p,2}^{\mathcal{A},\operatorname{e}}\delta_{2\mid R}\right) p^{\tfrac{\mathcal{A}_2}{2}-1}.
\end{multline*}

\item[(d)] If $\alpha_3\leq R<\alpha_4$, then 
\begin{multline*}
\beta_{Q_{\bm{a}\cdot\bm{d}^2},p}(m)=p^{\tfrac{\mathcal{A}_2}{2}}+\tfrac{\alpha_3-\alpha_2}{2}\left(\eta_{p,2}^{\mathcal{A},\operatorname{o}}+\eta_{p,2}^{\mathcal{A},\operatorname{e}}\right)p^{\tfrac{\mathcal{A}_2}{2}}\left(1-\tfrac{1}{p}\right)+\eta_{p,3}^{\mathcal{A},\operatorname{o}}p^{\tfrac{\mathcal{A}_2}{2}-1}\left(1-p^{\tfrac{\alpha_3}{2}-\left\lfloor\tfrac{R}{2}\right\rfloor}\right)-\eta_{p,3}^{\mathcal{A},\operatorname{o}}\delta_{2\nmid R} p^{\frac{\mathcal{A}_3-3-R}{2}}\\
+\eta_{p,3}^{\mathcal{A},\operatorname{e}}\delta_{2\mid R}p^{\frac{\mathcal{A}_3-R}{2}-1}\chi_p(-m').
\end{multline*}

\item[(e)] If $R\geq \alpha_4$, then 
\begin{multline*}
\beta_{Q_{\bm{a}\cdot\bm{d}^2},p}(m)=p^{\tfrac{\mathcal{A}_2}{2}}+\tfrac{\alpha_3-\alpha_2}{2}\left(\eta_{p,2}^{\mathcal{A},\operatorname{o}}+\eta_{p,2}^{\mathcal{A},\operatorname{e}}\right)p^{\tfrac{\mathcal{A}_2}{2}}\left(1-\tfrac{1}{p}\right)+\eta_{p,3}^{\mathcal{A},\operatorname{o}}p^{\tfrac{\mathcal{A}_2}{2}-1}\left(1-p^{\tfrac{\alpha_3}{2}-\left\lfloor\tfrac{\alpha_4}{2}\right\rfloor}\right)\\
+ \delta_{2\mid R}\eta_{p,4}^{\mathcal{A},\operatorname{e}} p^{\frac{\mathcal{A}_4}{2}-2(R+1)+R+\frac{1}{2}}\varepsilon_p \chi_p(-m')+  \delta_{2\nmid R}\eta_{p,4}^{\mathcal{A},\operatorname{o}} p^{\frac{\mathcal{A}_4}{2}-2(R+1)+R+\frac{1}{2}}\varepsilon_p \chi_p(-m').
\end{multline*}
\end{itemize}

\item If $\bm{\alpha}=(0,0,1,1)$, then the following hold.
\begin{itemize}
\item[(a)] If $R=0$, then 
\[
\beta_{Q_{\bm{a}\cdot\bm{d}^2},p}(m)=1-\eta_{p,2}^{\mathcal{A},\operatorname{e}}p^{-1}.
\]

\item[(b)] If $R\geq 1$, then 
\begin{multline*}
\beta_{Q_{\bm{a}\cdot\bm{d}^2},p}(m)=1+\eta_{p,2}^{\mathcal{A},\operatorname{e}}p^{-1}(p-1)+\eta_{p,4}^{\mathcal{A},\operatorname{o}}(p-1)p^{-2}\left(\frac{1-p^{-2\left\lfloor\frac{R}{2}\right\rfloor}}{1-p^{-2}}\right) + \eta_{p,4}^{\mathcal{A},\operatorname{e}}(p-1)p^{-3}\left(\frac{1-p^{-2\left\lfloor\frac{R-1}{2}\right\rfloor}}{1-p^{-2}}\right)\\
-\left(\delta_{2\mid R}\eta_{p,4}^{\mathcal{A},\operatorname{e}}+  \delta_{2\nmid R}\eta_{p,4}^{\mathcal{A},\operatorname{o}}\right) p^{-R-1}.
\end{multline*}
\end{itemize}

\end{enumerate}
\end{lemma}

For ease of notation, we write $p^{\nu}\| \bm{d}$ if $p^{\nu}\| \prod_{j=1}^4d_j$. For $\bm{d}\in\S^4$ with $p^{\nu}\|\bm{d}$ we define
\[
\omega_{\nu}(p)=\omega_{\nu,\bm{a}\cdot\bm{d}^2}(p):= \frac{\beta_{Q_{\bm{a}\cdot \bm{d}^2},p}(m)}{\beta_{Q_{\bm{a}},p}(m)},\qquad\qquad  \omega(\bm{d},m):=\prod_{p^{\nu}\| \bm{d}} \omega_{\nu}(p).
\]

\begin{lemma}\label{lem:LocalDensityBounds}
\noindent

\noindent
\begin{enumerate}[leftmargin=*,label={\rm(\arabic*)}]
\item 
Setting $X:=a_{E_{Q_{\bm{a}}}}(n)$, for any $\bm{d}\in\S^4$ with $\gcd(d_j,30)=1$, we have 
\[
a_{E_{Q_{\bm{a}\cdot \bm{d}^2}}}(m) =\frac{X}{d_1d_2d_3d_4} \omega(\bm{d},m).
\]  

\item Suppose that $p$ is an odd prime and $p\nmid \prod_{j=1}^4a_j$. We have $\omega_0(p) = 1$ and for $1 \leq \nu \leq 4$ and $m = p^R m'$ with $p\nmid m'$ we may evaluate $\omega_{\nu}(p)$ as follows. 
\noindent

\noindent
\begin{enumerate}
\item[(a)]  For $R = 0$, we have
\begin{align*}
\omega_\nu(p) = \begin{cases}
\frac{1 + \eta_{p,3}^{\mathcal{A},\operatorname{e}}(\bm{a}\cdot \bm{d}^2)\varepsilon_p\chi_p(-m')p^{-1}}{1 - \eta_{p,4}^{\mathcal{A},\operatorname{e}}(\bm{a})p^{-2}} & \text{if } \nu = 1,\vspace{.1cm}\\
\frac{1 - \eta_{p,2}^{\mathcal{A},\operatorname{e}}(\bm{a}\cdot\bm{d}^2) p^{-1}}{1 - \eta_{p,4}^{\mathcal{A},\operatorname{e}}(\bm{a})p^{-2}} & \text{if } \nu = 2,\vspace{.1cm}\\
\frac{1 + \eta_{p,1}^{\mathcal{A},\operatorname{e}}(\bm{a}\cdot\bm{d}^2)\varepsilon_p\chi_p(-m')}{1 - \eta_{p,4}^{\mathcal{A},\operatorname{e}}(\bm{a})p^{-2}} & \text{if } \nu = 3,\vspace{.1cm}\\
\hspace{1cm} 0 & \text{if } \nu = 4.
\end{cases}
\end{align*}
\item[(b)]
For $R = 1$, we have
 \begin{align*}
\omega_\nu(p) = \begin{cases}
\frac{1 - p^{-2}}{1 + \eta_{p,4}^{\mathcal{A},\operatorname{e}}(\bm{a})\left(p^{-1} - p^{-2}\right)  - p^{-3}} & \text{if } \nu = 1 \vspace{.1cm}\\
\frac{1 +\eta_{p,2}^{\mathcal{A},\operatorname{e}}(\bm{a}\cdot\bm{d}^2)\left(1-p^{-1}\right)  -  p^{-1}}{1 + \eta_{p,4}^{\mathcal{A},\operatorname{e}}(\bm{a})\left(p^{-1} - p^{-2}\right)  - p^{-3}} & \text{if } \nu = 2 \vspace{.1cm}\\
\hspace{1cm} 0 & \text{if } \nu = 3 \vspace{.1cm}\\
\hspace{1cm} 0 & \text{if } \nu = 4.
\end{cases}
\end{align*} 

\item[(c)] 
For $R \geq 2$ and $R$ even, we have
\begin{align*}
\omega_\nu(p) = \begin{cases}
\frac{1 + p^{-1}(1- p^{-1})+\frac{p^{-2}}{p+1}(1- p^{2-R})\left( 1+ 
 \eta_{p,4}^{\mathcal{A},\operatorname{e}}\left(\bm{a}\cdot \bm{d}^2\right)p \right) -  \eta_{p,4}^{\mathcal{A},\operatorname{e}}\left(\bm{a}\cdot \bm{d}^2\right)p^{-1-R}}{1 + \frac{p^{-1}}{p+1}(1- p^{-R})\left( 1+ \eta_{p,4}^{\mathcal{A},\operatorname{e}}(\bm{a})p \right) - \eta_{p,4}^{\mathcal{A},\operatorname{e}}(\bm{a})p^{-2-R}} & \text{if } \nu = 1,\vspace{.1cm}\\ 
\frac{1 + \left(1 + \eta_{p,2}^{\mathcal{A},\operatorname{e}}\left(\bm{a}\cdot \bm{d}^2\right)\right)(1 - p^{-1})+\frac{p^{-1}}{p+1}(1- p^{2-R})\left( 1+ \eta_{p,4}^{\mathcal{A},\operatorname{e}}\left(\bm{a}\cdot \bm{d}^2\right)p \right) - \eta_{p,4}^{\mathcal{A},\operatorname{e}}\left(\bm{a}\cdot \bm{d}^2\right)p^{-R}}{1 + \frac{p^{-1}}{p+1}(1- p^{-R})\left( 1+ \eta_{p,4}^{\mathcal{A},\operatorname{e}}(\bm{a})p \right) - \eta_{p,4}^{\mathcal{A},\operatorname{e}}(\bm{a})p^{-2-R}} & \text{if } \nu = 2, \vspace{.1cm}\\
\frac{p + \frac{1}{p+1}(1- p^{2-R})\left( 1+ \eta_{p,4}^{\mathcal{A},\operatorname{e}}\left(\bm{a}\cdot \bm{d}^2\right)p \right) - \eta_{p,4}^{\mathcal{A},\operatorname{e}}\left(\bm{a}\cdot \bm{d}^2\right)p^{1-R}}{1 + \frac{p^{-1}}{p+1}(1- p^{-R})\left( 1+ \eta_{p,4}^{\mathcal{A},\operatorname{e}}(\bm{a})p \right) - \eta_{p,4}^{\mathcal{A},\operatorname{e}}(\bm{a})p^{-2-R}} & \text{if } \nu = 3, \vspace{.1cm}\\
\frac{p^2 + \frac{p}{p+1}(1- p^{2-R})\left( 1+  \eta_{p,4}^{\mathcal{A},\operatorname{e}}\left(\bm{a}\cdot \bm{d}^2\right)p \right) - \eta_{p,4}^{\mathcal{A},\operatorname{e}}\left(\bm{a}\cdot \bm{d}^2\right)p^{2-R}}{1 + \frac{p^{-1}}{p+1}(1- p^{-R})\left( 1+ \eta_{p,4}^{\mathcal{A},\operatorname{e}}(\bm{a})p \right) - \eta_{p,4}^{\mathcal{A},\operatorname{e}}(\bm{a})p^{-2-R}} & \text{if } \nu = 4.
\end{cases}
\end{align*}

\item[(d)] For $R \geq 3$ and $R$ odd, we have
\begin{align*}
\omega_\nu(p) = \begin{cases}
\frac{1 + p^{-1}(1- p^{-1})+\frac{p^{-2}}{p+1} \left( 1 - p^{3-R} + \eta_{p,4}^{\mathcal{A},\operatorname{e}}\left(\bm{a}\cdot \bm{d}^2\right)p(1- p^{1-R}) \right) \ - \ p^{-1-R}}{1 + \frac{p^{-1}}{p+1} \left( 1 - p^{1-R} + \eta_{p,4}^{\mathcal{A},\operatorname{e}}(\bm{a})p(1- p^{-1-R}) \right) \ - \ p^{-2-R}} & \text{if } \nu = 1, \vspace{.1cm}\\ 
\frac{1 + \left(1 + \eta_{p,2}^{\mathcal{A},\operatorname{e}}\left(\bm{a}\cdot\bm{d}^2\right) \right)(1 - p^{-1})+\frac{p^{-1}}{p+1} \left( 1 - p^{3-R} + \eta_{p,4}^{\mathcal{A},\operatorname{e}}\left(\bm{a}\cdot\bm{d}^2\right)p(1- p^{1-R}) \right) \ - \ p^{-R}}{1 + \frac{p^{-1}}{p+1} \left( 1 - p^{1-R} +\eta_{p,4}^{\mathcal{A},\operatorname{e}}(\bm{a})p (1- p^{-1-R}) \right) \ - \ p^{-2-R}} & \text{if } \nu = 2, \vspace{.1cm}\\
\frac{p + \frac{1}{p+1} \left( 1 - p^{3-R} + \eta_{p,4}^{\mathcal{A},\operatorname{e}}\left(\bm{a}\cdot\bm{d}^2\right)p(1- p^{1-R}) \right) \ - \ p^{1-R}}{1 + \frac{p^{-1}}{p+1} \left( 1 - p^{1-R} + \eta_{p,4}^{\mathcal{A},\operatorname{e}}(\bm{a})p (1- p^{-1-R}) \right) \ - \ p^{-2-R}} & \text{if } \nu = 3, \vspace{.1cm}\\
\frac{p^2 + \frac{p}{p+1} \left( 1 - p^{3-R} + \eta_{p,4}^{\mathcal{A},\operatorname{e}}\left(\bm{a}\cdot\bm{d}^2\right)p(1- p^{1-R}) \right) \ - \ p^{2-R}}{1 + \frac{p^{-1}}{p+1} \left( 1 - p^{1-R} + \eta_{p,4}^{\mathcal{A},\operatorname{e}}(\bm{a})p (1- p^{-1-R}) \right) \ - \ p^{-2-R}} & \text{if } \nu = 4.
\end{cases}
\end{align*}
\end{enumerate}

\item
Suppose that $p\| \prod_{j=1}^4a_j$ and 
\[
\bm{\alpha}\in \{(0,0,0,3), (0,0,2,3),(0,2,2,3),(2,2,2,3), (0,0,1,2),(0,1,2,2),(1,2,2,2)\}.
\]
 For $1 \leq \nu \leq 4$ and $m = p^R m'$ with $p\nmid m'$ we may evaluate $\omega_{\nu}(p)$ as follows. 
\noindent

\noindent
\begin{enumerate}
\item[(a)]  For $R = 0$, we have
\begin{align*}
\omega_\nu(p) = \begin{cases}
1 & \text{if } \nu = 1\text{ and }\bm{\alpha}=(0,0,0,3) \vspace{.1cm}\\
\frac{1-\eta_{p,2}^{\mathcal{A},\operatorname{e}}(\bm{a}\cdot \bm{d}^2)p^{-1} }{1+\eta_{p,3}^{\mathcal{A},\operatorname{e}}(\bm{a}) \chi_p(-m')p^{-1}}&\text{if } \nu = 1\text{ and }\bm{\alpha}=(0,0,1,2), \vspace{.1cm}\\
\frac{1-\eta_{p,2}^{\mathcal{A},\operatorname{e}}(\bm{a}\cdot \bm{d}^2)p^{-1} }{1+\eta_{p,3}^{\mathcal{A},\operatorname{e}}(\bm{a}) \chi_p(-m')p^{-1}}&  \text{if } \nu = 2 \text{ and }\bm{\alpha}=(0,0,2,3),\vspace{.1cm}\\
\frac{1+\eta_{p,1}^{\mathcal{A},\operatorname{e}}(\bm{a}\cdot \bm{d}^2)\varepsilon_p^3\chi_p(m')}{1+\eta_{p,3}^{\mathcal{A},\operatorname{e}}(\bm{a}) \chi_p(-m')p^{-1}}&  \text{if } \nu = 2 \text{ and }\bm{\alpha}=(0,1,2,2),\vspace{.1cm}\\
\frac{1+\eta_{p,1}^{\mathcal{A},\operatorname{e}}(\bm{a}\cdot \bm{d}^2)\varepsilon_p^3\chi_p(m')}{1+\eta_{p,3}^{\mathcal{A},\operatorname{e}}(\bm{a}) \chi_p(-m')p^{-1}}&  \text{if } \nu = 3 \text{ and }\bm{\alpha}=(0,2,2,3),\vspace{.1cm}\\
0&  \text{if } \nu = 3 \text{ and }\bm{\alpha}=(1,2,2,2),\vspace{.1cm}\\
0&\text{if }\nu=4.
\end{cases}
\end{align*}
\item[(b)]
For $R = 1$, we have
 \begin{align*}
\omega_\nu(p) = \begin{cases}
\frac{1-p^{-2}}{1+\eta_{p,4}^{\mathcal{A},\operatorname{o}}(\bm{a})\varepsilon_p\chi_{p}(-m')p^{-2}}&\text{if }\nu=1\text{ and }\bm{\alpha}=(0,0,0,3),\vspace{.1cm}\\
\frac{1+\eta_{p,2}^{\mathcal{A},\operatorname{e}}(\bm{a}\cdot\bm{d}^2)\left(1-p^{-1}\right)}{1+\eta_{p,4}^{\mathcal{A},\operatorname{o}}(\bm{a})\varepsilon_p\chi_{p}(-m')p^{-2}}&\text{if }\nu=1\text{ and }\bm{\alpha}=(0,0,1,2),\vspace{.1cm}\\
\frac{1+\eta_{p,2}^{\mathcal{A},\operatorname{e}}(\bm{a}\cdot \bm{d}^2)\left(1-p^{-1}\right)-p^{-1}}{1+\eta_{p,4}^{\mathcal{A},\operatorname{o}}(\bm{a})\varepsilon_p\chi_{p}(-m')p^{-2}}&\text{if }\nu=2\text{ and }\bm{\alpha}=(0,0,2,3),\vspace{.1cm}\\
\frac{1+\varepsilon_p^3\chi_p(m')\eta_{p,2}^{\mathcal{A},\operatorname{o}}(\bm{a}\cdot\bm{d}^2)}{1+\eta_{p,4}^{\mathcal{A},\operatorname{o}}(\bm{a})\varepsilon_p\chi_{p}(-m')p^{-2}}&\text{if }\nu=2\text{ and }\bm{\alpha}=(0,1,2,2),\vspace{.1cm}\\
0&\text{if }\nu=3\text{ and }\bm{\alpha}=(0,2,2,3),\vspace{.1cm}\\
\frac{1+\eta_{p,1}^{\mathcal{A},\operatorname{e}}(\bm{a}\cdot\bm{d}^2)\eta_p^3 \chi_{p}(m')p}{1+\eta_{p,4}^{\mathcal{A},\operatorname{o}}(\bm{a})\varepsilon_p\chi_{p}(-m')p^{-2}}&\text{if }\nu=3\text{ and }\bm{\alpha}=(1,2,2,2),\vspace{.1cm}\\
0&\text{if }\nu=4.
\end{cases}
\end{align*} 
\item[(c)] For $R=2$ and $\alpha_4=3$, we have 
\[
\omega_\nu(p) = \begin{cases}
\frac{1+\eta_{p,3}^{\mathcal{A},\operatorname{o}}(\bm{a}\cdot \bm{d}^2) p^{-1}\left(1-p^{-1}\right) +\eta_{p,3}^{\mathcal{A},\operatorname{e}}\left(\bm{a}\cdot\bm{d}^2\right)\chi_p(-m')p^{-2} }{1+\eta_{p,4}^{\mathcal{A},\operatorname{e}}(\bm{a})\varepsilon_p \chi_p(-m')p^{-3 }}&\text{if }\nu=1 \text{ (i.e., $\bm{\alpha}=(0,0,0,3)$)},\\
\frac{1+\left(\eta_{p,2}^{\mathcal{A},\operatorname{o}}\left(\bm{a}\cdot\bm{d}^2\right)+\eta_{p,2}^{\mathcal{A},\operatorname{e}}\left(\bm{a}\cdot\bm{d}^2\right)\right)\left(1-p^{-1}\right)+\eta_{p,3}^{\mathcal{A},\operatorname{e}}\left(\bm{a}\cdot\bm{d}^2\right)\chi_p(-m')p^{-1} }{1+\eta_{p,4}^{\mathcal{A},\operatorname{e}}(\bm{a})\varepsilon_p \chi_p(-m')p^{-3 }}&\text{if }\nu=2 \text{ (i.e., $\bm{\alpha}=(0,0,2,3)$)},\\
\frac{p+\eta_{p,3}^{\mathcal{A},\operatorname{e}}\left(\bm{a}\cdot\bm{d}^2\right)\chi_p(-m') }{1+\eta_{p,4}^{\mathcal{A},\operatorname{e}}(\bm{a})\varepsilon_p \chi_p(-m')p^{-3 }}&\text{if }\nu=3 \text{ (i.e., $\bm{\alpha}=(0,2,2,3)$)},\\
\frac{p^2+\eta_{p,3}^{\mathcal{A},\operatorname{e}}\left(\bm{a}\cdot\bm{d}^2\right)\chi_p(-m')p }{1+\eta_{p,4}^{\mathcal{A},\operatorname{e}}(\bm{a})\varepsilon_p \chi_p(-m')p^{-3 }}&\text{if }\nu=4 \text{ (i.e., $\bm{\alpha}=(2,2,2,3)$)}.
\end{cases}
\]
\item[(d)] 
For $R \geq 2$ and $R\geq \alpha_4$ even, we have
\begin{align*}
\omega_\nu(p) = \begin{cases}
\frac{1+\eta_{p,3}^{\mathcal{A},\operatorname{o}}(\bm{a}\cdot \bm{d}^2)p^{-1}\left(1-p^{-1}\right)+\eta_{p,4}^{\mathcal{A},\operatorname{e}}p^{-R}}{1+ \eta_{p,4}^{\mathcal{A},\operatorname{e}}(\bm{a}) p^{-R-1}\varepsilon_p\chi_{p}(-m')}& \text{if } \nu = 1\text{ and }\bm{\alpha}=(0,0,0,3), \vspace{.1cm}\\ 
\frac{1+\eta_{p,2}^{\mathcal{A},\operatorname{e}}\left(\bm{a}\cdot\bm{d}^2\right)\left(1-p^{-1}\right)+\eta_{p,4}^{\mathcal{A},\operatorname{e}}\left(\bm{a}\cdot\bm{d}^2\right)\varepsilon_p\chi_{p}(-m') p^{-R}}{1+ \eta_{p,4}^{\mathcal{A},\operatorname{e}}(\bm{a}) p^{-R-1}\varepsilon_p\chi_{p}(-m')}& \text{if } \nu = 1\text{ and }\bm{\alpha}=(0,0,1,2), \vspace{.1cm}\\ 
\frac{1+\left(\eta_{p,2}^{\mathcal{A},\operatorname{o}}\left(\bm{a}\cdot\bm{d}^2\right)+\eta_{p,2}^{\mathcal{A},\operatorname{e}}\left(\bm{a}\cdot\bm{d}^2\right)\right)\left(1-\frac{1}{p}\right) +\eta_{p,4}^{\mathcal{A},\operatorname{e}}p^{1-R}}{1+ \eta_{p,4}^{\mathcal{A},\operatorname{e}}(\bm{a}) p^{-R-1}\varepsilon_p\chi_{p}(-m')}
 & \text{if } \nu = 2\text{ and }\bm{\alpha}=(0,0,2,3), \vspace{.1cm}\\
\frac{1+\eta_{p,4}^{\mathcal{A},\operatorname{e}}\left(\bm{a}\cdot\bm{d}^2\right) \varepsilon_p\chi_p(-m')p^{1-R}}{1+ \eta_{p,4}^{\mathcal{A},\operatorname{e}}(\bm{a}) p^{-R-1}\varepsilon_p\chi_{p}(-m')}
 & \text{if } \nu = 2\text{ and }\bm{\alpha}=(0,1,2,2), \vspace{.1cm}\\
\frac{p+\eta_{p,4}^{\mathcal{A},\operatorname{e}}\left(\bm{a}\cdot\bm{d}^2\right)\varepsilon_p\chi_p(-m')p^{2-R}}{1+ \eta_{p,4}^{\mathcal{A},\operatorname{e}}(\bm{a}) p^{-R-1}\varepsilon_p\chi_{p}(-m')}
 & \text{if } \nu = 3\text{ and }\bm{\alpha}=(0,2,2,3), \vspace{.1cm}\\
\frac{p+\eta_{p,4}^{\mathcal{A},\operatorname{e}}\left(\bm{a}\cdot\bm{d}^2\right)\varepsilon_p\chi_{p}(-m')p^{2-R}}{1+ \eta_{p,4}^{\mathcal{A},\operatorname{e}}(\bm{a}) p^{-R-1}\varepsilon_p\chi_{p}(-m')}
 & \text{if } \nu = 3\text{ and }\bm{\alpha}=(1,2,2,2), \vspace{.1cm}\\
\frac{p^2+\eta_{p,4}^{\mathcal{A},\operatorname{e}}\left(\bm{a}\cdot\bm{d}^2\right)\varepsilon_p\chi_{p}(-m')p^{3-R}}{1+ \eta_{p,4}^{\mathcal{A},\operatorname{e}}(\bm{a}) p^{-R-1}\varepsilon_p\chi_{p}(-m')}
& \text{if } \nu = 4.
\end{cases}
\end{align*}

\item[(e)] For $R \geq 3$ and $R$ odd, we have
\begin{align*}
\omega_\nu(p) = \begin{cases}
\frac{1+ \eta_{p,3}^{\mathcal{A},\operatorname{o}}\left(\bm{a}\cdot\bm{d}^2\right)p^{-1}\left(1-p^{-1}\right)+\eta_{p,4}^{\mathcal{A},\operatorname{o}}(\bm{a}\cdot\bm{d}^2) p^{-R}\varepsilon_p\chi_{p}(-m') }{1+ \eta_{p,4}^{\mathcal{A},\operatorname{o}}(\bm{a}) p^{-R-1}\varepsilon_p\chi_{p}(-m')}
& \text{if } \nu = 1 \text{ and }\bm{\alpha}=(0,0,0,3),\vspace{.1cm}\\ 
\frac{1+ \eta_{p,2}^{\mathcal{A},\operatorname{e}}\left(\bm{a}\cdot\bm{d}^2\right)\left(1-p^{-1}\right) +\eta_{p,4}^{\mathcal{A},\operatorname{o}}(\bm{a}\cdot\bm{d}^2) p^{-R}\varepsilon_p\chi_{p}(-m') }{1+ \eta_{p,4}^{\mathcal{A},\operatorname{o}}(\bm{a}) p^{-R-1}\varepsilon_p\chi_{p}(-m')}
& \text{if } \nu = 1 \text{ and }\bm{\alpha}=(0,0,1,2),\vspace{.1cm}\\ 
\frac{ 1+ \left(\eta_{p,2}^{\mathcal{A},\operatorname{o}}\left(\bm{a}\cdot\bm{d}^2\right)+\eta_{p,2}^{\mathcal{A},\operatorname{e}}\left(\bm{a}\cdot\bm{d}^2\right)\right)\left(1-p^{-1}\right) +\eta_{p,4}^{\mathcal{A},\operatorname{o}}(\bm{a}\cdot\bm{d}^2) p^{1-R}\varepsilon_p\chi_{p}(-m') }{1+ \eta_{p,4}^{\mathcal{A},\operatorname{o}}(\bm{a}) p^{-R-1}\varepsilon_p\chi_{p}(-m')}
& \text{if } \nu = 2\text{ and }\bm{\alpha}=(0,0,2,3), \vspace{.1cm}\\
\frac{1+\eta_{p,4}^{\mathcal{A},\operatorname{o}}(\bm{a}) p^{1-R}\varepsilon_p\chi_{p}(-m') }{1+ \eta_{p,4}^{\mathcal{A},\operatorname{o}}(\bm{a}) p^{-R-1}\varepsilon_p\chi_{p}(-m')}
& \text{if } \nu = 2\text{ and }\bm{\alpha}=(0,1,2,2), \vspace{.1cm}\\
\frac{p+  \eta_{p,4}^{\mathcal{A},\operatorname{o}}(\bm{a}) p^{2-R}\varepsilon_p\chi_{p}(-m') }{1+ \eta_{p,4}^{\mathcal{A},\operatorname{o}}(\bm{a}) p^{-R-1}\varepsilon_p\chi_{p}(-m')}
& \text{if } \nu = 3\text{ and }\bm{\alpha}=(0,2,2,3), \vspace{.1cm}\\
\frac{ p+  \eta_{p,4}^{\mathcal{A},\operatorname{o}}(\bm{a}) p^{2-R}\varepsilon_p\chi_{p}(-m')}{1+ \eta_{p,4}^{\mathcal{A},\operatorname{o}}(\bm{a}) p^{-R-1}\varepsilon_p\chi_{p}(-m')}
& \text{if } \nu = 3\text{ and }\bm{\alpha}=(1,2,2,2), \vspace{.1cm}\\
\frac{p^2+  \eta_{p,4}^{\mathcal{A},\operatorname{o}}(\bm{a}) p^{3-R}\varepsilon_p\chi_{p}(-m')}{1+ \eta_{p,4}^{\mathcal{A},\operatorname{o}}(\bm{a}) p^{-R-1}\varepsilon_p\chi_{p}(-m')}
 & \text{if } \nu = 4.
\end{cases}
\end{align*}
\end{enumerate}

\end{enumerate}
\end{lemma}
\begin{proof}
(1) The statement is trivial up the evaluation of the local density at infinity.  To evaluate this, we compute the volume of $\left\{\bm{x}\in\R^4: \sum_{j=1}^4 a_jd_j^2x_j^2\leq m\right\}$. This is 
\begin{equation}\label{eqn:localinfty}
\vol\left(B_{Q_{\bm{d}},m}\right):=\int_{-\sqrt{\frac{m}{a_1d_1^2}}}^{\sqrt{\frac{m}{a_1d_1^2}}}\int_{-\sqrt{\frac{m-a_1d_1^2x_1^2}{a_2d_2^2}}}^{\sqrt{\frac{m-a_1d_1^2x_1^2}{a_2d_2^2}}}\int_{-\sqrt{\frac{m-\sum_{j=1}^2a_jd_j^2x_j^2}{a_3d_3^2}}}^{\sqrt{\frac{m-\sum_{j=1}^2a_jd_j^2x_j^2}{a_3d_3^2}}}\int_{-\sqrt{\frac{m-\sum_{j=1}^3a_jd_j^2x_j^2}{a_4d_4^2}}}^{\sqrt{\frac{m-\sum_{j=1}^3a_jd_j^2x_j^2}{a_4d_4^2}}} dx_4dx_3dx_2dx_1.
\end{equation}
We then make the change of variables $x_j\mapsto \tfrac{x_j}{d_j}$ to yield
\[
\frac{1}{\prod_{j=1}^4d_j} \int_{-\sqrt{\frac{m}{a_1}}}^{\sqrt{\frac{m}{a_1}}}\int_{-\sqrt{\frac{m-a_1 x_1^2}{a_2}}}^{\sqrt{\frac{m-a_1 x_1^2}{a_2}}}\int_{-\sqrt{\frac{m-\sum_{j=1}^2a_jx_j^2}{a_3}}}^{\sqrt{\frac{m-\sum_{j=1}^2a_jx_j^2}{a_3}}}\int_{-\sqrt{\frac{m-\sum_{j=1}^3a_jx_j^2}{a_4}}}^{\sqrt{\frac{m-\sum_{j=1}^3a_jx_j^2}{a_4}}} dx_4dx_3dx_2dx_1=\frac{\vol\left(B_{Q_{\bm{1}},m}\right)}{d_1d_2d_3d_4}.
\]
Thus 
\[
\beta_{Q_{\bm{d}},\infty}(m)=\frac{\beta_{Q_{\bm{1}},\infty}(m)}{d_1d_2d_3d_4}.
\]
We next plug everything into \eqref{eqn:Siegellocal} to compare $a_{E_{Q_{\bm{d}}}}(m)$ and $a_{E_{Q_{\bm{1}}}}(m)$. Noting that $\beta_{Q_{\bm{d}},p}(m)=\beta_{Q_{\bm{1}},p}(m)$ for $p\nmid \prod_{j=1}^4 d_j$ (for finite primes), we evaluate
\[
\omega(\bm{d},m):=\frac{a_{E_{Q_{\bm{a}\cdot \bm{d}^2}}}(m)}{a_{E_{Q_{\bm{a}}}}(m)}=\prod_p \frac{\beta_{Q_{\bm{a}\cdot \bm{d}^2},p}(m)}{\beta_{Q_{\bm{a}},p}(m)}=\frac{1}{d_1d_2d_3d_4} \prod_{\substack{p^j\|d_1d_2d_3d_4\\ j\geq 1}}\frac{\beta_{Q_{\bm{a}\cdot \bm{d}^2},p}(m)}{\beta_{Q_{\bm{a}},p}(m)}.
\]
Rearranging and noting that $a_{E_{Q_{\bm{1}}}}(m)=r_{Q_{\bm{a}},1}(m)$ because $Q_{\bm{a}}$ has class number one, we obtain the claim.
\vspace{.05in}

\noindent
(2) These follows directly by plugging in Lemma \ref{lem:LocalDensitySimplify} (1).
\vspace{.05in}

\noindent
(3)  These follows directly by plugging in Lemma \ref{lem:LocalDensitySimplify} (2)--(5).

\end{proof}

Set 
\[
\mathcal{S}_{k}:=\{1\leq j\leq k: a_j'\equiv 1\pmod{4}\}
\]
For the prime $p=2$, we use the following lemma.
\begin{lemma}\label{lem:localdensityp=2}
\noindent

\noindent
\begin{enumerate}[leftmargin=*,label={\rm(\arabic*)}]
\item  Suppose that $0\leq \alpha_1\leq 1$, $\alpha_1\leq \alpha_2\leq 2$, $\alpha_2\leq \alpha_3\leq 3$, and $\alpha_3\leq \alpha_4\leq 4$ and $\alpha_{j+1}-\alpha_j\leq 2$ for $1\leq j\leq 2$, while $\alpha_4-\alpha_3\leq 3$. Then 
\begin{multline*}
\beta_{Q_{\bm{a}},2}(m)=1+ \delta_{\alpha_1=1}(-1)^m + \delta_{\alpha_2=2}\delta_{\alpha_1=0}\times \begin{cases} 1&\text{if }a_1\equiv m\pmod{4}\text{ or }4\mid m,\\ -1&\text{if }a_1\equiv -m\pmod{4}\text{ or }m\equiv 2\pmod{4},\end{cases}\\
+\delta_{(\alpha_1,\alpha_2)\neq (0,1)} \delta_{\alpha_3=\alpha_2+2}\delta_{R\geq \alpha_2}\delta_{a_1'\not\equiv a_2'(\bmod{4})}\delta_{2\mid \mathcal{A}_2}  2^{\frac{\mathcal{A}_2}{2}}+\delta_{(\alpha_1,\alpha_2,\alpha_3)=(0,1,3)}\delta_{2\nmid \mathcal{A}_2}\left(\frac{(-1)^{\frac{a_1'+a_2'}{2}}8}{a_1'm}\right)\\
+2^{\frac{\mathcal{A}_2}{2}-1}\delta_{2\mid\mathcal{A}_2}\delta_{\alpha_4\geq \alpha_3+2}\left(  \delta_{R\geq \alpha_3+2} - \delta_{R=\alpha_3+1} (-1)^{\delta_{3\mid\#\mathcal{S}_3}}  + \delta_{R=\alpha_3} (-1)^{\#\mathcal{S}_3}\left(\frac{-4}{m'}\right)\right)\\
+2^{\frac{\mathcal{A}_2-3}{2}}\delta_{2\nmid\mathcal{A}_2}\delta_{\alpha_4=\alpha_3+3}\left(  \delta_{R\geq \alpha_4} - \delta_{R=\alpha_4-1} (-1)^{\delta_{3\mid\#\mathcal{S}_3}}  + \delta_{R=\alpha_4-2} (-1)^{\#\mathcal{S}_3} \left(\frac{-4}{m'}\right)\right)\\
+ \delta_{R\equiv \mathcal{A}_3(\bmod{2})} \delta_{\alpha_3-1\leq R\leq \alpha_4-3} 2^{\frac{\mathcal{A}_3-R}{2}+4}\prod_{j=1}^3\left(\frac{2}{a_j'}\right)\left(\frac{-4}{\frac{m'-1}{2}}\right)\\
-\delta_{2\mid\mathcal{A}_4} \sum_{k=\alpha_4+2}^{R} 2^{\frac{\mathcal{A}_4}{2}-k+1}\left(\frac{-4}{\#\mathcal{S}_4+1}\right)+ \delta_{2\mid\mathcal{A}_4} 2^{\frac{\mathcal{A}_4}{2}-R}\delta_{R\geq \alpha_4+1}\left(\frac{-4}{\#\mathcal{S}_4+1}\right) \\
- \delta_{2\mid\mathcal{A}_4} \delta_{R\geq \alpha_4}\delta_{2\nmid \#\mathcal{S}_4} 2^{\frac{\mathcal{A}_4}{2}-R-1} (-1)^{\frac{m'-\#\mathcal{S}_4}{2}}\\
-\delta_{R\geq \alpha_4-1} \delta_{2\nmid \mathcal{A}_4} 2^{\frac{\mathcal{A}_4-1}{2}-R-2} \left(\frac{8}{m'}\right) \prod_{j=1}^4\left(\frac{8}{a_j'}\right)\left( \left(\frac{-4}{\#\mathcal{S}_4+1}\right)+\left(\frac{-4}{m' \#\mathcal{S}_4}\right) \right).
\end{multline*}
\item
Suppose that $\alpha_1=0$,
\begin{multline*}
(\alpha_2,\alpha_3,\alpha_4)\in\{(0,0,0),(0,0,1), (0,0,2),(0,1,1),(0,1,2), (0,1,3),\\
(1,1,1),(1,1,2),(1,2,2),(1,2,3)\},
\end{multline*}
and $0\leq \ord_2(m)=:R\leq 2$. Then 
\[
\beta_{Q_{\bm{a}},2}(m)\geq \frac{1}{2}.
\]
\item
Suppose that $\alpha_1=0$,
\begin{multline*}
(\alpha_2,\alpha_3,\alpha_4)\in\{(0,0,0),(0,0,1), (0,0,2),(0,1,1),(0,1,2), (0,1,3),\\
(1,1,1),(1,1,2),(1,2,2),(1,2,3)\},
\end{multline*}
and $0\leq \ord_2(m)=:R\leq 2$. Then 
\[
\beta_{Q_{\bm{a}},2}(m)\leq \frac{3}{2}.
\]

\end{enumerate}
\end{lemma}
\begin{proof}

We assume that $0\leq \alpha_1\leq 1$, $\alpha_1\leq \alpha_2\leq 2$, $\alpha_2\leq \alpha_3\leq 3$, and $\alpha_3\leq \alpha_4\leq 4$ and $\alpha_{j+1}-\alpha_j\leq 2$ for $1\leq j\leq 2$, while $\alpha_4-\alpha_3\leq 3$. Noting that for $a$ odd one has
\[
G_2(a,0,2)=0
\]
and $0\leq \alpha_1\leq 1$, then \eqref{eqn:RQp^r} becomes 
\begin{multline}\label{eqn:R2powgen}
2^{-3r}R_{Q,2^r}(m)= 1+ \delta_{\alpha_1=1}(-1)^m + \sum_{k=\alpha_1+2}^{\alpha_2}2^{\alpha_1-k}\sum_{n'\in(\Z/2^k\Z)^{\times}} e^{-\frac{2\pi i n'm}{2^k}}G_2\left(n'a_1',0,2^{k-\alpha_1}\right)\\
+ \sum_{k=\alpha_2+2}^{\alpha_3}2^{\mathcal{A}_2-2k}\sum_{n'\in(\Z/2^k\Z)^{\times}} e^{-\frac{2\pi i n'm}{2^k}}\prod_{j=1}^2G_2\left(n'a_j',0,2^{k-\alpha_j}\right)\\
+ \sum_{k=\alpha_3+2}^{\alpha_4}2^{\mathcal{A}_3-3k}\sum_{n'\in(\Z/2^k\Z)^{\times}} e^{-\frac{2\pi i n'm}{2^k}}\prod_{j=1}^3G_2\left(n'a_j',0,2^{k-\alpha_j}\right)\\
+ \sum_{k=\alpha_4+2}^{r} 2^{\mathcal{A}_4-4k}\sum_{n'\in(\Z/2^k\Z)^{\times}} e^{-\frac{2\pi i n'm}{2^k}} \prod_{j=1}^4G_2\left(n'a_j',0,2^{k-\alpha_j}\right)\\
=1+ \delta_{\alpha_1=1}(-1)^m +\frac{1}{4} \delta_{\alpha_2=\alpha_1+2} \sum_{n'\in(\Z/2^{\alpha_2}\Z)^{\times}} e^{-\frac{2\pi i n'm}{2^{\alpha_2}}}G_2\left(n'a_1',0,4\right)\\
+ \delta_{\alpha_3=\alpha_2+2} 2^{\mathcal{A}_2-2\alpha_3}\sum_{n'\in(\Z/2^{\alpha_3}\Z)^{\times}} e^{-\frac{2\pi i n'm}{2^{\alpha_3}}}\prod_{j=1}^2G_2\left(n'a_j',0,2^{\alpha_3-\alpha_j}\right)\\
+ \sum_{k=\alpha_3+2}^{\alpha_4}2^{\mathcal{A}_3-3k}\sum_{n'\in(\Z/2^k\Z)^{\times}} e^{-\frac{2\pi i n'm}{2^k}}\prod_{j=1}^3G_2\left(n'a_j',0,2^{k-\alpha_j}\right)\\
+ \sum_{k=\alpha_4+2}^{r} 2^{\mathcal{A}_4-4k}\sum_{n'\in(\Z/2^k\Z)^{\times}} e^{-\frac{2\pi i n'm}{2^k}} \prod_{j=1}^4G_2\left(n'a_j',0,2^{k-\alpha_j}\right).
\end{multline}
Note that due to our restrictions on $\alpha_j$, we have $\alpha_2=\alpha_1+2$ if and only if $\alpha_2=2$ and $\alpha_1=0$.  We next use the formula (valid for $4\mid c$)
\[
G_2(a,0,c)=(1+i)\varepsilon_a^{-1} \sqrt{c}\left(\frac{c}{a}\right)
\]
to simplify \eqref{eqn:R2powgen} as 
\begin{multline*}
1+ \delta_{\alpha_1=1}(-1)^m +\frac{1+i}{2} \delta_{\alpha_2=2}\delta_{\alpha_1=0} \sum_{n'\in(\Z/4\Z)^{\times}} e^{-\frac{2\pi i n'm}{4}}\varepsilon_{n'a_1'}^{-1}\\
+(1+i)^2 \delta_{\alpha_3=\alpha_2+2} 2^{\frac{\mathcal{A}_2}{2}-\alpha_3 }\left(\frac{2^{\alpha_2-\alpha_1}}{a_1'}\right)\sum_{n'\in(\Z/2^{\alpha_3}\Z)^{\times}} e^{-\frac{2\pi i n'm}{2^{\alpha_3}}}\left(\frac{2^{\alpha_2-\alpha_1}}{n'}\right) \prod_{j=1}^2\varepsilon_{n'a_j'}^{-1}\\
+(1+i)^3 \sum_{k=\alpha_3+2}^{\alpha_4}2^{\frac{\mathcal{A}_3-3k}{2}}\prod_{j=1}^3\left(\frac{2^{k-\alpha_j}}{a_j'}\right)\sum_{n'\in(\Z/2^k\Z)^{\times}} e^{-\frac{2\pi i n'm}{2^k}}\prod_{j=1}^3\varepsilon_{n'a_j'}^{-1}\left(\frac{2^{k-\alpha_j}}{n'}\right)\\
+(1+i)^4 \sum_{k=\alpha_4+2}^{r}2^{\frac{\mathcal{A}_4}{2}-2k}\prod_{j=1}^4\left(\frac{2^{k-\alpha_j}}{a_j'}\right)\sum_{n'\in(\Z/2^k\Z)^{\times}} e^{-\frac{2\pi i n'm}{2^k}}\prod_{j=1}^4\varepsilon_{n'a_j'}^{-1}\left(\frac{2^{k-\alpha_j}}{n'}\right).
\end{multline*}
We then simplify the last Jacobi symbol with 
\[
\left(\frac{2^{4k-\mathcal{A}_4}}{n'}\right) = \left(\frac{2^{4k+\mathcal{A}_4}}{n'}\right) = \left(\frac{2^{\mathcal{A}_4}}{n'}\right) = \left(\frac{2}{n'}\right)^{\mathcal{A}_4}=\left(\frac{8}{n'}\right)^{\mathcal{A}_4},
\]
while the Jacobi symbol in the second-to-last term is simplified as 
\[
\left(\frac{2^{3k-\mathcal{A}_3}}{n'}\right) = \left(\frac{2^{3k+\mathcal{A}_3}}{n'}\right) = \left(\frac{8}{n'}\right)^k \left(\frac{8}{n'}\right)^{\mathcal{A}_3}.
\]
This yields (plugging in $(1+i)^2=2i$, $(1+i)^3=2(1-i)$, and $(1+i)^4=-4$, 
\begin{multline}\label{eqn:local2expand}
1+ \delta_{\alpha_1=1}(-1)^m +\frac{1+i}{2} \delta_{\alpha_2=2}\delta_{\alpha_1=0} \sum_{n'\in(\Z/4\Z)^{\times}} e^{-\frac{2\pi i n'm}{4}}\varepsilon_{n'a_1'}^{-1}\\
+2i \delta_{\alpha_3=\alpha_2+2} 2^{\frac{\mathcal{A}_2}{2}-\alpha_3 }\left(\frac{8}{a_1'}\right)^{\mathcal{A}_2}\sum_{n'\in(\Z/2^{\alpha_3}\Z)^{\times}} e^{-\frac{2\pi i n'm}{2^{\alpha_3}}}\left(\frac{8}{n'}\right)^{\mathcal{A}_2} \prod_{j=1}^2\varepsilon_{n'a_j'}^{-1}\\
+2(i-1) \sum_{k=\alpha_3+2}^{\alpha_4}2^{\frac{\mathcal{A}_3-3k}{2}}\left(\frac{8}{a_j'}\right)^{k+\mathcal{A}_3}\sum_{n'\in(\Z/2^k\Z)^{\times}} e^{-\frac{2\pi i n'm}{2^k}}\prod_{j=1}^3\varepsilon_{n'a_j'}^{-1}\left(\frac{8}{n'}\right)^{k+\mathcal{A}_3}\\
-4 \sum_{k=\alpha_4+2}^{r}2^{\frac{\mathcal{A}_4}{2}-2k}\prod_{j=1}^4\left(\frac{8}{a_j'}\right)^{\mathcal{A}_4}\sum_{n'\in(\Z/2^k\Z)^{\times}} e^{-\frac{2\pi i n'm}{2^k}}\prod_{j=1}^4\varepsilon_{n'a_j'}^{-1}\left(\frac{8}{n'}\right)^{\mathcal{A}_4}.
\end{multline}
We next evaluate the inner sum on $n'$ in each case. Using our restrictions on the $\alpha_j$, the sum over $n'$ in the third, fourth, and fifth terms have length at most $8$ and may be directly computed on a case-by-case basis. The third term becomes (noting that $a_1'=a_1$ if $\delta_{\alpha_1=0}\neq 0$)
\[
\delta_{\alpha_2=2}\delta_{\alpha_1=0}\times \begin{cases} 1&\text{if }a_1\equiv m\pmod{4}\text{ or }4\mid m,\\ -1&\text{if }a_1\equiv -m\pmod{4}\text{ or }m\equiv 2\pmod{4}.\end{cases}
\]
The fourth term in \eqref{eqn:local2expand} simplifies as
\begin{equation}\label{eqn:fourthterm}
\delta_{(\alpha_1,\alpha_2)\neq (0,1)}\delta_{\alpha_3=\alpha_2+2}\delta_{R\geq \alpha_2}\delta_{a_1'\not\equiv a_2'(\bmod{4})}\delta_{2\mid \mathcal{A}_2}  2^{\frac{\mathcal{A}_2}{2}}+\delta_{(\alpha_1,\alpha_2,\alpha_3)=(0,1,3)}\delta_{2\nmid \mathcal{A}_2}\left(\frac{(-1)^{\frac{a_1'+a_2'}{2}}8}{a_1'm}\right).
\end{equation}
After a long simplification, the fifth term in \eqref{eqn:local2expand} becomes 
\begin{multline*}
2^{\frac{\mathcal{A}_2}{2}-1}\delta_{2\mid\mathcal{A}_2}\delta_{\alpha_4\geq \alpha_3+2}\left(  \delta_{R\geq \alpha_3+2} - \delta_{R=\alpha_3+1} (-1)^{\delta_{3\mid\#\mathcal{S}_3}}  + \delta_{R=\alpha_3} (-1)^{\#\mathcal{S}_3}\left(\frac{-4}{m'}\right)\right)\\
+2^{\frac{\mathcal{A}_2-3}{2}}\delta_{2\nmid\mathcal{A}_2}\delta_{\alpha_4=\alpha_3+3}\left(  \delta_{R\geq \alpha_4} - \delta_{R=\alpha_4-1} (-1)^{\delta_{3\mid\#\mathcal{S}_3}}  + \delta_{R=\alpha_4-2} (-1)^{\#\mathcal{S}_3} \left(\frac{-4}{m'}\right)\right)\\
+ \delta_{R\equiv \mathcal{A}_3(\bmod{2})} \delta_{\alpha_3-1\leq R\leq \alpha_4-3} 2^{\frac{\mathcal{A}_3-R}{2}+4}\prod_{j=1}^3\left(\frac{2}{a_j'}\right)\left(\frac{-4}{\frac{m'-1}{2}}\right).
\end{multline*}

It remains to compute the sixth term. This is 
\[
-4 \sum_{k=\alpha_4+2}^{r}2^{\frac{\mathcal{A}_4}{2}-2k}\prod_{j=1}^4\left(\frac{8}{a_j'}\right)^{\mathcal{A}_4}\sum_{n'\in(\Z/2^k\Z)^{\times}} e^{-\frac{2\pi i n'm}{2^k}}\prod_{j=1}^4\varepsilon_{n'a_j'}^{-1}\left(\frac{8}{n'}\right)^{\mathcal{A}_4}.
\]
If $\mathcal{A}_4$ is even, then the above equals (making the change of variables $n'\mapsto n_0+4n'$)
\[
-4 \sum_{k=\alpha_4+2}^{r}2^{\frac{\mathcal{A}_4}{2}-2k}\sum_{n_0\in\{\pm 1\}}\prod_{j=1}^4\varepsilon_{n_0a_j'}^{-1} e^{-\frac{2\pi i n_0m}{2^{k}}}\sum_{n'\in(\Z/2^{k-2}\Z)^{\times}} e^{-\frac{2\pi i n'm}{2^{k-2}}}.
\]
The inner sum vanishes unless $R\geq k-2$, giving
\begin{equation}\label{eqn:sixthterm}
- \sum_{k=\alpha_4+2}^{r}\delta_{R\geq k-2}2^{\frac{\mathcal{A}_4}{2}-k}\sum_{n_0\in\{\pm 1\}}\prod_{j=1}^4\varepsilon_{n_0a_j'}^{-1}e^{-\frac{2\pi i n_0m}{2^{k}}}.
\end{equation}

For $k\leq R$ we have $e^{-\frac{2\pi i n_0 m}{2^{k}}}=1$ and a short calculation (splitting depending on $\#\mathcal{S}_4$) shows that the inner sum in \eqref{eqn:sixthterm} evaluates as 
\[
2\left(\frac{-4}{\#\mathcal{S}_4+1}\right)
\]
Hence for $r$ sufficiently large \eqref{eqn:sixthterm} equals 
\begin{equation}\label{eqn:2powtosimplify4}
-\sum_{k=\alpha_4+2}^{R} 2^{\frac{\mathcal{A}_4}{2}-k+1}\left(\frac{-4}{\#\mathcal{S}_4+1}\right)-  \sum_{k=R+1}^{R+2}\delta_{k\geq \alpha_4+2}2^{\frac{\mathcal{A}_4}{2}-k}\sum_{n_0\in\{\pm 1\}}\prod_{j=1}^4\varepsilon_{n_0a_j'}^{-1}e^{-\frac{2\pi i n_0m}{2^{k}}}.
\end{equation}
For $k=R+1$ (assuming that $R\geq \alpha_4+1$) the exponential is $-1$ and simplifying we see that this term equals 
\[
2^{\frac{\mathcal{A}_4}{2}-R}\delta_{R\geq \alpha_4+1}\left(\frac{-4}{\#\mathcal{S}_4+1}\right).
\]
For $k=R+2$ (assuming that $R\geq \alpha_4$), we again split into cases depending on $\#\mathcal{S}_{4}$ and evaluate the inner sum as
\[
 (-1)^{\frac{m'-\#\mathcal{S}_4}{2}}\delta_{2\nmid \#\mathcal{S}_4} 2.
\]
We conclude that \eqref{eqn:2powtosimplify4} becomes 
\begin{multline}\label{eqn:2powsimple4}
-\sum_{k=\alpha_4+2}^{R} 2^{\frac{\mathcal{A}_4}{2}-k+1}\left(\frac{-4}{\#\mathcal{S}_4+1}\right)+ 2^{\frac{\mathcal{A}_4}{2}-R}\delta_{R\geq \alpha_4+1}\left(\frac{-4}{\#\mathcal{S}_4+1}\right) \\
- \delta_{R\geq \alpha_4}\delta_{2\nmid \#\mathcal{S}_4} 2^{\frac{\mathcal{A}_4}{2}-R-1} (-1)^{\frac{m'-\#\mathcal{S}_4}{2}}.
\end{multline}
Finally, for $\mathcal{A}_4$ odd, the sixth term in \eqref{eqn:local2expand} equals (making the change of variables $n'\mapsto n_0+8n'$ with $n_0\in \{\pm 1,\pm 3\}$)
\[
-4\prod_{j=1}^4\left(\frac{8}{a_j'}\right)\ \sum_{k=\alpha_4+2}^{r}2^{\frac{\mathcal{A}_4}{2}-2k}\sum_{n_0\in\{\pm 1,\pm 3\}}e^{-\frac{2\pi i n_0m}{2^k}}\prod_{j=1}^4\varepsilon_{n_0a_j'}^{-1}\left(\frac{8}{n_0}\right) \sum_{n'\in(\Z/2^{k-3}\Z)^{\times}} e^{-\frac{2\pi i n'm}{2^{k-3}}}.
\]
The inner sum vanishes unless $R\geq k-3$, yielding (for $r$ sufficiently large)
\begin{equation}\label{eqn:sixthterm2}
- \prod_{j=1}^4\left(\frac{8}{a_j'}\right) \sum_{k=\alpha_4+2}^{R+3} 2^{\frac{\mathcal{A}_4}{2}-k-1}\sum_{n_0\in\{\pm 1,\pm 3\}}e^{-\frac{2\pi i n_0m}{2^k}}\prod_{j=1}^4\varepsilon_{n_0a_j'}^{-1}\left(\frac{8}{n_0}\right).
\end{equation}
For $k\leq R+2$, the exponential only depends on $n_0$ modulo $4$, so the terms $n_0$ and $n_0+4$ (modulo $8$) cancel due to $\left(\frac{8}{n_0+4}\right)=-\left(\frac{8}{n_0}\right)$. Hence the sum is zero in these cases. 

Only the $k=R+3$ term remains. Therefore \eqref{eqn:sixthterm2} equals
\begin{equation}\label{eqn:sixthtermk=R+3}
- \prod_{j=1}^4\left(\frac{8}{a_j'}\right)\delta_{R\geq \alpha_4-1}  2^{\frac{\mathcal{A}_4}{2}-R-4}\sum_{n_0\in\{\pm 1,\pm 3\}}e^{-\frac{2\pi i n_0m'}{8}}\prod_{j=1}^4\varepsilon_{n_0a_j'}^{-1}\left(\frac{8}{n_0}\right).
\end{equation}
Splitting depending on $\#\mathcal{S}_4$ and evaluating the finite sum over $n_0$ yields the claim.

Parts (2)--(3) directly follow by splitting into different cases depending on $R$, evaluating part (1) on a case-by-case basis, and bounding all characters trivially.
\end{proof}

 We define another function
\begin{equation*}
\Omega(p):= \sum_{\substack{\bm{d}\in \S^4\\ \prod_{j=1}^4d_j=p}} \omega_{1,\bm{a}\cdot\bm{d}^2}(p) - \sum_{\substack{\bm{d}\in\S^4\\ \prod_{j=1}^4d_j=p^2}}\frac{ \omega_{2,\bm{a}\cdot\bm{d}^2}(p)}{p} + \sum_{\substack{\bm{d}\in\S^4\\ \prod_{j=1}^4 d_j=p^3}} \frac{ \omega_{3,\bm{a}\cdot\bm{d}^2}(p)}{p^2} - \frac{\omega_{4,\bm{a}\cdot\bm{p}^2}(p)}{p^3}
\end{equation*}
and $$\Omega(\ell) := \prod_{p\mid \ell} \Omega(p).$$
In the sieve theory calculation, we will obtain the main term which contain the following term : 
\[
W(z_0) = \prod_{p < z_0} \left(1- \frac{\Omega(p)}{p}\right).
\]
Hence we finally need the lower bound of $W(z_0)$ i.e, upper bound of $\Omega(p)$. In order to obtain such a bound, we next prove bounds on $\omega_{\nu}(p)$.
\begin{lemma}\label{lem:LambdapUpper}
Let $\bm{a}\in\N^4$ and a prime $p\neq 2$ be given.
\noindent

\noindent
\begin{enumerate} [leftmargin=*,label={\rm(\arabic*)}]
\item Suppose that $p\nmid \prod_{j=1}^r a_j$. 
\begin{enumerate} [leftmargin=*]
\item
For $R=0$ we have 
\[
\Omega(p)\leq   4+\frac{-2p^{-1}+18p^{-2}}{ 1-p^{-2}}.
\]
In particular, for $p\geq 11$ we have $\Omega(p)\leq 4$, $\Omega(7)\leq \frac{33}{8}$, and $\Omega(5)\leq \frac{9}{2}$. 

\item
For $R=1$ we have 
\[
\Omega(p)\leq  4+4\frac{p^{-1}-p^{-2}}{1+p^{-2}}\leq \frac{24}{5}.
\]

\item
For $R\geq 2$ with $R$ even, we have 
\[
\Omega(p)\leq\begin{cases} 4.93 &\text{if }p\geq 7,\\ \frac{1265}{273}&\text{if }p=5.\end{cases}
\]

\item
For $R\geq 3$ with $R$ odd, we have 
\[
\Omega(p)\leq \begin{cases} 4.93 &\text{if }p\geq 7,\\ \frac{301}{65}&\text{if }p=5.\end{cases}
\]
\end{enumerate}
\item Suppose that $p\| \prod_{j=1}^{4}a_j$. 
\noindent

\noindent
\begin{enumerate} [leftmargin=*]
\item For $R=0$ we have 
\[
\Omega(p)\leq 4+3p^{-1}+\frac{12p^{-2}}{1-p^{-1}}.
\]
In particular, for $p\geq 7$ we have $\Omega(p)<5$. 
\item For $R=1$ we have 
\[
\Omega(p)\leq 7-2p^{-1}+\frac{7p^{-2}+2p^{-3}}{1-p^{-2}}.
\]
In particular, for $p> 7$ we have $\Omega(p)<7$ and $\Omega(7)\leq 6.87$.  

\item For $R=2$ we have.
\[
\Omega(p)\leq 7-5p^{-1}+14p^{-2} +\frac{7p^{-3}-5p^{-4}+14p^{-5}}{1-p^{-3}}.
\]
In particular, for $p\geq 7$ we have $\Omega(p)\leq 7-5p^{-1}+14p^{-2}+7p^{-3}$. Moreover, for $p>7$ we have $\Omega(p)<7$ and $\Omega(7)<6.6$. 

\item For $R\geq 3$ with $R$ odd we have 
\[
\Omega(p)\leq 7-5p^{-1}-p^{-2}+\frac{7p^{-R}+7p^{-R-1}-5p^{-R-2}-p^{-R-3}}{1-p^{-R-1}}.
\]
For $p\geq 7$ we have $\Omega(p)\leq 7-4p^{-1}$.

\item For $R\geq 4$ with $R$ even we have 
\[
\Omega(p)\leq 7-5p^{-1}-p^{-2}+\frac{15p^{-4}+7p^{-5}}{1-p^{-5}}.
\]
For $p\geq 7$ we have $\Omega(p)\leq 7-5p^{-1}$.
\end{enumerate}

\end{enumerate}
\end{lemma}
\begin{proof}
(1) We use Lemma \ref{lem:LocalDensityBounds} (2) to compute the $\omega_{\nu}(p)$. Noting that for each choice of $R$ the denominator is independent of $\bm{d}$ (and hence also $\nu$), we may combine all of the terms in the numerators. After combining and simplifying, we obtain the bounds listed in the statement.

(2) We use Lemma \ref{lem:LocalDensityBounds} (3) to compute the $\omega_{\nu}(p)$ and again combine the numerators together and simplify trivially. A long but straightforward calculation yields the claim.

\end{proof}

\section{Representations of integers as sums of squares of nearly $P_r$-numbers}\label{sec:sieving}
In this section, we prove Theorem \ref{thm:PrSuniversal}. 
\subsection{Setup for sieving theory}
We apply sieving theory to remove the representations that have $p\mid d_j$ for $p\leq y$ with some $y$ depending on $n$. For $Q$ the quadratic form (or sum of squares) defined in \eqref{eqn:diagonalsum} with $\ell=4$, the set to be sieved is
\[
\mathscr{A}=\mathscr{A}_m:=\{x\in \Z^{4}:Q(x)=m\}.
\]
For $\bm{d}\in\S^{4}$ with $\gcd(d_j,30)=1$, we define 
\[
\mathscr{A}_{\bm{d}}:=\{  x\in \mathscr{A}:d_j\mid x_j\}.
\]
One has 
\[
r_{Q_{\bm{a}\cdot\bm{d}^2}}(n)=\# \mathscr{A}_{\bm{d}}.
\]
Defining 
\[
R(\bm{d},m)=R_{\bm{a}}(\bm{d},m):=r_{Q_{\bm{a}\cdot \bm{d}^2}}(m) - a_{E_{Q_{\bm{a}\cdot \bm{d}^2}}}(m)
\]
to be the coefficient of the cuspidal part of the theta function, we have the following proposition, which is a direct corollary of Lemma \ref{lem:LocalDensityBounds}
\begin{proposition}\label{prop:BrudernFouvry}
For $\bm{d}\in\N^4$ with $\gcd(d_j,30)=1$ we have 
\[
\#\mathscr{A}_{\bm{d}} = \frac{X}{d_1d_2d_3d_4}\omega\left(\bm{d},m\right) + R(\bm{d},m).
\]


\end{proposition}

We require the following bound on $\omega_1(p)$. 
\begin{lemma}\label{lem:omega1produpper}
  Suppose that $\bm{a}\in\N^4$ is only divisible by primes $\leq 13$ and for each prime $7\leq p\leq 13$ we have $\ord_{p}\prod_{j=1}^4a_j\leq 1$. Let $\mathscr{P}_{\bm{a}}$ denote the primes dividing $\prod_{j=1}^4a_j$. Then for $w\geq 5$, we have 
\[
\prod_{w<p< z} \left(1-\frac{\omega_1(p)}{p}\right)^{-1}\leq \frac{\prod_{w<p< z} \left(1-\frac{1+\frac{2}{p}}{p}\right)^{-1}}{\prod_{p\in\mathscr{P}_{\bm{a}}} \left(1-\frac{\omega_1(p)}{p}\right)}\leq 2\prod_{w<p< z} \left(1-\frac{9}{7p}\right)^{-1}\leq 2\prod_{w<p< z} \left(1-\frac{1}{p}\right)^{-2}.
\]
\end{lemma}
\begin{proof}
First assume that $p\nmid \prod_{j=1}^4 a_j$ and we may use Lemma \ref{lem:LocalDensityBounds} (2).

Bounding case-by-case, we find that for $R=0$ we have $\omega_1(p)\leq 1-p^{-1}$ and for $R=1$ we have $\omega_1(p)\leq (1-p^{-1})^{-1}$. For $R\geq 2$, we split into $R$ even and $R$ odd and plug in the two choices $\eta_{p,4}^{\mathcal{A},\operatorname{e}}(\bm{a}\cdot\bm{d}^2)=\eta_{p,4}^{\mathcal{A},\operatorname{e}}(\bm{a})=\pm 1$ to obtain $\omega_1(p)\leq 1+2p^{-1}$. We conclude that for all choices of $R$ and $p\geq 7$ we have 
\[
\omega_1(p)\leq 1+\frac{2}{p}.
\]
This gives the first inequality. We then bound $1+2p^{-1}\leq \frac{9}{7}$ for $p\geq 7$ and therefore, noting that $1-\frac{7}{5p}<1$, 
\[
\prod_{w<p< z} \left(1-\frac{\omega_1(p)}{p}\right)^{-1}\leq \prod_{w<p< z} \left(1-\frac{9}{7p}\right)^{-1}\prod_{p\in\mathscr{P}_{\bm{a}}}\left(1-\frac{\omega_1(p)}{p}\right)^{-1}.
\]
For the primes $p\in\mathscr{P}_{\bm{a}}$ we have by assumption that $7\leq p\leq 13$ and we may use the bounds for $\omega_1(p)$ in Lemma \ref{lem:LocalDensityBounds} (3). 
A direct calculation for $7\leq p\leq 13$ then shows that 
\[
\prod_{p\in\mathscr{P}_{\bm{a}}}\left(1-\frac{\omega_1(p)}{p}\right)^{-1}<2. 
\]

The final inequality holds by comparing $1-\frac{9}{7p}$ with $\left(1-\frac{1}{p}\right)^2$. 
\end{proof}

For a set $S$, we define $\chi_S(x)$ to be the characteristic function $\chi_S(x):=1$ if $x\in S$ and $\chi_S(x)=0$ otherwise. 

\begin{lemma}\label{lem:omega1upperlog}
For $1<w\leq z$ and $S\subseteq\N$, we have 
\[
\prod_{\max(w,7)\leq   p< z} \left(1-\frac{\chi_{S}(p) \omega_1(p)}{p}\right)^{-1}
\leq 2\left(\frac{\log(z)}{\log(w)}\right)\left(1+\frac{6}{\log(w)}\right).
\]
\end{lemma}
\begin{proof}
Since 
\[
\left(1-\frac{\chi_{S}(p) \omega_1(p)}{p}\right)^{-1}\leq \left(1-\frac{ \omega_1(p)}{p}\right)^{-1},
\]
it suffices to prove the claim for $S=\N$.

Using Lemma \ref{lem:omega1produpper}, we bound 
\begin{multline*}
\prod_{\max(w,7)\leq  p< z} \left(1-\frac{\omega_1(p)}{p}\right)^{-1}\leq \prod_{\max(w,7)\leq  p< z} \left(1-\frac{1+\frac{2}{p}}{p}\right)^{-1}\\
\leq  \prod_{\max(w,7)\leq  p< z}\frac{p}{p-1} \prod_{\max(w,7)\leq  p< z}\left(1-\frac{3}{p^2}\right)^{-1}.
\end{multline*}
We then bound 
\[
\prod_{\max(w,7)\leq p< z}\left(1-\frac{3}{p^2}\right)^{-1}\leq \prod_{p\geq \max(w,7)}\left(1-\frac{3}{p^2}\right)^{-1} = 1+\sum_{\substack{n> 1\\ p\mid n\implies p\geq \max(w,7)}} \frac{3^{\omega(n)}}{n^2}.
\]
We then use Lemma \ref{lem:2omega(n)bnd} and bound the sum against the Riemann integral to obtain 
\begin{align*}
1+\sum_{\substack{n> 1\\ p\mid n\implies p\geq \max(w,7)}} \frac{3^{\omega(n)}}{n^2}&\leq 1+1.614\sum_{n\geq \max(w,7)} \frac{1}{n^{\frac{3}{2}}}\leq 1+\frac{3.228}{\sqrt{\max(w,7)}}.
\end{align*}
We next use Lemma \ref{lem:RosserSchoenfeld} to obtain
\begin{multline*}
\prod_{\max(w,7)\leq p< z} \left(1-\frac{\omega_1(p)}{p}\right)^{-1}\leq \left(1+\frac{3.228}{\sqrt{\max(w,7)}}\right) \prod_{p< z}\frac{p}{p-1}\prod_{p\leq \max(w,7)}\frac{p-1}{p}\\
\leq \left(1+\frac{3.228}{\sqrt{\max(w,7)}}\right) \frac{\log(z)}{\log(\max(w,7))}\left(1+\frac{1}{\log^2(z)}\right)\left(1+\frac{1}{2\log^2(\max(w,7))}\right)\\
\leq \left(1+\frac{3.228}{\sqrt{\max(w,7)}}\right) \frac{\log(z)}{\log(\max(w,7))}\left(1+\frac{1}{\log^2(\max(w,7))}\right)\left(1+\frac{1}{2\log^2(\max(w,7))}\right).
\end{multline*}
We then obtain the claim by bounding
\[
\left(1+\frac{3.228}{\sqrt{\max(w,7)}}\right)\left(1+\frac{1}{\log^2(\max(w,7))}\right)\left(1+\frac{1}{2\log^2(\max(w,7))}\right)
\leq  \left(1+\frac{6}{\log(w)}\right).\qedhere
\]
\end{proof}

For $\beta,D>0$ and $d\in\S$ of the form $d=p_1p_2\cdots p_{r}$ with $p_1>p_2>\dots>p_r$, we next define the \begin{it}Rosser weights\end{it} $\lambda_{d}^{\pm}=\lambda_{d,D}^{\pm}(\beta)$. Setting 
\[
y_m=y_{m}(D,\beta):=\left(\frac{D}{p_1\cdots p_m}\right)^{\frac{1}{\beta}},
\]
these are defined by
\begin{align*}
\lambda_d^+=\lambda_{d,D}^+(\beta)&:=\begin{cases} (-1)^r&\text{if }p_{2\ell+1}<y_{2\ell+1}(D,\beta)\quad \forall 0\leq \ell\leq \frac{r-1}{2},\\
0&\text{otherwise},
\end{cases}\\
\lambda_d^-=\lambda_{d,D}^-(\beta)&:=\begin{cases} (-1)^r&\text{if }p_{2\ell}<y_{2\ell}(D,\beta)\quad \forall 0\leq \ell\leq \frac{r}{2},\\
0&\text{otherwise}.
\end{cases}
\end{align*}
As is standard, we consider $D$ and $\beta$ to be fixed throughout and omit these in the notation. For $\beta>1$, we define $a=a_{\beta}:=e\frac{\beta}{\beta-1}\log\left(\frac{\beta}{\beta-1}\right)$, $r=r_{\beta}:=\frac{\log\left(1+\frac{6}{\log(7)}\right)}{\log\left(\frac{\beta}{\beta-1}\right)}$, and 
\[
\mathfrak{C}_{\beta}(s):=2e^{r_{\beta}-1}\left(1+\frac{6}{\log(7)}\right)\frac{a_{\beta}^{\left\lfloor s-\beta\right\rfloor +1}}{1-a_{\beta}}.
\]
 We often simply write $a$ for $a_{\beta}$ and $r$ for $r_{\beta}$.
\begin{lemma}\label{lem:lambdadsums}
Suppose that for $\bm{a}\in\N^4$ we have $p\nmid a_j$ for every $p\geq 17$ and for $7\leq p\leq 13$ we have $\ord_{p}\prod_{j=1}^4 a_j\leq 1$. Let a subset $P$ of primes be given and set $S=S_P$ to be the set of all squarefree integers for which $d\in S$ if and only if all prime divisors of $d$ are in $P$. Let $D>0$ and $\beta\geq 5 $ be given and set $s:=\frac{\log(D)}{\log(z)}$.  Then for $s\geq \beta$ and $z\geq 7$ the following hold:
\begin{align*}
\sum_{d\mid P_7(z)}\lambda_d^+\frac{\chi_{S}(d)\omega_1(d)}{d}&\leq \prod_{7\leq p\leq z}\left(1-\frac{\chi_{S}(p)\omega_1(p)}{p}\right)\left(1+\mathfrak{C}_{\beta}(s)\right),\\
\sum_{d\mid P_7(z)}\lambda_d^-\frac{\chi_S(d)\omega_1(d)}{d}&\geq \prod_{7\leq p\leq z}\left(1-\frac{\chi_S(p)\omega_1(p)}{p}\right)\left(1-\mathfrak{C}_{\beta}(s)\right).
\end{align*}
\end{lemma}
\begin{proof}
We take $z_n:=\max\Big(z^{\left(\frac{\beta-1}{\beta}\right)^n},7\Big)$ in \cite[(6.29) and (6.30)]{IwaniecKowalski}. We define
\begin{align*}
V_S^+(z)&:= \sum_{d\mid P_7(z)}\lambda_d^+\frac{\omega_1(d)}{d}=V_S(z)+ \sum_{n\text{ odd}} V_{S,n}(z),\\
V_S^-(z)&:=\sum_{d\mid P_7(z)}\lambda_d^-\frac{\omega_1(d)}{d}=V_S(z)- \sum_{n\text{ even}} V_{S,n}(z),
\end{align*}
where $V_S(z):=\prod_{7\leq p< z} \left(1-\frac{\chi_S(p)\omega_1(p)}{p}\right)$ and 
\[
V_{S,n}(z):=\sum_{\substack{y_n<p_n<\dots p_1<z\\ p_m<y_m,\ m<n,\ m\equiv n\pmod{2}}} \frac{\omega_1\left(p_1p_2\cdots p_n\right)}{p_1p_2\cdots p_n} V_S\left(p_n\right).
\]
By \cite[p. 157]{IwaniecKowalski}, we have 
\begin{equation}\label{eqn:VnzVznbnd}
V_{S,n}(z)\leq \frac{V_S(z_n)}{n!}\left(\log\left(\frac{V_S(z_n)}{V_S(z)}\right)\right)^n.
\end{equation}
We then write 
\[
V_S(z_n)= \prod_{7\leq p< z_n} \left(1-\frac{\chi_S(p)\omega_1(p)}{p}\right) =V_S(z)\prod_{z_n\leq  p< z}\left(1-\frac{\chi_S(p)\omega_1(p)}{p}\right)^{-1}.
\]
Hence we may now use Lemma \ref{lem:omega1upperlog} (noting that since $z_n\geq 7$ we have $\max(z_n,7)=z_n$ and also $z\geq z^{\left(\frac{\beta-1}{\beta}\right)^n}$ together with the assumption $z\geq 7$ implies $z\geq z_n$) to bound 
\[
V_S(z_n)\leq 2 V_S(z)\frac{\log(z)}{\log(z_n)} \left(1+\frac{6}{\log(z_n)}\right).
\]
Plugging this into \eqref{eqn:VnzVznbnd} and then using $z_n\geq 7$, we therefore conclude that 
\[
V_{S,n}(z)\leq 2\frac{V_S(z)}{n!}\frac{\log(z)}{\log(z_n)} \left(1+\frac{6}{\log(7)}\right) \left(\log\left(\frac{\log(z)}{\log(z_n)} \left(1+\frac{6}{\log(7)}\right)\right)\right)^n.
\]
By Stirling's bound (using a more precise version by Robbins \cite{Robbins}), we have $n!\geq \sqrt{2\pi n} \left(\frac{n}{e}\right)^n$, from which we conclude that $n!\geq e \left(\frac{n}{e}\right)^n$. Thus we conclude that 
\[
V_{S,n}(z)\leq 2\frac{V_S(z)}{e\left(\frac{n}{e}\right)^n}\frac{\log(z)}{\log(z_n)} \left(1+\frac{6}{\log(7)}\right) \left(\log\left(\frac{\log(z)}{\log(z_n)} \left(1+\frac{6}{\log(7)}\right)\right)\right)^n.
\] 
Using $z_n\geq  z^{\left(\frac{\beta-1}{\beta}\right)^n}$, we have $\frac{\log(z)}{\log(z_n)} \leq  \left(\frac{\beta}{\beta-1}\right)^n$, and hence we conclude that
\begin{align*}
V_{S,n}(z)&\leq 2\frac{V_S(z)}{e n^n }\left(e\frac{\beta }{\beta-1}\right)^n \left(1+\frac{6}{\log(7)}\right) \left(\log\left(\left(\frac{\beta}{\beta-1}\right)^n \left(1+\frac{6}{\log(7)}\right)\right)\right)^n\\
&= 2\frac{V_S(z)}{e }\left(e\frac{\beta }{\beta-1} \log\left(\frac{\beta}{\beta-1}\right)\right)^n \left(1+\frac{6}{\log(7)}\right)  \left(1 + \frac{\log\left(1+\frac{6}{\log(7)}\right)}{n\log\left(\frac{\beta}{\beta-1}\right)}\right)^n\\
&= 2\frac{V_S(z)}{e}a^n\left(1+\frac{6}{\log(7)}\right)   \left(1+\frac{r}{n}\right)^n\leq  2V_S(z)a^n  e^{r-1}\left(1+\frac{6}{\log(7)}\right).
\end{align*}
Note that $V_n(z)=0$ for $n\leq s-\beta$ and $a<1$ because $\beta\geq 5$, so the geometric series converges to give
\[
\sum_{n\geq 1} V_{S,n}(z)\leq 2V_S(z) e^{r-1}\left(1+\frac{6}{\log(7)}\right)\sum_{n>s-\beta} a^n\leq V_S(z) \mathfrak{C}_{\beta}(s).
\]
Hence 
\[
V_S^{+}(z)\leq V_S(z)\left(1+  \mathfrak{C}_{\beta}(s)\right)\qquad\text{ and }\qquad  V_S^{-}(z)\geq V_S(z)\left(1-  \mathfrak{C}_{\beta}(s)\right).\qedhere
\]
\end{proof}
 
We next prove bounds on sums of the type in Lemma \ref{lem:lambdadsums} under the additional restriction that we only sum over those $d$ with $\delta\mid d$ (for fixed $\delta\in\N$).
\begin{lemma}\label{lem:lambdadivsums}
Let $D>0$ and $\beta\geq 5 $ be given and set $s:=\frac{\log(D)}{\log(z)}$.  Then for $s\geq \beta$, $z\geq 7$, and squarefree $\delta\in\N$ with $\gcd(\delta,30)=1$ the following hold:
\begin{align*}
\sum_{\delta\mid d\mid P_7(z)}\lambda_d^+\frac{\omega_1(d)}{d}&\leq \mu(\delta)\prod_{p\mid \delta}\frac{\omega_1(p)}{p-\omega_1(p)} \prod_{7\leq p\leq z}\left(1-\frac{\omega_1(p)}{p}\right)\left(1+\mathfrak{C}_{\beta}(s)\right),\\
\sum_{\delta\mid d\mid P_7(z)}\lambda_d^-\frac{\omega_1(d)}{d}&\geq \mu(\delta)\prod_{p\mid \delta}\frac{\omega_1(p)}{p-\omega_1(p)} \prod_{7\leq p\leq z}\left(1-\frac{\omega_1(p)}{p}\right)\left(1-\mathfrak{C}_{\beta}(s)\right).
\end{align*}
Moreover, we have 
\[
\mu(\delta)\prod_{p\mid \delta}\frac{\omega_1(p)}{p-\omega_1(p)}\leq \frac{\left(\frac{5}{3}\right)^{\omega(\delta)}}{\delta}.
\]
\end{lemma}
\begin{proof}
Setting 
\[
f_{\delta}(n):=\begin{cases} 1&\text{if }\delta\mid n,\\ 0&\text{otherwise}\end{cases}
\]
and 
\[
\widetilde{f}_{\delta}(n):=\begin{cases} 1&\text{if }\gcd(n,\delta)=1,\\ 0&\text{otherwise}\end{cases},
\]
for a prime $p$ we have $f_{p}(n)=1-\widetilde{f}_{p}(n)$. Since $\delta$ is squarefree, we have 
\[
f_{\delta}(n)=\prod_{p\mid \delta}f_{p}(n)= \prod_{p\mid \delta}\left(1-\widetilde{f}_{p}(n)\right)=\sum_{d\mid \delta} \mu(d)\widetilde{f}_d(n). 
\]
Hence we may rewrite 
\begin{multline}\label{eqn:divsum}
\sum_{\delta\mid d\mid P_7(z)}\lambda_d^{\pm}\frac{\omega_1(d)}{d}=\sum_{d\mid P_7(z)}f_{\delta}(d)\lambda_d^{\pm}\frac{\omega_1(d)}{d}=\sum_{d\mid P_7(z)}\sum_{u\mid \delta}\mu(u) \widetilde{f}_{u}(d)\lambda_d^{\pm}\frac{\omega_1(d)}{d}\\
 =\sum_{u\mid \delta}\mu(u) \sum_{d\mid P_7(z)} \widetilde{f}_{u}(d)\lambda_d^{\pm}\frac{\omega_1(d)}{d}.
\end{multline}
Letting $S_u$ be the set of squarefree integers $d$ with $\gcd(d,u)=1$ (i.e., $S_u=S_P$ for $P$ the set of all primes not dividing $u$ in Lemma \ref{lem:lambdadsums}), the inner sum may be written as 
\[
\sum_{d\mid P_7(z)} \widetilde{f}_{u}(d)\lambda_d^{\pm}\frac{\omega_1(d)}{d}=\sum_{d\mid P_7(z)} \chi_{S_u}(d)\lambda_d^{\pm}\frac{\omega_1(d)}{d}.
\]
We may therefore use Lemma \ref{lem:lambdadsums} (and $\sum_{d\mid P_7(z)} \chi_{S_u}(d)\lambda_d^{-}\frac{\omega_1(d)}{d}\leq \sum_{d\mid P_7(z)} \chi_{S_u}(d)\lambda_d^{+}\frac{\omega_1(d)}{d}$) to bound 
\[
\left|\sum_{d\mid P_7(z)} \chi_{S_u}(d)\lambda_d^{\pm}\frac{\omega_1(d)}{d}-V_{S_u}(z)\right|\leq V_{S_u}(z)\mathfrak{C}_{\beta}(s).
\]
Thus Lemma \ref{lem:lambdadsums} and \eqref{eqn:divsum} imply that 
\begin{align}
\label{eqn:lambdadiv+}
\sum_{\delta\mid d\mid P_7(z)} \lambda_d^{+}\frac{\omega_1(d)}{d}&\leq \sum_{u\mid \delta}\mu(u)V_{S_u}(z)\left(1 +\mathfrak{C}_{\beta}(s)\right),\\
\label{eqn:lambdadiv-}\sum_{\delta\mid d\mid P_7(z)}\lambda_d^{-}\frac{\omega_1(d)}{d}&\geq \sum_{u\mid \delta}\mu(u)V_{S_u}(z)\left(1 -\mathfrak{C}_{\beta}(s)\right).
\end{align}
We then evaluate (abbreviating $V(z):=V_{\N}(z)$)
\begin{align*}
\sum_{u\mid \delta}\mu(u) V_{S_u}(z)&=\sum_{u\mid \delta}\mu(u)\prod_{7\leq p<z}\left(1-\frac{\chi_{S_u}(p)\omega_1(p)}{p}\right)=V(z) \sum_{u\mid \delta}\frac{\mu(u)}{\prod_{p\mid u} \left(1-\frac{\omega_1(p)}{p}\right)}\\
&=V(z)\prod_{p\mid \delta} \left(1-\frac{p}{p-\omega_1(p)}\right)=V(z)\mu(\delta)\prod_{p\mid \delta} \frac{\omega_1(p)}{p-\omega_1(p)},
\end{align*}
where in the last step we used the fact that $(-1)^{\omega(\delta)}=\mu(\delta)$ because $\delta$ is squarefree. Plugging back into \eqref{eqn:lambdadiv+} and \eqref{eqn:lambdadiv-} yields the claim. 
 \end{proof}

We have the following lower bound for $X=a_{E_{Q_{\bm{a}}}}(m)$.
\begin{lemma}\label{lem:Xlower}
Suppose that 
\begin{multline*}
\bm{a}\in \{(1,1,1,k):1\leq k\leq 7\}\cup\{(1,1,2,k):2\leq k\leq 8\}\cup \{(1,1,3,k):3\leq k\leq 6\}\\
\cup \{(1,2,2,k):2\leq k\leq 7\}\cup\{(1,2,3,k):3\leq k\leq 8\}\cup\{(1,2,4,k):4\leq k\leq 14\}\\
\cup\{(1,2,5,k):4\leq k\leq 15\}.
\end{multline*}
Then we have 
\[
X=a_{E_{Q_{\bm{a},\bm{1}}}}(m)\geq  0.024  m^{1-10^{-6}} \prod_{p\mid \gcd(30,\Delta m)} \frac{\beta_{Q_{\bm{a}},p}(m)}{\left(1-p^{-2}\right)\left(1-p^{-1}\right)}. 
\]
In particular, if $8\nmid m$, $27\nmid m$, and $25\nmid m$, then we have 
\[
a_{E_{Q_{\bm{a},\bm{1}}}}(m)\geq  0.00083  m^{1-10^{-6}}.  
\]

\end{lemma}
\begin{proof}

Taking $\bm{\alpha}=\bm{0}$ in Lemma \ref{lem:LocalDensitySimplify} (1)(e), for those primes $p$ not dividing $\Delta$ (in particular, for $p>13$ in every case considered in this lemma) we have (letting $R_p:=\ord_p(m)$)
\begin{multline*}
\beta_{Q_{\bm{a}},p}(m)= \lim_{r\to\infty} p^{-3r}R_{Q,p^r}(m) = 1+ \frac{p^{-2}}{1+\frac{1}{p}}\left(1-p^{-2\left\lfloor \frac{R_p}{2}\right\rfloor}+\eta_{p,4}^{\mathcal{A},\operatorname{e}}(\bm{a})p\left(1-p^{-2\left\lfloor\frac{R_p+1}{2}\right\rfloor}\right)\right)\\
 - \left(\delta_{2\nmid R_p}+\eta_{p,4}^{\mathcal{A},\operatorname{e}}(\bm{a})\delta_{2\mid R_p}\right)p^{-R_p-2}\geq \begin{cases} 1-p^{-2}&\text{if }R_p=0,\\   \left(1-p^{-1}\right)\left(1-p^{-2}\right)
&\text{if }R_p\geq 1.\end{cases}
\end{multline*}
For the infinite prime, comparing \eqref{eqn:localinfty} between $\bm{a}$ and $\bm{1}$ (making the change of variables $x_j\mapsto \frac{x_j}{\sqrt{a_j}}$) for $U=(m-\varepsilon,m+\varepsilon)$ we have 
\begin{equation}\label{eqn:BQinfty}
\beta_{Q_{\bm{a}},\infty} (m) =\frac{\pi^2}{2\sqrt{a_1a_2a_3a_4}}\lim_{\varepsilon\to 0^+} \frac{(m+\varepsilon)^2-(m-\varepsilon)^2}{2\varepsilon}= \frac{2\pi^2}{\sqrt{\Delta}}\lim_{\varepsilon\to 0^+} \frac{4m\varepsilon}{2\varepsilon}=\frac{4\pi^2}{\sqrt{\Delta}} m.
\end{equation}
Therefore, evaluating $\zeta(2)=\frac{\pi^2}{6}$,
\begin{equation}\label{eqn:aEQboundd=1}
a_{E_{Q}}(m)\geq \frac{4\pi^2}{\sqrt{\Delta}} \frac{6}{\pi^2}m \prod_{\substack{p|m\\ p\nmid \Delta}}\left(1-p^{-1}\right) \prod_{\substack{p\mid \Delta}} \frac{\beta_{Q_{\bm{a}},p}(m)}{1-p^{-2}}.
\end{equation}
For $p\geq 7$ with $p\mid \Delta$, we have $\bm{\alpha}=(0,0,0,1)$ and for $p=5$ we have $\bm{\alpha}=(0,0,0,1)$ or $\bm{\alpha}=(0,0,1,1)$. Suppose that $\bm{\alpha}=(0,0,0,1)$. If $p\mid m$, then $R_p\geq \alpha_4$, while for $p\nmid m$ we have $\alpha_3=R_p=0<\alpha_4=1$. For $p\mid m$ we use Lemma \ref{lem:LocalDensitySimplify} (5)(e) to bound
\begin{equation}\label{eqn:Rbigalpha=0,0,0,1}
\beta_{Q_{\bm{a}},p}(m)\geq 1-p^{-R_p-1}\geq 1-p^{-2}>1-p^{-1}.
\end{equation}
For $R_p=0$, we have $\alpha_3\leq R_p<\alpha_4$, so we may use Lemma \ref{lem:LocalDensitySimplify} (5)(d) to bound
\begin{equation}\label{eqn:R=0alpha=0,0,0,1}
\beta_{Q_{\bm{a}},p}(m)\geq 1-p^{-1}.
\end{equation}
Plugging back into \eqref{eqn:aEQboundd=1}, we have 
\[
a_{E_Q}(m)\geq 
\frac{24}{\sqrt{\Delta}} m \prod_{\substack{p|\Delta m\\ p\nmid 30}} \left(1-p^{-1}\right) \prod_{p\mid \gcd(30,\Delta m)} \frac{\beta_{Q_{\bm{a}},p}(m)}{1-p^{-2}}.
\]
Taking $m\mapsto m\Delta$ in Lemma \ref{lem:cdelbnd} (2), we obtain 
\[
\prod_{\substack{p\mid m\Delta\\ p\nmid 30}} \left(1-p^{-1}\right)=\frac{\prod_{p\mid m\Delta}\left(1-p^{-1}\right)}{\prod_{p\mid \gcd(m\Delta,30)} \left(1-p^{-1}\right)} \geq \frac{\frac{1}{20}(m\Delta)^{-10^{-6}}}{\prod_{p\mid \gcd(m\Delta,30)}\left(1-p^{-1}\right)},
\]
so 
\[
a_{E_Q}(m)\geq \frac{6}{5} \Delta^{-\frac{1}{2}-10^{-6}}  m^{1-10^{-6}} \prod_{p\mid \gcd(30,\Delta m)} \frac{\beta_{Q_{\bm{a}},p}(m)}{\left(1-p^{-2}\right)\left(1-p^{-1}\right)}.
\]
Since $\Delta\leq 2400$, we have 
\begin{equation}\label{eqn:aEQno2or3}
a_{E_Q}(m)\geq 0.024  m^{1-10^{-6}} \prod_{p\mid \gcd(30,\Delta m)} \frac{\beta_{Q_{\bm{a}},p}(m)}{\left(1-p^{-2}\right)\left(1-p^{-1}\right)}.
\end{equation}
This is the first claim.

Now assume that $8\nmid m$, $27\nmid m$, and $25\nmid m$. We compute a lower bound for $\beta_{Q_{\bm{a}},2}(m)$, $\beta_{Q_{\bm{a}},3}(m)$, and $\beta_{Q_{\bm{a}},5}(m)$ under this assumption. Using Lemma \ref{lem:LocalDensitySimplify} (1) and Lemma \ref{lem:LocalDensitySimplify} (6), a case-by-case analysis yields that
\[
\frac{\beta_{Q_{\bm{a}},5}(m)}{\left(1-5^{-2}\right)\left(1-5^{-1}\right)}  \frac{\beta_{Q_{\bm{a}},3}(m)}{\left(1-3^{-2}\right)\left(1-3^{-1}\right)} \geq \frac{5}{24}\cdot \frac{1}{8}=\frac{5}{192}. 
\]
Plugging this into \eqref{eqn:aEQno2or3} and then using Lemma \ref{lem:localdensityp=2} (2) yields (note that $2\mid \Delta$ in our case, so the prime $2$ always occurs)
\[
a_{E_Q}(m)\geq 0.024  m^{1-10^{-6}} \frac{\frac{5}{192} \beta_{Q_{\bm{a}},2}(m)}{\left(1-2^{-1}\right)\left(1-2^{-2}\right)} \geq 0.024  m^{1-10^{-6}} \frac{5}{144} > 0.00083 m^{1-10^{-6}}.\qedhere
\]
\end{proof}

For $w\in\R$, define 
\[
S\left(\mathscr{A}_{\bm{1}},z\right)=S_w\left(\mathscr{A},z\right):=\#\left\{\bm{x}\in \mathscr{A}_{\bm{1}}: \gcd(x_j,P_w(z))=1\right\}=\sum_{\bm{x}\in \mathscr{A}_{\bm{1}}}\prod_{j=1}^4 (\mu\ast \bm{1})\left(\gcd\left(x_j,P_w(z)\right)\right).
\]
As in \cite{BrudernFouvry}, we now define $\Lambda_{d}^-:=4\lambda_{d}^--3\lambda_{d}^+$ and 
\begin{equation}\label{eqn:Sigmadef}
\Sigma\left(D,z\right):=\sum_{d_1\mid P_7(z)}\sum_{d_2\mid P_7(z)}\sum_{d_3\mid P_7(z)}\sum_{d_4\mid P_7(z)}\Lambda_{d_1}^-\lambda_{d_2}^+\lambda_{d_3}^+\lambda_{d_4}^+ \frac{\omega(\bm{d},m)}{d_1d_2d_3d_4}.
\end{equation}
Combining this with the function
\begin{equation}\label{eqn:Sigma'def}
\Sigma'(D,z):=\sum_{d_1\mid P_7(z)}\sum_{d_2\mid P_7(z)}\sum_{d_3\mid P_7(z)}\sum_{d_4\mid P_7(z)}\lambda_{d_1}^+\lambda_{d_2}^+\lambda_{d_3}^+\lambda_{d_4}^+ \frac{\omega(\bm{d},m)}{d_1d_2d_3d_4},
\end{equation}
we next obtain an upper and lower bound for $S(\mathscr{A},z)$.
\begin{lemma}\label{lem:SAz}
For $w\geq 7$ we have 
\[
X \Sigma(D,z) -4\sum_{\substack{\bm{d}\in\Z^4\\ d_j\mid P_w(z)\\ |d_j|\leq \frac{D}{11^{\beta-1}}}}\left| R(\bm{d},m)\right|\leq S\left(\mathscr{A},z\right)\leq X \Sigma'(D,z)+\sum_{\substack{\bm{d}\in\Z^4\\ d_j\mid P_w(z)\\ |d_j|\leq \frac{D}{11^{\beta-1}}}}\left| R(\bm{d},m)\right|.
\]
\end{lemma}
\begin{proof}
We claim first that 
\begin{multline}\label{eqn:Supperlower1}
\sum_{d_1\mid P_w(z)}\sum_{d_2\mid P_w(z)}\sum_{d_3\mid P_w(z)}\sum_{d_4\mid P_w(z)}\left(4\lambda_{d_1}^--3\lambda_{d_1}^+\right)\lambda_{d_2}^+\lambda_{d_3}^+\lambda_{d_4}^+ \#\mathscr{A}_{\bm{d}}\leq S\left(\mathscr{A},z\right)\\
\leq \sum_{d_1\mid P_w(z)}\sum_{d_2\mid P_w(z)}\sum_{d_3\mid P_w(z)}\sum_{d_4\mid P_w(z)}\lambda_{d_1}^+\lambda_{d_2}^+\lambda_{d_3}^+\lambda_{d_4}^+ \# \mathscr{A}_{\bm{d}}.
\end{multline}
To show \eqref{eqn:Supperlower1}, we first write (by inclusion-exclusion)
\[
 S\left(\mathscr{A},z\right)=\sum_{d_1\mid P_w(z)}\sum_{d_2\mid P_w(z)}\sum_{d_3\mid P_w(z)}\sum_{d_4\mid P_w(z)} \mu(d_1)\mu(d_2)\mu(d_3)\mu(d_4)\#\mathscr{A}_{\bm{d}}.
\]
We then use the inequality (see \cite[(6.19)]{IwaniecKowalski})
\[
\sum_{d\mid \ell} \lambda_{d}^-\leq \sum_{d\mid \ell} \mu(d)\leq \sum_{d\mid \ell} \lambda_{d}^+.
\]
From this the second inequality immediately follows. For the first inequality, we argue as in \cite[Lemma 13]{BrudernFouvry}.

We now plug Proposition \ref{prop:BrudernFouvry} in for $\#\mathscr{A}_d$ in \eqref{eqn:Supperlower1} to obtain 
\begin{multline*}
\sum_{d_1\mid P_w(z)}\sum_{d_2\mid P_w(z)}\sum_{d_3\mid P_w(z)}\sum_{d_4\mid P_w(z)}\left(4\lambda_{d_1}^--3\lambda_{d_1}^+\right)\lambda_{d_2}^+\lambda_{d_3}^+\lambda_{d_4}^+ \left(\frac{X}{d_1d_2d_3d_4} \omega(\bm{d},m) + R(\bm{d},m)\right)\\
 \leq S\left(\mathscr{A},z\right)\leq \sum_{d_1\mid P_w(z)}\sum_{d_2\mid P_w(z)}\sum_{d_3\mid P_w(z)}\sum_{d_4\mid P_w(z)}\lambda_{d_1}^+\lambda_{d_2}^+\lambda_{d_3}^+\lambda_{d_4}^+ \left(\frac{X}{d_1d_2d_3d_4} \omega(\bm{d},m) + R(\bm{d},m)\right).
\end{multline*}
We then use $|\lambda_{d}^{\pm}|\leq 1$ and $\lambda_d^{\pm}=0$ for $d>\frac{D}{11^{\beta-1}}$ (by the definition of the Rosser weights and the fact that the second-smallest prime dividing $d$ is at least $11$) and plug in the absolute value termwise for the sum on $R(\bm{d},m)$ to obtain the claim. 
\end{proof}

\subsection{Upper and lower bounds for the main term from sieving}
We next bound $\Sigma(D,z)$ from below to obtain a lower bound for $S(\mathscr{A}_{\bm{e}},z)$ from Lemma \ref{lem:SAz}. To state the result, we define
\begin{multline*}
\Sigma_{\operatorname{MT}}:=\prod_{7\leq p\leq z}\left(1-\frac{\omega_1(p)}{p}\right)^4\sum_{d_{1,2}\mid P_7(z)} \sum_{d_{1,3}\mid P_7(z)} \cdots \sum_{d_{3,4}\mid P_7(z)}g\left((d_{i,j})\right)\\
\times \sum_{\ell_{i,j}\mid \frac{P_7(z)}{d_{i,j}}}\mu\left(\ell_{1,2}\right)\cdots\mu\left(\ell_{3,4}\right)\prod_{j=1}^4 \mu(\xi_j)\prod_{p\mid \xi_j}\frac{\omega_1(p)}{p-\omega_1(p)}.
\end{multline*}

\begin{lemma}\label{lem:SigmaDzbound}
We have 
\[
\Sigma\left(D,z\right)\geq \left(1-7\mathfrak{C}_{\beta}(s)\right)\left( 1-\mathfrak{C}_{\beta}(s)\right)^3\Sigma_{\operatorname{MT}}.
\]
\end{lemma}
\begin{proof}
As in \cite[Lemma 12]{BrudernFouvry}, there exists a 6-variable function $g$ such that for $d_{i,j}:=\gcd(d_i,d_j)$
\[
\omega(\bm{d},m)= \omega_{1}(d_1)\omega_{1}(d_2)\omega_{1}(d_3)\omega_{1}(d_4) g\left((d_{i,j})\right).
\]
In other words, the ratio of $\omega(\bm{d},m)$ and the product of the $\omega_1(d_j)$ only depends on $d_{i,j}$. Therefore 
\[
\Sigma\left(D,z\right)=\sum_{d_1\mid P_7(z)}\sum_{d_2\mid P_7(z)}\sum_{d_3\mid P_7(z)}\sum_{d_4\mid P_7(z)}\Lambda_{d_1}^-\lambda_{d_2}^+\lambda_{d_3}^+\lambda_{d_4}^+ \frac{\omega_1(d_1)\omega_1(d_2)\omega_1(d_3)\omega_1(d_4)}{d_1d_2d_3d_4}g\left((d_{i,j})\right).
\]
We then rewrite the sum by taking $d_{i,j}$ to a series of 6 outer sums. This gives 
\begin{equation}\label{eqn:Sdijintro}
\sum_{d_{1,2}\mid P_7(z)} \sum_{d_{1,3}\mid P_7(z)} \cdots \sum_{d_{3,4}\mid P_7(z)}g\left((d_{i,j})\right)S\left((d_{i,j})\right),
\end{equation}
where
\[
S\left((d_{i,j})\right):= \underset{\gcd \left(d_i,d_j\right)=d_{i,j}}{\sum_{d_1\mid P_7(z)}\sum_{d_2\mid P_7(z)}\sum_{d_3\mid P_7(z)}\sum_{d_4\mid P_7(z)}}\Lambda_{d_1}^-\lambda_{d_2}^+\lambda_{d_3}^+\lambda_{d_4}^+ \frac{\omega_1(d_1)\omega_1(d_2)\omega_1(d_3)\omega_1(d_4)}{d_1d_2d_3d_4}.
\]
We then rewrite the condition $\gcd(d_i,d_j)=d_{i,j}$ by assuming that $\lcm\left(d_{i,j}:1\leq j\leq 4, j\neq i\right)\mid d_i$ and using inclusion-exclusion for divisibility by $\ell_{i,j}\mid \frac{P_7(z)}{d_{i,j}}$. Setting $\xi_{i}:=\lcm\left(\ell_{i,j}d_{i,j}:1\leq j\leq 4,\ j\neq i\right)$, this yields that 
\begin{multline*}
S\left((d_{i,j})\right)=\sum_{\ell_{i,j}\mid \frac{P_7(z)}{d_{i,j}}}\mu\left(\ell_{1,2}\right)\cdots\mu\left(\ell_{3,4}\right)
\\
\times \left(\sum_{\substack{d_1\mid P_7(z)\\ \xi_1\mid d_1}}\frac{\Lambda_{d_1}^- \omega_1(d_1)}{d_1}\right)\left(\sum_{\substack{d_2\mid P_7(z)\\ \xi_2\mid d_2}}\frac{\lambda_{d_2}^+ \omega_1(d_2)}{d_2}\right)\left(\sum_{\substack{d_3\mid P_7(z)\\ \xi_3\mid d_3}}\frac{\lambda_{d_3}^+ \omega_1(d_3)}{d_3}\right)\left(\sum_{\substack{d_4\mid P_7(z)\\ \xi_4\mid d_4}}\frac{\lambda_{d_4}^+ \omega_1(d_4)}{d_4}\right).
\end{multline*}
By Lemma \ref{lem:lambdadivsums} we can bound 
\begin{multline*}
\sum_{\substack{d_1\mid P_7(z)\\ \xi_1\mid d_1}}\Lambda_{d_1}^- \frac{\omega_1(d_1)}{d_1}\geq 4\mu(\xi_{1})\prod_{p\mid \xi_1}\frac{\omega_1(p)}{p-\omega_1(p)}\prod_{7\leq p\leq z}\left(1-\frac{\omega_1(p)}{p}\right)\left(1-\mathfrak{C}_{\beta}(s)\right)\\
 - 3 \mu(\xi_1) \prod_{p\mid \xi_1} \frac{\omega_1(p)}{p-\omega_1(p)}\prod_{7\leq p\leq z}\left(1-\frac{\omega_1(p)}{p}\right)\left(1+\mathfrak{C}_{\beta}(s)\right).
\end{multline*}
and 
\[
\sum_{\substack{d_j\mid P_7(z)\\ \xi_j\mid d_j}}\lambda_{d_j}^+ \frac{\omega_1(d_1)}{d_1}\geq \mu(\xi_{1})\prod_{p\mid \xi_1}\frac{\omega_1(p)}{p-\omega_1(p)}\prod_{7\leq p\leq z}\left(1-\frac{\omega_1(p)}{p}\right)\left(1-\mathfrak{C}_{\beta}(s)\right).
\]
Putting this all together, we have 
\begin{multline*}
S\left((d_{i,j})\right)\geq  \left(1-7\mathfrak{C}_{\beta}(s)\right)\left( 1-\mathfrak{C}_{\beta}(s)\right)^3\prod_{7\leq p\leq z}\left(1-\frac{\omega_1(p)}{p}\right)^4\\
\times \sum_{\ell_{i,j}\mid \frac{P_7(z)}{d_{i,j}}}\mu\left(\ell_{1,2}\right)\cdots\mu\left(\ell_{3,4}\right)\prod_{j=1}^4 \mu(\xi_j)\prod_{p\mid \xi_j}\frac{\omega_1(p)}{p-\omega_1(p)}.
\end{multline*}
Plugging this into \eqref{eqn:Sdijintro} yields the claim.
\end{proof}

In order to obtain an upper bound for $S(\mathscr{A},z)$, we now bound $\Sigma'(D,z)$ from above. 
\begin{lemma}\label{lem:Sigma'Dzbound}
We have 
\[
\Sigma'\left(D,z\right)\leq \left( 1+\mathfrak{C}_{\beta}(s)\right)^4\Sigma_{\operatorname{MT}}.
\]
\end{lemma}
\begin{proof}
As in the proof of Lemma \ref{lem:SigmaDzbound}, we first rewrite
\[
\Sigma'\left(D,z\right)=\sum_{d_1\mid P_7(z)}\sum_{d_2\mid P_7(z)}\sum_{d_3\mid P_7(z)}\sum_{d_4\mid P_7(z)}\lambda_{d_1}^+\lambda_{d_2}^+\lambda_{d_3}^+\lambda_{d_4}^+ \frac{\omega_1(d_1)\omega_1(d_2)\omega_1(d_3)\omega_1(d_4)}{d_1d_2d_3d_4}g\left((d_{i,j})\right).
\]
We then rewrite the sum by taking $d_{i,j}$ to a series of 6 outer sums. This gives 
\begin{equation}\label{eqn:Sdijintro'}
\sum_{d_{1,2}\mid P_7(z)} \sum_{d_{1,3}\mid P_7(z)} \cdots \sum_{d_{3,4}\mid P_7(z)}g\left((d_{i,j})\right)S^+\left((d_{i,j})\right),
\end{equation}
where
\[
S^+\left((d_{i,j})\right):= \underset{\gcd \left(d_i,d_j\right)=d_{i,j}}{\sum_{d_1\mid P_7(z)}\sum_{d_2\mid P_7(z)}\sum_{d_3\mid P_7(z)}\sum_{d_4\mid P_7(z)}}\lambda_{d_1}^+\lambda_{d_2}^+\lambda_{d_3}^+\lambda_{d_4}^+ \frac{\omega_1(d_1)\omega_1(d_2)\omega_1(d_3)\omega_1(d_4)}{d_1d_2d_3d_4}.
\]
Using Lemma \ref{lem:lambdadivsums}, the proof follows precisely as in the proof of Lemma \ref{lem:SigmaDzbound}. \qedhere
\end{proof}

\begin{lemma}\label{lem:SigmaMTeval}
We have 
\[
\Sigma_{\operatorname{MT}}= \prod_{7\leq p\leq z}\left(1-\frac{\Omega(p)}{p}\right).
\]
\end{lemma}
\begin{proof}
Combining Lemma \ref{lem:SigmaDzbound} with Lemma \ref{lem:Sigma'Dzbound}, we have 
\[
\left(1-7\mathfrak{C}_{\beta}(s)\right)\left(1-\mathfrak{C}_{\beta}(s)\right)^3\Sigma_{\operatorname{MT}}\leq \Sigma(D,z)\leq \Sigma'(D,z)\leq \left(1-\mathfrak{C}_{\beta}(s)\right)^4\Sigma_{\operatorname{MT}}. 
\]
We then take the limit $D\to \infty$ and in the limit $\lambda_d^+=\lambda_d^-=\mu(d)$ so that by the definitions \eqref{eqn:Sigmadef} and \eqref{eqn:Sigma'def} we have
\[
\lim_{D\to\infty} \Sigma(D,z)=\lim_{D\to\infty} \Sigma'(D,z)
=\prod_{7\leq p\leq z}\left(1-\frac{\Omega(p)}{p}\right)
\]
and $\mathfrak{C}_{\beta}(s)\to 0$ as $D\to\infty$. This gives the claim.
\end{proof}

\begin{lemma}\label{lem:BigsSigmaBound}
Suppose that $\bm{a}\in\N^4$ has at most one prime $p\geq 7$ dividing $\prod_{j=1}^4 a_j$ and moreover that $p\|\prod_{j=1}^{4}a_j$ and $7\leq p\leq 13$. If $\beta=10$ and $D\geq z^{25}$, then 
\[
\Sigma(D,z)\geq \frac{5}{8}\prod_{7\leq p<z} \left(1-\frac{\Omega(p)}{p}\right)\geq \frac{3}{80}\prod_{7\leq p<z}\left(1-\frac{4.93}{p}\right)\geq 12.6\prod_{p<z}\left(1-\frac{1}{p}\right)^{5}.
\]
\end{lemma}
\begin{proof}
A direct calculation shows that 
\[
\mathfrak{C}_{\beta}(s)\leq \frac{3}{80},
\]
from which we conclude the first inequality via Lemma \ref{lem:SigmaDzbound} and Lemma \ref{lem:SigmaMTeval}.  For the second inequality, if $p\nmid \prod_{j=1}^4 a_j$ then we use Lemma \ref{lem:LambdapUpper} (1) to bound $\Omega(p)\leq 4.93$, while for $p\|\prod_{j=1}^4a_j$ with $p\geq 7$ we have $7\leq p\leq 13$ (with at most one of them occurring) and we may use Lemma \ref{lem:LambdapUpper} (2) to bound 
\[
\prod_{\substack{7\leq p\leq 13\\ p\|\prod_{j=1}^{4}a_j}} \frac{1-\frac{\Omega(p)}{p}}{1-\frac{4.93}{p}}\geq \min\left( \frac{1-\frac{6.87}{7}}{1-\frac{4.93}{7}}, \frac{1-\frac{7}{11}}{1-\frac{4.93}{11}}, \frac{1-\frac{7}{13}}{1-\frac{4.93}{13}}\right)\geq 0.0628. 
\]
this yields the second inequality. 

 For the last inequality, we note that for $p>139$ we have $1-\frac{4.93}{p}\geq \left(1-\frac{1}{p}\right)^5$ and thus 
\[
\frac{5}{8}\Sigma_{\operatorname{MT}}\geq \frac{3}{80}\prod_{7\leq p<z}\left(1-\frac{4.93}{p}\right)\geq 27.8\prod_{7\leq p\leq 139}\frac{\left(1-\frac{4.93}{p}\right)}{\left(1-\frac{1}{p}\right)^5} \prod_{p<z} \left(1-\frac{1}{p}\right)^5\geq 12.6  \prod_{p<z} \left(1-\frac{1}{p}\right)^5.
\]
For the inequality with $\Sigma'(D,z)$, we plug $\mathfrak{C}_{\beta}(s)\leq \frac{3}{40}$ into Lemma \ref{lem:Sigma'Dzbound} to get the first inequality. The second inequality trivially holds because $0<\Omega(p)<p$. 
 \end{proof}

\subsection{Bounds of the error term from sieving}
We next bound the cuspidal contribution to obtain a bound for $S(\mathscr{A},z)$.
\begin{lemma}\label{lem:Rsumbound}
For $\beta\geq 10$, we have 
\[
\sum_{\substack{\bm{d}\in\Z^4\\ d_j\mid P_w(z)\\ d_j\leq \frac{D}{11^{\beta-1}}}}\left| R(\bm{d},m)\right|\leq
3.07\times 10^{-92}   m^{\frac{17}{30}}D^{24.05}.
\]
\end{lemma}
\begin{proof}
Since $R(\bm{d},m)=a_{f_{Q,\bm{d}}}(m)$, we may plug in \eqref{eqn:cuspboundfinal3} to obtain 
\begin{multline}\label{eqn:finalRbound}
\sum_{\substack{\bm{d}\in\Z^4\\ d_j\mid P_w(z)\\ d_j\leq \frac{D}{11^{\beta-1}}}}\left| R(\bm{d},m)\right|\leq 1.184\times 10^{131}m^{\frac{17}{30}}(520)^{1+2\cdot 10^{-6}+\frac{1}{200}}\\
\times\sum_{\substack{\bm{d}\in\Z^4\\ d_j\mid P_w(z)\\ d_j\leq \frac{D}{11^{9}}}}\lcm(\bm{d})^{2+4\cdot 10^{-6}+\frac{1}{100}}\left(27\pi \cdot 2400\prod_{j=1}^4d_j^2 + 16(520\lcm(\bm{d})^2)^3\right)^{\frac{1}{2}}.
\end{multline}
Trivially bounding $\lcm(\bm{d})\leq \prod_{j=1}^4d_j$, the inner sum may be bounded against
\begin{multline*}
\sum_{\substack{\bm{d}\in\Z^4\\ d_j\mid P_w(z)\\ d_j\leq \frac{D}{11^{9}}}}\prod_{j=1}^4d_j^{3+4\cdot 10^{-6}+\frac{1}{100}}\left(27\pi \cdot 2400 + 16(520)^3 \prod_{j=1}^4d_j^4\right)^{\frac{1}{2}}\\
\leq  \left(27\pi \cdot 2400 + 16(520)^3\right)^{\frac{1}{2}} \sum_{\substack{\bm{d}\in\Z^4\\ d_j\mid P_w(z)\\ d_j\leq \frac{D}{11^{9}}}}\prod_{j=1}^4d_j^{5+4\cdot 10^{-6}+\frac{1}{100}}\leq 47434 \prod_{j=1}^4 \sum_{\substack{d_j\mid P_w(z)\\ d_j\leq \frac{D}{11^{9}}}}d_j^{5+4\cdot 10^{-6}+\frac{1}{100}}.
\end{multline*}
We then trivially bound 
\[
\prod_{j=1}^4\sum_{\substack{d_j\mid P_w(z)\\ d_j\leq \frac{D}{11^{9}}}}d_j^{5+4\cdot 10^{-6}+\frac{1}{100}}\leq \prod_{j=1}^4\sum_{d_j\leq \frac{D}{11^{9}}}d_j^{5+4\cdot 10^{-6}+\frac{1}{100}}\leq \left(\frac{D}{11^{9}}\right)^{4\left(6+4\cdot 10^{-6}+\frac{1}{100}\right)}.
\]
Plugging back into \eqref{eqn:finalRbound} and simplifying yields the claim. \qedhere
\end{proof}

\subsection{The proof of Theorem \ref{thm:PrSuniversal}}
Plugging Lemma \ref{lem:Rsumbound} into Lemma \ref{lem:SAz} (plugging in $\beta=10$) yields 
\begin{equation}\label{eqn:SAez}
X \Sigma(D,z)-1.23\times 10^{-91}m^{\frac{17}{30}} D^{24.05}\leq S\left(\mathscr{A},z\right)\leq X\Sigma'(D,z) + 3.07\times 10^{-92}  m^{\frac{17}{30}}D^{24.05}.
\end{equation}

\begin{lemma}\label{lem:cuspidalbound}
Suppose that $8\nmid m$, $27\nmid m$, $25\nmid m$, and 
\begin{multline*}
\bm{a}\in \{(1,1,1,k):1\leq k\leq 7\}\cup\{(1,1,2,k):2\leq k\leq 8\}\cup \{(1,1,3,k):3\leq k\leq 6\}\\
\cup \{(1,2,2,k):2\leq k\leq 7\}\cup\{(1,2,3,k):3\leq k\leq 8\}\cup\{(1,2,4,k):4\leq k\leq 14\}\cup\{(1,2,5,k):4\leq k\leq 15\}.
\end{multline*}
Then we have 
\[
S\left(\mathscr{A}_{\bm{1}},z\right)\geq 0.063  m^{1-10^{-6}} \frac{e^{-5\gamma}}{\log(z)^5} \left(1-\frac{1}{\log^2(z)}\right)^5- 1.23\times 10^{-91}  m^{\frac{17}{30}}D^{24.05}.
\]
\end{lemma}
\begin{proof}
We plug in the bound from Lemma \ref{lem:BigsSigmaBound} for $\Sigma(D,z)$ and the bound from $X=X_{\bm{1}}$ from Lemma \ref{lem:Xlower} into \eqref{eqn:SAez}. We then use Lemma \ref{lem:RosserSchoenfeld} (expanding the alternating geometric series) to bound 
\begin{equation}\label{eqn:RosserSchoenfeld}
\prod_{p<z}\left(1-\frac{1}{p}\right)^{5}\geq \frac{e^{-5\gamma}}{\log(z)^5} \left(1-\frac{1}{\log^2(z)}\right)^5,
\end{equation}
giving the claim. 
\end{proof}

We are now ready to prove Theorem \ref{thm:PrSuniversal}.
\begin{proof}[Proof of Theorem \ref{thm:PrSuniversal}]
Note first that 
\begin{equation}\label{eqn:tosolve}
\sum_{j=1}^{\ell} a_jx_j^2=m
\end{equation}
 is solvable with $x_j\in P_{r,S}$, then by multiplying $x_j$ by $2$, $3$, or $5$ we obtain that
\[
\sum_{j=1}^{\ell} a_j(2x_j)^2=4m,\qquad\qquad \sum_{j=1}^{\ell} a_j(3x_j)^2=9m,\qquad\text{ and }\qquad \sum_{j=1}^{\ell} a_j(5x_j)^2=25m
\]
are solvable with $2x_j,3x_j,5x_j\in P_{r,S}$. Hence if \eqref{eqn:tosolve} is solvable with $x_j\in P_{r,S}$ for every $m\in\N$ with $8\nmid m$, $27\nmid m$, and $25\nmid m$, then $\sum_{j=1}^{\ell} a_j x_j^2$ is $P_{r,S}$-universal. 

Any $P_{r,S}$-universal sum must also be universal, so by Bhargava's escalation method the first $4$ parameters in $\bm{a}$ must be one of 
\begin{multline*}
 \{(1,1,1,k):1\leq k\leq 7\}\cup\{(1,1,2,k):2\leq k\leq 14\}\cup \{(1,1,3,k):3\leq k\leq 6\}\\
\cup \{(1,2,2,k):2\leq k\leq 7\}\cup\{(1,2,3,k):3\leq k\leq 10\}\\
\cup\{(1,2,4,k):4\leq k\leq 14\}\cup\{(1,2,5,k):5\leq k\leq 15\}.
\end{multline*}
We may therefore assume that $8\nmid m$, $27\nmid m$, and $25\nmid m$ and use Lemma \ref{lem:cuspidalbound} to obtain that the number of representations $\sum_{j=1}^{4}a_j x_j^2=m$ of $m$ with $p\mid x_j\implies p\in\{2,3,5\}\text{ or }p\geq z$ is 
\[
S\left(\mathscr{A}_{\bm{1}},z\right)\geq 0.063  m^{1-10^{-6}} \frac{e^{-5\gamma}}{\log(z)^5} \left(1-\frac{1}{\log^2(z)}\right)^5- 1.23 \times 10^{-91}  m^{\frac{17}{30}}D^{24.05}
\]
with $D\geq z^{25}$. We therefore choose $D=z^{25}$. Taking $z=\max(m^{\frac{1}{1388}},7)$, we have 
\begin{multline*}
S\left(\mathscr{A}_{\bm{1}},z\right)\geq 0.063  m^{1-10^{-6}} \frac{1388^5 e^{-5\gamma}}{\log(m)^5} \left(1-\frac{1}{\log^2(7)}\right)^5- 1.23\times 10^{-91}  m^{0.999844}\\
\geq 3.91\times 10^{12}  m^{1-10^{-6}} \log(m)^{-5}- 1.23\times 10^{-91}  m^{0.999844}.
\end{multline*}
We then use the bound 
\begin{equation}\label{eqn:logbound}
\log(m)\leq \frac{1}{r} m^r
\end{equation}
with $r=10^{-6}$ to obtain 
\[
S\left(\mathscr{A},z\right)\geq 3.91\times 10^{-18} m^{1-10^{-6}- 5\cdot 10^{-6}}  - 1.23\times 10^{-91}  m^{0.999844}.
\]
This is positive as long as 
\[
m^{1.5\cdot 10^{-4}} \geq 3.15\times 10^{-74},
\]
which holds trivially for all $m\in\N$.  Hence for every $m$ with $8\nmid m$, $27\nmid m$, and $25\nmid m$, we have a representation $m=\sum_{j=1}^{4} a_j x_j^2$ where $p\mid x_j$ implies that $p\in\{2,3,5\}$ or $p\geq m^{\frac{1}{1388}}$. Since $a_j\geq 1$, if the number of primes $p\notin \{2,3,5\}$ dividing $x_{j_0}$ for some $1\leq j_0\leq 4$ is $\geq L$, then 
\[
m= \sum_{j=1}^4 a_jx_j^2\geq x_{j_0}^2 \geq m^{\frac{L}{694}}.
\]
Therefore $L\leq 694$, as claimed.
\end{proof}

\end{document}